\documentclass[11pt]{amsart}
\usepackage{amsmath}
\usepackage{amssymb}
\usepackage{amsthm}
\usepackage[all]{xy}
\usepackage{graphicx}
\usepackage{stmaryrd}
\usepackage[margin=1in]{geometry}
\usepackage{color}
\usepackage{hyperref}
\usepackage{verbatim}
\usepackage{pb-diagram,pb-xy}

\newcommand{\bpsi}{\boldsymbol\psi}
\newcommand{\bPsi}{\boldsymbol\Psi}

\swapnumbers
\newtheorem*{mthm}{Main Theorem}
\newtheorem*{thmA}{Theorem A}
\newtheorem*{corB}{Corollary B}
\newtheorem*{corC}{Corollary C}
\newtheorem*{corD}{Corollary D}

\newtheorem*{ndefn}{Definition}
\newtheorem*{uclaim}{Claim}
\numberwithin{equation}{subsection}

\newcommand{\cf}[1]
{
\textcolor{blue}{#1} }

\newcommand{\R}{\ensuremath{\mathbb{R}}}
\newcommand{\N}{\ensuremath{\mathbb{N}}}

\newcommand{\Z}{\ensuremath{\mathbb{Z}}}
\newcommand{\Q}{\ensuremath{\mathbb{Q}}}

\newtheorem{theorem}[equation]{Theorem}
\newtheorem{corollary}[equation]{Corollary}
\newtheorem{proposition}[equation]{Proposition}

\newtheorem{lemma}[equation]{Lemma}

\theoremstyle{definition} 
\newtheorem{defn}[equation]{Definition}
\newtheorem{Construction}[equation]{Construction}
\newtheorem{remark}[equation]{Remark} 

\newtheorem{Observation}[equation]{Observation}

\newcommand{\MT}{{\mathsf M}{\mathsf T}}

\newcommand{\mb}[1]{\mathbf{#1}}

\newcommand{\Cob}{{\mathbf{Cob}}}

\DeclareMathOperator*{\colim}{colim\ }

\DeclareMathOperator*{\Int}{Int}

\DeclareMathOperator*{\Diff}{Diff}

\DeclareMathOperator*{\BDiff}{BDiff}

\DeclareMathOperator*{\Hom}{Hom}

\DeclareMathOperator*{\kernel}{Ker}

\DeclareMathOperator*{\rel}{rel}

\DeclareMathOperator*{\Id}{\text{Id}}

\DeclareMathOperator*{\Bun}{\mathrm{Bun}}

\DeclareMathOperator*{\Ob}{\text{Ob}}

\DeclareMathOperator*{\Image}{Im}
\DeclareMathOperator*{\cpt}{lf}

\author{Fabian Hebestreit}
\address{Mathematisches Institut, Universt\"at Bonn, Endenicher Allee 60, 53115 Bonn, Germany}
\email{f.hebestreit@math.uni-bonn.de}

\author{Nathan Perlmutter}
\address{Stanford University Department of Mathematics, Building 380, Stanford, California,  94305, USA}
\email{nperlmut@stanford.edu}

\title[Cobordism Categories and Moduli Spaces of Odd Dimensional Manifolds]{Cobordism Categories and Moduli Spaces of Odd Dimensional Manifolds}

\begin{document}
\maketitle

\begin{abstract}
We prove that the stable moduli space of $(n-1)$-connected, $n$-parallelizable, $(2n+1)$-dimensional manifolds is homology equivalent to an infinite loopspace for $n \geq 4, n \neq 7$. The main novel ingredient is a version of the cobordism category incorporating surgery data in the form of Lagrangian subspaces. 
\end{abstract}

\setcounter{tocdepth}{1}
\tableofcontents
\vspace*{-5mm}
\section{Introduction}

\subsection{History and Motivation}
The study of diffeomorphisms of smooth manifolds has been a focus of differential topology from its inception. After initial geometric techniques were developed, the method of choice for attacking such automorphism groups was a combination of surgery theory and Waldhausen's A-theory, with the former providing input on block diffeomorphisms and the latter on the difference between block and honest diffeomorphisms, see \cite{WW3} for a modern formulation. With Tillmann's work \cite{T 97} on cobordism categories of surfaces and Madsen and Weiss' proof of the Mumford conjecture \cite{MW 07}, however, a new method emerged, whose application to high dimensional manifolds was pioneered by Galatius and Randal-Williams in \cite{GRW 09, GRW 14, GRW 14c} culminating in \cite{GRW 16}. Their work focuses on manifolds of even dimension and in the simplest case proceeds roughly as follows: Write $d$ for $2n$ or $2n+1$ and set $\theta^n_{d} \colon BO(d)\langle n \rangle \rightarrow BO(d)$ for the classifying map of the universal $d$-dimensional vector bundle on an $n$-connected space. One can then form the monoid (under boundary-connected sum)
\[\mathcal M_{d} = \coprod_{W} {\BDiff}^\partial(W),\]
where the index runs through all $(n-1)$-connected $n$-parallelizable nullcobordisms of $S^{d-1}$; $n$-parallelizable means that the classifying map for the tangent bundle $W \rightarrow BO(d)$ admits a lift along $\theta_{d}^{n} \colon BO(d)\langle n \rangle \rightarrow BO(d)$, or equivalently that $W$ admits a framing on some $n$-skeleton. 
In three largely independent steps, Galatius and Randal-Williams now show that 
\begin{enumerate}
	\item[i)] The \emph{scanning map} of Madsen and Tillmann
	\[\mathcal M_{2n} \longrightarrow \Omega^{\infty}\MT\theta_{2n}^{n}\]
	induces a weak equivalence
	\[\Omega B\mathcal M_{2n} \longrightarrow \Omega^{\infty}\MT\theta_{2n}^{n},\]
	where the right hand side is a certain Thom spectrum, whose homology is readily computable by the standard toolkit of algebraic topology.
	\item[ii)] Denoting by $W^{2n}_{g,1}$ the $g$-fold connected sum of $S^n\times S^n$ with a disc removed the inclusions
	\[{\BDiff}^\partial(W_{g,1}) \longrightarrow \mathcal M_{2n}\]
	induce an isomorphism
	\[\colim_{g \rightarrow \infty} H_*({\BDiff}^\partial(W_{g,1})) \longrightarrow H_*(\Omega_0 B\mathcal M_{2n}),\]
	with structure maps extending diffeomorphisms by the identity outside $W_{g,1}^{2n} \subset W_{g+1,1}^{2n}$.
	\item[iii)] The structure maps 
	\[{\BDiff}^\partial(W_{g,1}) \longrightarrow {\BDiff}^\partial(W_{g+1,1})\]
	in the colimit above induce isomorphism on homology in degrees below $\frac{g-3} 2$.
\end{enumerate}
The results of Galatius and Randal-Williams hold in far greater generality, but a similarly complete description of the homology of diffeomorphism groups in the case of any odd dimensional manifolds is as of yet conjectural at best, even in the highly connected case. Homological stability (i.e. the analogue of the third part) for the diffeomorphism groups of certain odd dimensional manifolds (among them connected sums of $S^n \times S^{n+1}$) was recently established by the second author in \cite{P 15} and \cite{P 16} and part ii) has a well-known (if more complicated) analogue, as we shall explain below. The first statement, however, does not yet have a meaningful replacement. Indeed, due to the vanishing of certain characteristic classes (by work of Ebert \cite{E 09}) it was previously known that $\Omega_0 B \mathcal M_{2n+1}$ is not equivalent to $\Omega_0^\infty \MT\theta_{2n+1}^n$ or any close variant. Thus it was an open question, whether the homology of stabilized diffeomorphism groups in odd dimensions is governed by a spectrum at all. The main result of the present paper answers this affirmatively:

\begin{mthm}
For $n \geq 4$ an integer except $7$ the group completion $\Omega B \mathcal M_{2n+1}$ of $\mathcal M_{2n+1}$ carries an infinite loop space structure.
\end{mthm}

Let us remark that the monoid $\mathcal M_d$ is well-known to carry an $E_d$-structure but is not expected to carry even an $E_{d+1}$-structure before group completion (we shall, however, not make use of these structures in the remainder and therefore will say nothing further about this point). We will exhibit an explicit infinite loop space structure on $\Omega B\mathcal M_{2n+1}$ by modifying the definition of cobordism categories in odd dimensions, see Section \ref{statement} below. To explain this construction, which is probably more important than the result itself, and to place our results in context, we begin by presenting in Section \ref{section: even dimensional outline} a more detailed sketch of the techniques from \cite{GRW 14}, \cite{GRW 09} and \cite{GMTW 08} used to establish the first two statements of the list above, explain in Section \ref{subsection: intro odd dim case} their failure in odd dimensions and our solution in Section \ref{statement}.

Note that our result makes the situation for odd dimensions analogous to that in the surface case in the 1990's, where Tillman in \cite{T 97} showed the homotopy type of the group completion of $\mathcal M_2$ and thus the homology of $\colim {\BDiff^\partial}(\Sigma_{g,1})$ as $g$ goes to infinity to be that of an infinite loop space, whereas the homotopy type of the underlying spectrum was only identified much later by Madsen and Weiss in \cite{MW 07}.

\subsection{The even dimensional case} \label{section: even dimensional outline}
For any map $\theta: B \longrightarrow BO(d)$ with $d = 2n$ or $2n+1$, recall the cobordism category $\Cob_{\theta}$ defined in \cite{GMTW 08}: Objects of $\Cob_{\theta}$ are given by $(d-1)$-dimensional, closed submanifolds $M \subset \R^{\infty}$, equipped with a bundle map $\ell_{M}: TM\oplus\epsilon^{1} \longrightarrow \theta^{*}\gamma^{d}$, i.e. a $\theta$-structure. 
A morphism between objects $M$ and $N$ is given by a $d$-dimensional embedded cobordism $W \subset [0, t]\times\R^{\infty}$, equipped with a bundle map $\ell_{W}: TW \longrightarrow \theta^{*}\gamma^{d}$ that restricts to $\ell_{M}$ and $\ell_{N}$ over the boundary. 
The category $\Cob_{\theta}$ is topologized so that for each $M, N \in \Ob\Cob_{\theta}$ there is a weak homotopy equivalence
$$
\Cob_{\theta}(M, N) \; \simeq \; \coprod_{W}\textstyle{\BDiff_{\theta}}(W, M\sqcup N), 
$$
where the union ranges over all diffeomorphism classes of compact manifolds $W$ equipped with a specified identification $\partial W \cong M\sqcup N$. 
The main theorem from \cite{GMTW 08} yields a weak equivalence 
\[\Omega B\Cob_{\theta} \simeq \Omega^{\infty}\MT\theta,\]
where $\MT\theta$ is the Thom spectrum associated to the virtual vector bundle $-\theta^{*}\gamma^{d}$ over $B$.

Now for any $(n-1)$-connected nullcobordism $W$ of $S^{d-1}$ the space of extensions of a fixed $\theta^n_{d}$-structure on $S^{d-1}$ to $W$ is either empty or contractible, hence in the latter case
\[{\BDiff}^\partial(W) \simeq {\BDiff}_{\theta_{d}^{n}}(W \cup_{\partial} D^{d},D^{d}).\]
Removing two disks from the fixed disk on the right hand side therefore produces a multiplicative map
\[\mathcal{M}_{d} \longrightarrow \Cob_{\theta_d^{n}}(S^{d-1}, S^{d-1}),\]
where the right-hand side is the endomorphism monoid on the standard sphere $S^{d-1} \subset \R^{\infty}$, equipped with its essentially unique $\theta^{n}_{d}$-structure compatible with the orientation. 

The scanning map of i) above can then be factored as 
\[\xymatrix{ 
\Omega B \mathcal M_{d} \ar[r] \ar[d] & \Omega^\infty\MT\theta_{d}^{n} \\
                     \Omega B \Cob_{\theta_{d}^n}(S^{d-1},S^{d-1}) \ar[r] &\Omega B\Cob_{\theta_{d}^{n}}. \ar[u]_{\simeq}
}\]
The majority of the technical work in \cite{GRW 14} is then devoted to establishing that the composite
\[ B\mathcal{M}_{2n} \longrightarrow  B \Cob_{\theta_{d}^n}(S^{d-1},S^{d-1}) \longrightarrow B\Cob_{\theta_{2n}^{n}}\]
is a weak equivalence onto a path component (the term in the middle needs modification to make this true for the individual maps as well). This result is achieved via a sequence of parametrized surgery arguments, making first the morphisms and then the objects ever higher connected. 

Part ii) is obtained by an application of the group completion theorem of McDuff and Segal \cite{MS 75}. It shows that
\[H_*(\Omega B\mathcal M_{d}) = H_*(\mathcal M_{d})[\mathcal W_{d}^{-1}],\]
the localization taken with respect to ${\mathcal W_{d}} = \pi_0(\mathcal M_d) \subseteq H_0(\mathcal M_d)$. By a result of Kreck, the monoid $\mathcal W_{2n}$ is generated under saturation by the single element $s = [W^{2n}_{1,1}]$, which implies 
\[H_*(\mathcal M_{2n})_{\mathcal W_{2n}} = \colim\left(H_*(\mathcal M_{2n}) \stackrel{\cdot s} \longrightarrow H_*(\mathcal M_{2n})  \stackrel{\cdot s} \longrightarrow \dots \right).\]
Denoting by $\mb{B}_\infty^{2n}$ the path component of $D^{2n}$ in the homotopy colimit over 
\[\mathcal M_{2n} \stackrel{W_{1,1}^{2n}}\longrightarrow \mathcal M_{2n} \stackrel{W_{1,1}^{2n}}\longrightarrow \dots\]
we obtain a homology isomorphism \[\mb B_\infty^{2n} \rightarrow \Omega_0B\mathcal M_{2n},\]
which gives assertion ii). We will refer to the space $\mb{B}_{\infty}^{2n}$ as the \textit{stable moduli space of highly-connected manifolds} of dimension $2n$. 

As it will play no role in the paper, let us refrain from expounding the proof of part iii).

\subsection{The odd-dimensional case} \label{subsection: intro odd dim case}
It is tempting to try to carry out a similar program to study the stable moduli spaces of odd-dimensional manifolds. While it is not true that $\mathcal W_{2n+1}$ is obtained from a single element under saturation, the localization $H_*(\mathcal M_{2n+1})_{\mathcal W_{2n+1}}$ can still be computed as a colimit: Fixing a system of generators $\{s_i\}_{i \in \N}$ for $\mathcal W_{2n+1}$ (it is countable), we have
\[H_*(\mathcal M_{2n+1})_{\mathcal W_{2n+1}} = \colim\left(H_*(\mathcal M_{2n+1}) \stackrel{\cdot s_1}\longrightarrow H_*(\mathcal M_{2n+1}) \stackrel{\cdot s_1^2s_2}\longrightarrow H_*(\mathcal M_{2n+1})\stackrel{\cdot s_1^3s_2^2s_3}\longrightarrow \dots \right).\]
Again form the stable moduli space of manifolds $\mb{B}_{\infty}^{2n+1}$ as the path component of the disk $D^{2n+1}$ in the corresponding homotopy colimit over $\mathcal M_{2n+1}$. In replacement of statement ii) one obtains a homology isomorphism
\[\mb B^{2n+1}_\infty \longrightarrow \Omega_0 B\mathcal M_{2n+1}.\]
By homotopy commutativity of $\mathcal{M}_{2n+1}$, the homotopy type of the limiting space $\mb{B}^{2n+1}_{\infty}$ does not depend on the choice of generating set.

However, the arguments from \cite{GRW 14} cannot be used to prove that the map
$$B\mathcal{M}_{2n+1} \longrightarrow B\Cob_{\theta_{2n+1}^n}$$
is the inclusion of a path component, as the techniques employed by Galatius and Randal-Williams
meet a surgery obstruction when applied in the middle dimension of even dimensional manifolds (in the step making the objects of $\Cob_{\theta_{2n+1}^n}$ into homotopy spheres). 
This failure is not a mere technicality since as mentioned above it had previously been established by Ebert in \cite{E 09} that for an odd dimensional manifold the scanning map is never injective in rational cohomology, not even in a range. 
Thus, in order to complete the picture in odd dimensions, replacements for the cobordism category $\Cob_{\theta}$ and the spectrum $\MT\theta$ are needed.

\subsection{Statement of results}\label{statement}
Let $n \neq 1,3,7$; this restriction implies the existence of the quadratic form $\mu$ used below. The following category is the main object of study: 

\begin{ndefn}
The topological category $\Cob^{\mathcal{L}}_{2n+1}$ has as its objects pairs $(M, L)$ that satisfy the following conditions:
\begin{enumerate} 
\item[(i)] $M$ is an object of $\Cob_{\theta_{2n+1}^n}$, i.e.\ $M \subset \R^{\infty}$ is a $2n$-dimensional closed submanifold equipped with a $\theta_{2n+1}^n$-structure.
\item[(ii)] $L \leq H_{n}(M)$ is a \textit{Lagrangian subspace} with respect to the intersection and self-intersection form $(H_{n}(M), \lambda, \mu)$. 
By \textit{Lagrangian} we mean that $L^{\perp} = L$ with respect to $\lambda$ and $\mu_{|L} = 0$.
\end{enumerate}
The morphism space $\Cob^{\mathcal L}_{2n+1}((M,L_{M}),(N,L_{N}))$ is the following subspace of $\Cob_{\theta_{2n+1}^n}(M,N)$. 
A cobordism $W \subseteq [0,t]\times\R^\infty$ from $M$ to $N$ is a morphism in $\Cob^{\mathcal L}_{2n+1}((M,L_{M}),(N,L_{N}))$ if:
\begin{enumerate} 
\item[(a)] The pair $(W, N)$ is $(n-1)$-connected;
\item[(b)] $\iota^{\text{in}}(L_{M}) \; = \; \iota^{\text{out}}(L_{N}),$
where
$\iota^{\text{in}}: H_{n}(M) \longrightarrow H_{n}(W)$ and $\iota^{\text{out}}: H_{n}(N) \longrightarrow H_{n}(W)$ are the maps induced by the boundary inclusions.
\end{enumerate} 
\end{ndefn}

Since $H_{n}(S^{2n}) = 0$, we obtain a factorization
\[\mathcal{M}_{2n+1} \longrightarrow \Cob^{\mathcal L}_{2n+1}(S^{2n}, S^{2n}) \longrightarrow  \Cob_{\theta^n_{2n+1}}(S^{2n}, S^{2n})\]
and thus a map $\Omega B \mathcal M_{2n+1} \longrightarrow \Omega B \Cob^{\mathcal L}_{2n+1}$.
The following theorem is our main technical result:

\begin{thmA}
Let $n \geq 4$ be an integer except $7$.  
Then the map 
$$
\Omega B\mathcal{M}_{2n+1} \longrightarrow \Omega B\Cob^{\mathcal{L}}_{2n+1}
$$
just described is a weak homotopy equivalence.
\end{thmA}

The operation of disjoint union almost makes $\Cob^{\mathcal{L}}_{2n+1}$ into a symmetric monoidal category (almost, due to embedded cobordism having a length). More precisely, we endow $B\Cob^{\mathcal{L}}_{\theta}$ with the structure of a special $\Gamma$-space. Applying the results of Segal \cite{S 74}, we therefore obtain a (connective) spectrum $\MT\mathcal L_{2n+1}$, together with a map $\MT\mathcal L_{2n+1} \rightarrow \MT\theta_{2n+1}^n$, such that:

\begin{corB}
Let $n \geq 4$ be an integer except $7$.  
Then the scanning map factors through an equivalence $\Omega B \mathcal M_{2n+1} \rightarrow \Omega^{\infty}\MT\mathcal L_{2n+1}$.
\end{corB}

As explained in the previous sections one can now deduce information about stable diffeomorphism groups by applying the group completion theorem.

\begin{corC}
Let $n \geq 4$ be an integer except $7$. 
Then the stable moduli space of manifolds $\mb{B}^{2n+1}_{\infty}$ has the homology type of the infinite loopspace $\Omega^{\infty}_0\MT\mathcal L_{2n+1}$. 
\end{corC}

Finally, we can say a little bit about the cohomology of $\Omega^\infty\MT\mathcal L_{2n+1}$: Denoting the scanning map 
$\BDiff(W, D^{2n+1})  \longrightarrow  \Omega^{\infty}_{0}\MT\theta^{n}_{2n+1}$
for $W \in \mathcal{W}_{2n+1}$ by $\mathcal P_W$ and
$F: \Omega^{\infty}\MT\mathcal L_{2n+1} \longrightarrow \Omega^{\infty}\MT\theta^{n}_{2n+1}$
the infinite loop map induced by applying $\Omega$ to the composite
$$\xymatrix{B\Cob^{\mathcal{L}}_{2n+1} \ar[r] & B\Cob_{\theta^{n}_{2n+1}} \ar[r]^-{\simeq} & \Omega^{\infty-1}\MT\theta^{n}_{2n+1},
}$$
we find: 

\begin{corD}
Let $n \geq 4$ be an integer except $7$. 
Then the kernel of the homomorphism
$$F^{*}: H^{*}(\Omega^{\infty}_{0}\MT\theta^{n}_{2n+1}; \Q) \; \longrightarrow \; H^{*}(\Omega^{\infty}_0\MT\mathcal L_{2n+1}; \Q)$$
is equal to the common kernel of the collection of the maps 
$$\mathcal{P}^{*}_{W}: H^{*}(\Omega^{\infty}_{0}\MT\theta^{n}_{2n+1}; \Q) \longrightarrow H^{*}(\BDiff(W, D^{2n+1}; \Q))$$
for all $W \in \mathcal{W}_{2n+1}$.
\end{corD} 

It follows from the main result of \cite{E 09} that the common kernel of the maps
$$\mathcal{P}^{*}_{W}: H^{*}(\Omega^{\infty}_{0}\MT\theta^{n}; \Q) \longrightarrow H^{*}(\BDiff(W, D^{2n+1}); \Q)$$ 
contains the \emph{tautological classes} associated to Hirzebruch's $\mathcal L$-polynomials, and it is a simple calculation that most of these do not vanish in the source.

\subsection{Remarks}
While our construction of the category $\Cob^\mathcal L_{2n+1}$ works equally well for arbitrary $n$-connected $\theta$ it is not currently clear how to define it in general. The reason is the non-existence of a self-intersection form at the homology level. A possible solution is suggested by the work of Ranicki, i.e. to work directly with the chains on a manifold rather than its homology. At the chain level all relevant refinements of the intersection pairing should be present without connectivity assumptions. Carrying out such a generalization is beyond the scope of the present paper, however. The necessary chain level constructions are part of a large joint project of the first author with Calm\`es, Dotto, Harpaz, Land, Moi, Nardin, Nikolaus and Steimle \cite{CDHHLMNNS}, which in particular defines cobordism categories of Poincar\'e chain complexes $\Cob_d(R)$ for any commutative ring $R$, and relates their homotopy types to Grothendieck-Witt spectra. The category $\Cob_d^{\mathcal L}$ is then closely related to the over-category of $0$ with respect to an evident functor $\Cob_{\theta^n_{2n+1}} \rightarrow \Cob_{2n+1}(\mathbb Z)$ taking a manifold to its chains. This perspective also suggests that (at least for highly connected manifolds) Ebert's vanishing result can be strengthened to include classes originating in the Grothendieck-Witt spectrum $\mathrm{GW}(\mathbb Z)$; the tautological classes of Hirzebruch's $\mathcal L$-polynomials originate further in the $L$-theory $L(\mathbb Z)$ of the integers (and one of the main results of \cite{CDHHLMNNS} is to establish a precise relation between these spectra by way of a forgetful map $\mathrm{GW}(\mathbb Z) \rightarrow \mathrm{L}(\mathbb Z)$). These consequences will be explored in future joint work with Steimle.

The stabilisation with respect to many manifolds instead of just the $W^{2n}_{g,1}$ is also present in \cite{GRW 14} for the treatment of more general tangential structures and manifolds and is therefore not a new phenomenon in odd dimensions. For even dimensions, however, the need for it was later removed as the main result of \cite{GRW 16} by showing that the homology of $\mb B_\infty^{\theta}$ no longer is affected by further stabilisation. An analogue of this result is work in progress and could replace stabilization with respect to all elements of $\mathcal{W}_{2n+1}$, by stabilization with only $S^n \times S^{n+1}$, for which the second author proved homological stability in \cite{P 15}. 
Such results improving on homological stability are largely independent of the present work though. 

Finally, the most pressing question to be addressed in future work is that of the homotopy type of $\MT\mathcal L_{2n+1}$, in particular whether its cohomology contains entirely new classes not coming from $\MT\theta^{n}_{2n+1}$. One would expect such an analysis to be simpler than a direct analysis of the monoid $\mathcal M_{2n+1}$ just as $B\Cob_\theta$ can be identified with $\Omega^{\infty-1}\MT\theta$ using Pontryagin-Thom theory. With the results of \cite{CDHHLMNNS} there are obvious guesses to be made, but we cannot currently offer a precise conjecture based on more than speculation.

\subsection{Organization of the paper}
Section \ref{section: Cobordism categories an spaces of manifolds} is devoted to recollections about spaces of submanifolds and the extension of their definition to include homological data. In Section \ref{section: cobordism categories} we repeat the definition of the usual cobordism category, explain those steps in the program of Galatius and Randal-Williams that also work in odd dimension. In Section \ref{sec 4} we define the category $\Cob^{\mathcal{L}}_{\theta}$ and show how to derive Theorem A and its corollaries from a number of technical results whose proofs make up the rest of the paper: Section \ref{subsection: infinite loopspaces} constructs the $\Gamma$-space structure on $B\Cob^{\mathcal L}_{2n+1}$, Sections \ref{section: A substitute for the cobordism category} and \ref{section: preliminaries on surgery} contain preliminary technicalities, which enable us to apply the parametrized surgery techniques of \cite{GRW 14} in the final three sections.

The paper is not meant to be read in isolation: To avoid lengthy repetitions of established material we shall make frequent use of the terminology and results of the two papers \cite{GRW 09} and \cite{GRW 14}. For example in several places we will only indicate which parts of a definition needs to be changed. The reader is therefore advised to have copies of both papers close by.

\subsection{Acknowledgements}
We heartily thank Oscar Randal-Williams for his encouragement when this project was in its infancy and 
S\o ren Galatius for his ongoing guidance. 
We are also indebted to Boris Botvinnik and Johannes Ebert for their valuable advice and particularly to Diarmuid Crowley and Tibor Macko for pointing out the existence and properties of various quadratic refinements.
FH would like to further thank Markus Land, Stephan Stolz and in particular Daniel Kasprowski for very useful discussions about various subtleties in surgery theory. 
NP thanks Anibal Medina for many enlightening conversations around cobordism categories and surgery theory. 
FH was supported by a scholarship of the German Academic Exchange Service during a year at the University of Notre Dame. 
NP was supported by an NSF Post-Doctoral Fellowship, DMS-1502657, at Stanford University and by the ERC-grant 'KL2MG-interactions' of Wolfgang L\"uck during a visit to the University of Bonn.


\section{Preliminaries on Spaces of Manifolds} \label{section: Cobordism categories an spaces of manifolds}

\subsection{Spaces of manifolds}
We begin by briefly reviewing some basic constructions from \cite{GRW 09} and \cite{GRW 14}. Recall that a \textit{tangential structure} is a map $\theta: B \longrightarrow BO(d)$.  A $\theta$-structure on an $m$-dimensional manifold $M$ (with $m \leq d$) is a bundle map $TM\oplus\epsilon^{d-m} \longrightarrow \theta^{*}\gamma^{d}$ (i.e. a fiberwise linear isomorphism). 

Fix a tangential structure $\theta: B \longrightarrow BO(d)$. Recall, for $U \subseteq \R^{n}$, the space $\bPsi_{\theta, l}(U)$ from \cite[Section 2]{GRW 09}, consisting of pairs $(M, \ell)$ where $M \subseteq U$ is an $l$-dimensional submanifold without boundary and closed as a subspace of $U$, while $\ell$ is a $\theta$-structure on $M$.  These are topologized so that the assignment $U \mapsto \bPsi_{\theta, l}(U)$ defines a sheaf on $\R^{n}$, valued in topological spaces. We will slightly extend the construction of that topology below in Section \ref{subsection: spaces of manifolds with homology data}. As in \cite{GRW 09} we will need to consider particular subspaces of $\bPsi_{\theta, l}(\R^{n})$ consisting of submanifolds $M \subset \R^{n}$ that are open in a fixed number of directions. We repeat \cite[Definition 3.5]{GRW 09}:

\begin{defn} \label{defn: manifolds open in k directions}
For $k \leq n$, $\bpsi_{\theta,l}(n, k) \subset \bPsi_{\theta,l}(\R^{n})$ is
the subspace consisting of those $\theta$-manifolds $(M, \ell)$ such
that $M \subset \R^{k}\times(-1, 1)^{n-k}$.  
The space
$\bpsi_{\theta,l}(\infty, k)$ is defined to be the colimit of
the $\bpsi_{\theta}(n, k)$ taken as $n\to\infty$.  
\end{defn}

For $l = d$, we shall drop the index $l$. Let $x_{1}:  \R\times\R^{\infty} \longrightarrow \R,$ denote the projection onto the first factor. We will often consider $\bpsi_{\theta,l}(\infty, k)$ as the space submanifolds $W \subseteq  \R\times\R^{\infty}$ and for any subset $K \subseteq \R$, we write 
$$W|_{K} = W\cap x_{1}^{-1}(K)$$
If $\ell$ is a $\theta$-structure on $W$ and $W_{|K}$ a submanifold of $W$, then we  write $\ell|_{K}$ for the restriction of $\ell$ to $TW|_{K}$.

\subsection{Homological preliminaries} \label{section: homological preliminaries}
We will need to work with homology groups of elements of $\bPsi_{\theta}(\R^{\infty})$, which are in general non-compact manifolds, and it turns out that for our purposes the locally finite/Borel-Moore homology $H^{lf}_*$ as defined in \cite{Spa} is a convenient set-up. We shall only really have to consider $H^{lf}_k(M,A)$ for an $m$-manifold $M$ with a closed codimension $0$ submanifold $A$ both of which have cylindrical ends (see Definition \ref{defn: cylindrical ends} below). In this case
\[H^{lf}_k(M,A) = \lim_{K \subseteq M}H_{k}(M, A \cup (M\setminus K)),\]
see e.g. by \cite[Theorem 10.1]{Spa} for the absolute case and then use the long exact sequences. We will effectively treat the right hand side as a definition. Recall then that locally finite homology is covariantly functorial in proper maps (by restricting to the final system $f^{-1}(K)$ in the limit defining the source) and contravariantly functorial in open embeddings. In the special case above with $A = 0$ the contravariant functoriality for $V \subseteq M$ is given by
\[\xymatrix{
H^{\cpt}_{k}(M) \ar[r]^-{\cong} & \displaystyle{\lim_{K \subseteq M}}H_{k}(M, M \setminus K) \ar[r] &  \displaystyle{\lim_{K \subseteq V}}H_{k}(M, M \setminus K) & \ar[l]_\cong \displaystyle{\lim_{K \subseteq V}}H_{k}(V, V \setminus K)  \ar[r]^-{\cong} & H^{\cpt}_{k}(V),
}\]
with the middle maps the projection to the indicated components of the limit and excision, respectively. Furthermore, stemming from Poincar\'e duality for non-compact manifolds there are Gysin homomorphisms. We shall only need the following version: Let $j: Z \hookrightarrow M$ be the inclusion of a neat, compact, oriented submanifold of dimension $n$ into $M$, which we also assume oriented and further $U \subseteq M$ a closed tubular neighborhood of $Z$. 
The homomorphism 
\[j_{!}: H^{\cpt}_{k}(M) \; \longrightarrow \; H_{k+n-m}(Z,\partial Z)\]
is defined to be the composition 
$$\xymatrix{
H^{\cpt}_{k}(M) \ar[r] & H_k(M, M \setminus \Int U) & \ar[l]_-\cong H_{k}(U, \partial U) &\ar[l]_\cong H^{m-k}(U) & \ar[l]_\cong H^{m-k}(Z) \ar[r]^{\cong \ \ \ } & H_{k+n-m}(Z,\partial Z),
}$$
where the first map is the canonical projection (of the inverse limit onto one of its factors), the second is excision and the third and fifth arrows are given by Lefschetz duality. Since any two tubular neighborhoods of $Z$ are isotopic (see \cite[Theorem 5.3]{Hi 76}) it follows that the definition of the map $j_{!}$ is independent of the choice of tubular neighborhood $U$. We will call both these types of maps \emph{restrictions} and denote them by 
$$
x \longmapsto x|_{Z}.
$$

\begin{remark}
This leads to very little ambiguity: If the open subset $V \subseteq X$ is given as the interior of some compact codimension $0$ submanifold $Z \subseteq X$ we claim that $H^{\cpt}_k(V)$ and $H_k(Z,\partial Z)$ are canonically isomorphic in a fashion making 
$$\xymatrix{ & H_k^{\cpt}(X) \ar[ld]_{\cdot|_V} \ar[rd]^{\cdot|_Z}& \\
            H_k^{\cpt}(V) \ar@{<->}[rr]_\cong & & H_k(Z,\partial Z)}$$
commutative.
To this end choose an open collar $C$ of $\partial Z$ in $Z$. As $V-C$ is a compact subset of $V$ we obtain a map $H_k^{\cpt}(V) \rightarrow H_k(V,C-\partial Z)$, which we shall momentarily see is an isomorphism (since $V$ has $C-\partial Z$ as its cylindrical ends as defined below). Clearly, $H_k(V,C-\partial Z)$ and $H_k(Z,\partial Z)$ are canonically isomorphic (for example via their inclusions into $H_k(Z,C)$). It is now readily checked that this identification is independent of the chosen collar and the diagram above indeed commutes.
\end{remark}

As mentioned we will mainly work with a particularly simple class of non-compact manifolds: 

\begin{defn}
A manifold $M$ is said to have \textit{cylindrical ends} if there exists some compact codimension-$0$ submanifold $B \subset X$ (possibly with boundary), such that the complement $M\setminus\Int(B)$ is homeomorphic to the cylinder $\partial B \cup A \times[0, \infty)$, relative to $\partial B$, for some codimension $0$ submanifold $A$ of $\partial B$.
\end{defn}

In this case we find an isomorphism $H^{\cpt}_{k}(M) \cong H_{k}(B, A)$, since every compact subset $K \subset M$ is contained in a submanifold $B$ as above, so
\[H^{\cpt}_k(M) \cong \displaystyle{\lim_{B \subseteq M}}H_{k}(M, M \setminus B) \cong \displaystyle{\lim_{B \subseteq M}}H_{k}(B, A)\]
by finality and the latter system is evidently constant.

\subsection{Spaces of manifolds equipped with homological data} \label{subsection: spaces of manifolds with homology data}
We will need to consider spaces of manifolds equipped with a choice of subspace of its homology group. These spaces (defined below) will enable us to topologize the cobordism category $\Cob^{\mathcal{L}}_{\theta}$ (discussed in the introduction) and the semi-simplicial spaces introduced in Section \ref{section: A substitute for the cobordism category}. 

\begin{defn}
Fix a tangential structure $\theta: B \longrightarrow BO(d)$. For an open subset $U \subset \R^{m}$, 
let $\bPsi^{\Delta,k}_{\theta,l}(U)$ denote the set of triples $(M, \ell, V)$ with $(M, \ell) \in \bPsi_{\theta,l}(U)$ and $V \leq H^{\cpt}_{k}(M)$ a subgroup. 
\end{defn}

We topologize the set $\bPsi^{\Delta}_{\theta}$ in complete analogy with \cite[Section 2.1]{GRW 09}, where a topology on the spaces $\bPsi_{\theta}$ is described in three steps. Let us briefly indicate the necessary changes:\\

\begin{Construction} \label{construction: construction of topology}
\textit{Step 1:} Define the \textit{compactly supported topology} on $\bPsi^{\Delta}_{\theta}(U)$ for $U \subseteq \mathbb R^n$ in the same fashion as \cite[Step 1, page 1248]{GRW 09}, using instead of the map $c_M$ the partially defined map
\[\Gamma_c(\nu M) \longrightarrow \bPsi^{\Delta}_{\theta}(U), \quad s \longmapsto \big((id_M + p \circ s)(M), \ell \circ D(id_M + p \circ s)^{-1}, (id_M + p \circ s)_*(V)\big)\]
for some $(M,\ell,V) \in \bPsi^{\Delta}_{\theta}(U)$, where $p \colon \nu M \rightarrow \mathbb R^n$ is the projection in the fibre direction. This makes $\bPsi^{\Delta}_{\theta}(U)$ into a covering of $\bPsi_{\theta}(U)$ with fiber over $(M,\ell)$ the set of subgroups of $H_n(M)$ although we shall not need this.

\textit{Step 2:} 
To construct the \textit{$K$-topology}, for some compact $K \subset U$, proceed as in \cite[Step 2, page 1249]{GRW 09}, identifying two elements of $\bPsi^{\Delta}_{\theta}(U)$ if 
$$(M\cap A, \; \ell|_{M\cap A}, \; V|_{M\cap A}) \; = \; (M'\cap A, \; \ell'|_{M'\cap A}, \; V'|_{M'\cap A}),$$
for some neighbourhood $A$ of $K$.

\textit{Step 3:}
Finally, give $\bPsi^{\Delta}_{\theta}(U)$ the terminal topology making all identities into the various $K$-topologies continuous, no changes required.
\end{Construction}

With $\bPsi^{\Delta}_{\theta}(U)$ topologized as above one may proceed as in \cite[Section 2.2]{GRW 09} to obtain direct analogues to the basic properties of $\bPsi_{\theta}(U)$ proven in that section. For example we shall need the analogue of \cite[Theorem 2.7]{GRW 09}, its proof applies verbatim:

\begin{proposition}\label{conti}
Let $U' \subseteq U$ be an open subset.
The restriction map 
$$\bPsi^{\Delta}_{\theta}(U) \longrightarrow \bPsi^{\Delta}_{\theta}(U'), \quad (M, \ell, V) \; \mapsto \; (M\cap U', \; \ell|_{M\cap U'}, \; V|_{M\cap U'})$$
is continuous. 
\end{proposition}

\begin{defn} \label{defn: cylindrical ends}
For each $k \leq m$ we define $\bpsi^{\Delta}_{\theta}(m, k) \subset \bPsi^{\Delta}_{\theta}(\R^{m})$ to be the subspace consisting of those $(M, \ell, V)$ such that $(M, \ell) \in \bpsi_{\theta}(m, k)$.
We define $\bPsi^{\Delta}_{\theta}(\mathbb R^\infty)$ and $\bpsi_{\theta}(\infty, k)$ to be the colimits of the above spaces, taken as $m \to \infty$. 
\end{defn}

Finally, let us give a criterion for lifting the continuity of maps into $\bPsi^{\Delta}_{\theta}(U)$ from that of their projection to $\bPsi_{\theta}(U)$. For the maps that we will be concerned with later (in Sections \ref{section: surgery on objects below the middle dimension} and \ref{section: surgery on objects in the middle dimension}) it is satisfied by the discussion in \cite[Sections 4 and 5]{GRW 14}.

\begin{defn} \label{defn: locally generated by vector fields}
A map $\phi \colon X \rightarrow \bPsi_{\theta}(U)$, written $(W_{t}, \ell_{t}) \in \bPsi_{\theta}(U)$, $t \in X$, is said to be \textit{locally generated by vector fields} if for every $t_{0} \in X$ and compact subset $A \subset U$, there exists a neighbourhood $V \subset X$ of $t_0$ and a map $s \colon V \rightarrow \Gamma_c(\nu W_{t_0})$ such that 
$$W_{t}\cap A  \; = \; (id_M + p \circ s_{t})(W_{t_{0}})\cap A \quad \text{and} \quad
 \ell_{t}|_{A\cap W_{t}} \; = \; (\ell_{t_{0}}\circ D(id_M + p \circ s_{t})^{-1})|_{A\cap W_{t}}$$
for all $t \in V$.
\end{defn}

\begin{Construction} \label{construction: one parameter family of subspaces}
Let $\phi \colon X \rightarrow \bPsi_{\theta}(U)$, written $t \longrightarrow (W_{t}, \ell_{t})$, be a continuous map and $Q \subseteq U$ be an open subset such that the family $(W_{t}, \ell_{t})$ is constant when restricted to $U\setminus Q$. Let $W'$ denote the (constant) complement $W_t\setminus(W_{t}\cap Q)$ and 
$$\beta_{t}: H^{\cpt}_{*}(W') \longrightarrow H^{\cpt}_{*}(W_{t})$$ 
be the homomorphism induced by inclusion $W' \hookrightarrow W_{t}$. For a subspace $V \leq H^{\cpt}_{*}(W')$ let 
\[V_{t} := \beta_{t}(V)  \leq  H^{\cpt}_{*}(W_{t}).\]
This gives a lift 
\[\phi_Q \colon X \rightarrow \bPsi^\Delta_{\theta}(U), t \longmapsto (W_t,\ell_t,V_t).\]
\end{Construction}

\begin{proposition} \label{proposition: continuity of the family}
If $\phi \colon X \rightarrow \bPsi_{\theta}(U)$ is locally generated by vector fields, then $\phi_Q \colon X \rightarrow \bPsi^\Delta_{\theta}(U)$ is continuous.
\end{proposition}

\begin{proof} 
By definition of the topology it suffices to prove that the families
\[(W_{t}, \ell_{t}, V_{t})_{K} \in \bPsi^{\Delta}_{\theta}(U),\]
are continuous for the various $K$-topologies. 
Since $\phi$ is locally generated by vector fields, there exists a neighbourhood $V$ of $t$, such that $(W_t \cap A,\ell_t|_{A})$ is given by a family of vector fields $s \colon V \rightarrow \Gamma_c(\nu W_{t_0})$ on $V$. But then each $s_t$ has to vanish on $W' \cap A$, whence
\[(s_t)_*(V_{t_0})|_{W_t \cap A} = V_{t}|_{W_t \cap A}.\]
This means that the manifestly continuous family $(s_t(W_{t_0}) \cap A, \ell|_{s_t(W_{t_0}) \cap A}, (s_t)_*(V_{t_0})|_{s_t(W_{t_0}) \cap A})$ identifies to the same map as $\phi|_V$ in the quotient defining the $K$-topology.
\end{proof}


\section{Cobordism Categories of Highly Connected Odd Dimensional Manifolds}  \label{section: cobordism categories}
In this section we collect the relevant parts of \cite{GRW 14} that still hold true in the odd dimensional setting and briefly explain the failure of the key statement. In particular, we give the relationship between the stabilized diffeomorphism group, the monoid $\mathcal M_{2n+1}$ from the introduction and the cobordism category, which are entirely analogous to the even dimensional, non-highly connected situation.

\subsection{Some subcategories of $\Cob_{\theta}$} \label{subsection: subcategories}
We start out by repeating \cite[Definition 3.7]{GRW 09}:

\begin{defn} \label{defn: the cobordism category}
We let the non-unital topological category $\Cob_{\theta}$ have object space $\bpsi_{\theta,d-1}(\infty, 0)$. The morphism space is the following subspace of $\mathbb R \times \bpsi_{\theta}(1 + \infty,1)$: A pair $(t,(W, \ell))$ is a morphism if there exists an $\varepsilon > 0$ with
$$W|_{(-\infty, \varepsilon)} \;  = \; (-\infty, \varepsilon)\times W|_{0} \quad \text{and} \quad W|_{(t - \varepsilon, \infty)} \;  = \; (t-\varepsilon, \infty)\times W|_{t}$$ 
as $\theta$-manifolds, where $(-\infty, \varepsilon)\times W|_{0}$ and $(t-\varepsilon, \infty)\times W|_{t}$ are equipped with the product $\theta$-structures induced from $\ell|_{0}$ and $\ell|_{t}$ on $W|_{0}$ and $W|_{t}$ using. The source of such a morphism is $W|_0$ and the target $W|_t$, equipped with their respective restrictions of the $\theta$-structure $\ell$ on $W$.
\end{defn}

We need to establish some notation. Let $n \geq 0$ be an integer and fix a map $\theta \colon B \rightarrow BO(2n+1)\langle n \rangle$ giving rise to a tangential structure, which we shall also call $\theta$ (we follow Galatius and Randal-Williams in using $\langle n \rangle$ to indicate the $n$-connected cover of space). Also, fix once and for all a $2n$-dimensional disk 
\begin{equation} \label{equation: fixed disk}
D \subset (-\tfrac{1}{2}, 0]\times(-1, 1)^{\infty-1},
\end{equation}
which near $\{0\}\times\R^{\infty-1}$ agrees with $(-1, 0]\times\partial D$ and a $\theta$-structure $\ell_{D}: TD\oplus\epsilon^{1} \longrightarrow \theta^{*}\gamma^{2n+1}$ on it. Let $\ell_{\R\times D}$ denote the $\theta$-structure on $\R\times D$ induced by $\ell_{D}$. Recall, finally, the notion of \textit{weak once-stability} of a tangential structure $\theta: B \longrightarrow BO(2n+1)$ from \cite[Definition 5.4]{GRW 14}, which implies that the $\theta$-structure on a cobordism $(M, \ell_{M})$ to $(N, \ell_{N})$ can be changed to give a cobordism from $(N, \ell_{N})$ to $(M, \ell_{M})$, something that is not true in general (\cite[Section 5.2]{GRW 14}). 

\begin{defn} \label{defn: disk category no L}\label{defn: monoid on sphere}
Define a sequence of subcategories of $\Cob_{\theta}$ as follows:
\begin{enumerate} 
\item[(i)]
$\Cob^{\mb{m}}_{\theta} \subset \Cob_{\theta}$ has the same space of objects, and the morphisms from $(M, \ell_{M})$ to $(N, \ell_{N})$ are given by those $(t, W, \ell)$ for which the pair $(W|_{[0, t]}, W|_{t})$ is $(n-1)$-connected. 
\item[(ii)] $\Cob^{D}_{\theta} \subset \Cob^{\mb{m}}_{\theta}$ has as its objects those $(M, \ell)$ such that 
$$M\cap\left[(-1, 0]\times(-1, 1)^{\infty-1}\right] \; = \; D,$$ 
and that the restriction of $\ell$ to $D$ agrees with $\ell_{D}$.
Similarly, it has as its morphisms those $(W, \ell)$ such that 
$W\cap\left[\R\times(-1, 0]\times(-1, 1)^{\infty-1}\right] \; = \; \R\times D,$
and the restriction of $\ell$ to $\R\times D$ agrees with $\ell_{\R\times D}$.
\item[(iii)] Let $l \in \Z_{\geq -1}$.
The topological subcategory $\Cob^{l}_{\theta} \subset \Cob^{D}_{\theta}$ is the full subcategory on those objects $(M, \ell)$ such that $M$ is $l$-connected.
\item[(iv)] Assume $\theta$ weakly once-stable. Then $\Cob^{\emptyset}_{\theta} \subset \Cob_{\theta}$ is the full subcategory on those $\theta$-manifolds $(M, \ell)$ that are $\theta$-nullcobordant, i.e. admit a morphism to (or from) the empty set.
\item[(v)] Define $\Cob^{l, \emptyset}_{\theta}$ to be the intersection $\Cob^{\emptyset}_{\theta}\cap\Cob^{l}_{\theta}$.
\end{enumerate}
\end{defn}

Just as in the even dimensional case $\Cob^{n}_{\theta}$ relates via a group completion argument to the diffeomorphisms of $n-1$-connected, $n$-parallelizable manifolds. We establish this in the next section. In the last we will discuss the relation between $\Cob^n_\theta$ and  $\Cob_\theta$ and the homotopy type of the latter.

\subsection{Reduction to a monoid and group completion} \label{section: reduction to a monoid}\label{subsection: group completion}
We proceed to establish the relation between $\Cob^{n}_{\theta}$ and stabilised diffeomorphism groups. Assume that the tangential structure $\theta: B \longrightarrow BO(2n+1)$ is weakly once-stable and $n \geq 3$. It follows directly from the discussion of weak once-stability (and reversibility) that the subspace $B\Cob^{\emptyset}_{\theta} \subset B\Cob_{\theta}$ is a single path component of $B\Cob_{\theta}$. For $B\Cob^{l, \emptyset}_{\theta}$ the analoguous property relies on a connectivity assumption on $\theta$. For $l \leq n-1$ it then follows from Theorem \ref{theorem: connected object equivalence} and for $l=n$ we need the following strengthening:

\begin{proposition} \label{proposition: path connectivity of object space}
If $B$ is $n$-connected, the object space $\Ob\Cob^{n, \emptyset}_{\theta}$ is a path component of $\Ob\Cob^{n}_\theta$ and consequently $B\Cob^{n, \emptyset}_{\theta}$ is a path component of $B\Cob^{n}_{\theta}$.
\end{proposition}

\begin{proof}
Note first, that $\Ob\Cob^{n, \emptyset}_{\theta}$ is clearly a union of path components of $\Ob\Cob^{n}_\theta$, so it will suffice to show that $\Ob\Cob^{n, \emptyset}_{\theta}$ is path connected. To this end observe that, since $BO(2n+1)\langle n \rangle$ is $n$-connected, it follows that for any object $(M, \ell) \in \Ob\Cob^{n, \emptyset}_{\theta}$, the manifold $M$ is diffeomorphic to the standard sphere $S^{2n}$: Indeed, $M$ is $n$-connected and thus is a homotopy sphere. Let $(W, \ell_{W})$ be a $\theta$-null-bordism of $(M, \ell)$. By the connectivity assumption the $\theta$-structure provides $W$ with a parallelization over an $n$-skeleton and then by \cite{W 62a}, we may perform a sequence of surgeries on the interior of $W$ so that the resulting manifold $\widetilde{W}$ is a contractible manifold. Since $n \geq 3$, it follows from the h-cobodism theorem that $\widetilde{W}$ is diffeomorphic to $D^{2n+1}$, and thus $M \cong S^{2n}$. Thus there is a weak homotopy equivalence 
$$\Ob\Cob^{n, \emptyset}_{\theta} \simeq  \textstyle{\Bun^{\emptyset}}(TS^{2n}\oplus\epsilon^{1}, \theta^{*}\gamma^{2n+1}; \ell_{D}) \sslash \Diff(S^{2n}, D^{2n}),$$ 
where $\Bun^{\emptyset}(TS^{2n}\oplus\epsilon^{1}, \theta^{*}\gamma^{2n+1}; \ell_{D})$ is the space of $\theta$-structures $\ell$ on $S^{2n}$ that agree with the structure $\ell_{D}$ when restricted to a fixed disk  $D^{2n} \subset S^{2n}$, such that $(S^{2n}, \ell)$ is $\theta$-cobordant to the empty set. We will show that $\textstyle{\Bun^{\emptyset}}(TS^{2n}\oplus\epsilon^{1}, \theta^{*}\gamma^{2n+1}; \ell_{D})$ is path connected. 

Since the space of $\theta$-structures on $D^{2n+1}$ (fixed on a boundary hemisphere) is contractible, it will suffice to show that every element of $\Bun^{\emptyset}(TS^{2n}\oplus\epsilon^{1}, \theta^{*}\gamma^{2n+1}; \ell_{D})$ is the restriction to the boundary of some $\theta$-structure on $D^{2n+1}$. To see this we need only observe that the surgeries making $W$ into a disk can be chosen compatible with $\theta$. This is ensured by the dicussion in \cite[Section 4.1]{GRW 14}.

The addendum stating that $B\Cob^{n, \emptyset}_{\theta}$ is a path component of $B\Cob^{n}_{\theta}$ now follows from the fact that $\Cob^{n, \emptyset}_{\theta} \subseteq \Cob^{n}_{\theta}$ is a full subcategory.
\end{proof}

To proceed, fix once and for all an object $(S, \ell_{S}) \in \Ob\Cob^{n, \emptyset}_{\theta}$. We define 
$$\mathcal{M}_{\theta} \subset \Cob^{n, \emptyset}_{\theta}$$
to be the endomorphism monoid on the object $(S, \ell_{S})$. Since the object space $\Ob\Cob^{n, \emptyset}_{\theta}$ is path-connected and the combined source-target map is well known to be a fibration (for example it follows from the work of Lima \cite{Lim}, the extension to $\theta$-structures is explicitely handled in \cite[Lemma 4.1]{RaSt}) the homotopy type of $\mathcal{M}_{\theta}$ as a topological monoid is independent of the choice of object $(S, \ell_{S})$.

\begin{proposition} \label{corollary: monoid reduction}
For an $n$-connected tangential structure $\theta$, the inclusion $B\mathcal{M}_{\theta} \hookrightarrow B\Cob^{n, \emptyset}_{\theta}$ is a weak homotopy equivalence. 
\end{proposition}

The proof is entirely analogous to \cite[Section 7.1]{GRW 14} using the path-connectivity of $\Ob\Cob^{n, \emptyset}_{\theta}$.
By the contractibilty of embedding spaces into infinite euclidean space we find
\[\mathcal{M}_{\theta} \simeq \coprod_{W}\textstyle{\BDiff_{\theta}}(W, D^{2n+1}),
\]
with union ranging over diffeomorphism classes of $(n-1)$-connected, $(2n+1)$-dimensional, closed $\theta$-manifolds $W$, equipped with an embedding $D^{2n+1} \hookrightarrow W$ compatible with the $\theta$-structure. Finally, we have:

\begin{proposition} \label{lemma: homotopy commutativity and homotopy type}
The monoid $\mathcal{M}_{\theta}$ is homotopy commutative. 
\end{proposition}

The proof is just as in \cite[Proposition 4.27]{GRW 09} and so we also omit it. For a particular choice of tangential structure we obtain the monoid $\mathcal{M}_{2n+1}$ defined in the introduction: Let $\theta^{n}\colon BO(2n+1)\langle n \rangle \longrightarrow BO(2n+1)$ denote the projection. Since for any $(n-1)$-connected, $(2n+1)$-dimensional closed manifold $W$ that admits a $\theta^{n}$-structure, the space of $\theta^{n}$-structures on $W$ is weakly contractible (relative to the one chosen on the embedded disk, see \cite[Lemma 7.16]{GRW 14}) it follows that there is a weak homotopy equivalence \[\mathcal{M}_{\theta^{n}} \simeq \mathcal{M}_{2n+1}.\] We will use these two monoids interchangeably.

\begin{defn}
A manifold admitting a $\theta^n$-structure is called $n$-parallelizable. Let $\mathcal{W}_{2n+1}$ denote the set of diffeomorphism classes of oriented, $(n-1)$-connected, $(2n+1)$-dimensional, closed, manifolds, that are $n$-parallelizable.
\end{defn}

Clearly $\mathcal W_{2n+1} \cong \pi_{0}\mathcal{M}_{2n+1}$, in particular $\mathcal W_{2n+1}$ is a monoid under connected sum. Recall that a monoid $M$ is said to be \textit{finitely saturated} if its group completion can be constructed by inverting just finitely many elements of $M$ (or equivalently a single one).

\begin{proposition} \label{proposition: countable generation of monoid}
The monoid $\mathcal{W}_{2n+1}$ is countable, but not finitely generated or even finitely saturated.
\end{proposition}

\begin{proof}
Recall that two oriented manifolds $M$ and $M'$ are said to be \textit{almost diffeomorphic} if $M$ is diffeomorphic to $M'\#\Sigma$ where $\Sigma$ is an oriented homotopy sphere. Since the set of homotopy spheres in a given dimension is finite, it follows that there are only finitely many diffeomorphism types in a given almost-diffeomorphism class. In \cite{W 67}, Wall shows that the almost diffeomorphism class of any $(n-1)$-connected, $(2n+1)$-dimensional manifold $M$, is determined by a finite collection of algebraic invariants, each of which it turns out can take only countably many values. In the case that $n$ is even and $W \in \mathcal{W}_{2n+1}$, these (almost) diffeomorphism invariants are given by the linking form $b: \tau H_{n}(W)\otimes\tau H_{n}(W) \longrightarrow \Q/\Z$, and cohomology classes $\widehat{\phi} \in H^{n+1}(W; \Z/2)$ and $\widehat{\beta} \in H^{n+1}(W; \pi_{n}(SO))$ (see \cite[Theorem 7]{W 67}). The linking form is a nonsingular, $(-1)^{n+1}$-symmetric, bilinear pairing, and according to the classification in \cite{W 64}, there are countably many such objects up to isomorphism. It follows that the set of almost-diffeomorphism classes (and hence the diffeomorphism classes) of elements of $\mathcal{W}_{2n+1}$ is countably infinite. 

That $\mathcal{W}_{2n+1}$ cannot be finitely saturated, follows from the analogous fact for the monoid (under direct sum) of isomorphism classes of finite abelian groups, to which $\mathcal{W}_{2n+1}$ surjects (via taking $n$-th homology), since by \cite{W 67} all non-degenerate pairings on finite abelian groups are realised by linking forms (and e.g. by \cite{W 64} there is at least one such on every finite group); the latter monoid clearly is not finitely saturated (as finitely many elements can only ever account for torsion at finitely many primes). 
Since finite saturation is inherited by quotients the claim follows.
\end{proof}

Since $\mathcal{M}_{2n+1}$ is homotopy commutative we may apply the group completion theorem of McDuff and Segal from \cite{MS 75} (see \cite{Ni 18} for a concise modern treatment).  The main result of \cite{MS 75} implies that the natural map $\mathcal{M}_{2n+1} \longrightarrow \Omega B\mathcal{M}_{2n+1}$, induces an isomorphism 
\[H_{*}(\mathcal{M}_{2n+1})\left[\pi_{0}(\mathcal{M}_{2n+1})^{-1}\right] \; \stackrel{\cong}  \longrightarrow \; H_{*}(\Omega B\mathcal{M}_{2n+1}).\]
Using the language of \cite{RW 13}, this isomorphism may also be expressed as a certain map 
$$(\mathcal{M}_{2n+1})_\infty \longrightarrow \Omega B \mathcal{M}_{2n+1}$$
being a homology equivalence (even acyclic), where $(\mathcal{M}_{2n+1})_\infty$ is the colimit of the direct system
\[\xymatrix{\mathcal{M}_{2n+1} \ar[r]^-{\cdot W_{1}} & \mathcal{M}_{2n+1} \ar[r]^-{\cdot W_{1}\cdot W_{2}} & \mathcal{M}_{2n+1} \ar[rr]^-{\cdot W_{1}\cdot W_{2}\cdot W_{3}} && \mathcal{M}_{2n+1} \ar[rr]^-{\cdot W_{1}\cdot W_{2}\cdot W_{3}\cdot W_{4}} && \cdots}\]
where the $W_{i}$ give a generating system of $\pi_0(M)$ under saturation. Restricting to the path component of $\mathcal{M}_{2n+1}$ corresponding to the sphere $S^{2n+1}$ produces the direct system
$$\xymatrix{\BDiff(W_{1}, D^{2n+1}) \ar[r] & \BDiff(W_{1}^{\#2}\# W_{2}, D^{2n+1}) \ar[r] & \BDiff(W_{1}^{\#3}\# W^{\# 2}_{2}\#W_{3}, D^{2n+1}) \ar[r] & \cdots}$$
We denote the colimit of this direct system by $\mb{B}_{\infty}$, dropping the superscripted dimension from the introduction. Let $\Omega_{0}B\mathcal{M}_{2n+1} \subseteq \Omega B\mathcal{M}_{2n+1}$ denote the path-component that contains the constant loop. We obtain the desired conclusion:

\begin{proposition}\label{grp-cplt}
The above construction produces a homology equivalence, in fact an acyclic map,
$$\mb{B}_{\infty} \longrightarrow \Omega_{0}B\mathcal M_{2n+1}$$
and therefore the composite
$$\mb{B}_{\infty} \longrightarrow \Omega_0^\infty\Cob^{n}_{\theta^n}$$
is one as well.
\end{proposition}

\subsection{The homotopy type of the cobordism category}
Let us finally explain the failure of the method from \cite{GRW 14} in odd dimensions briefly. First of all we recall the main result of \cite{GMTW 08} in the language of \cite{GRW 09}: If we denote by $\MT\theta$ the Thom spectrum associated to the $-d$-dimensional virtual vector bundle $-\theta^{*}\gamma^{d}$ over $B$, Galatius and Randal-Williams construct (zig-zags of) weak equivalences 
\[B\Cob_{\theta} \simeq \bpsi_{\theta}(\infty, 1) \simeq \Omega^{\infty-1}\MT\theta,\]
that model the scanning map mentioned from the introduction. Since we will not explicitely make use of further properties of these equivalence let us refrain from spelling them out. Note that this result is completely insensitive to the parity of the manifold dimension. The same is true for the following two results.

\begin{theorem} \label{theorem: disk category to m}
The inclusions $B\Cob^{D}_{\theta} \hookrightarrow B\Cob^{\mb{m}}_{\theta} \hookrightarrow B\Cob_{\theta}$ are weak homotopy equivalences. 
\end{theorem}

This is immediate from \cite[Proposition 2.15 \& Theorem 3.1]{GRW 14} and the next statement follows from \cite[Theorem 4.1]{GRW 14}.

\begin{theorem} \label{theorem: connected object equivalence}
Let $l \leq n-1$ and suppose that $\theta: B \longrightarrow BO(2n+1)$ has $B$ $l$-connected and $\pi_{l+1}(B)$ finitely generated. 
Then the inclusion $B\Cob^{l}_{\theta} \hookrightarrow B\Cob^{l-1}_{\theta}$ is a weak homotopy equivalence. 
\end{theorem}

The alert reader will notice that our assumption are weaker than those in the cited theorem of Galatius and Randal-Williams: Their result requires $B$ to be $(l+1)$-connected, since (in their notation) the map $L \rightarrow B$ is assumed an $(l+1)$-equivalence. The stronger assumption is used at only one place in the proof, namely to show that the relative homotopy group $\pi_{l+1}(\ell)$ is finitely generated (where $\ell \colon M \rightarrow B$ denotes a $\theta$-structure). Following Wall \cite[Section 1A]{W 70}, this is known to follow from the weaker assumption that $\ell$ be an $l$-equivalence and $B$ of type $F_{l+1}$, i.e. admitting an $l+1$-equivalence from \emph{some} finite complex (see the proof of \cite[Lemma 3.81]{CML} for an exposition where this becomes evident). This in turn is guaranteed by our assumptions.

Theorems \ref{theorem: disk category to m} and \ref{theorem: connected object equivalence} together imply the weak homotopy equivalences
$$B\Cob^{n-1}_{\theta} \simeq B\Cob_{\theta} \simeq \Omega^{\infty-1}\MT\theta$$
for $(n-1)$-connected $B$ with $\pi_n(B)$ finitely generated. In contrast to the even dimensional case (see \cite[Theorem 5.3]{GRW 14}), however, Theorem \ref{theorem: connected object equivalence} cannot even be extended to yield a weak homotopy equivalence between $B\Cob^{n,\emptyset}_{\theta}$ and $B\Cob^{n-1,\emptyset}_{\theta}$ for $n$-connected $B$! Indeed, the inclusion $\Omega B\Cob_{\theta^{n}}^{n} \hookrightarrow \Omega B\Cob_{\theta^{n}}$ cannot be a weak homotopy equivalence: Consider the homotopy commutative diagram 
\[\xymatrix{\Omega_{0}B\mathcal{M}_{2n+1} \ar[r] & \Omega^{\infty}_{0}\MT\theta^{n} \\
\mb{B}_{\infty} \ar[u]^{\simeq_{H_{*}}} & \coprod_{i \in \mathbb N} \BDiff(W_{1}^{\#i}\# \dots \#W_{i}, D^{2n+1}) \ar[u]_{\coprod\alpha_i} \ar[l],}\]
from the previous section, where the bottom-horizontal map is induced by the inclusions of the terms into their colimit, the top arrow is induced by the composite 
$$B\mathcal{M}_{2n+1} \stackrel{\simeq} \longrightarrow B\Cob^{n}_{\theta^{n}} \longrightarrow B\Cob_{\theta^{n}} \stackrel{\simeq} \longrightarrow \Omega^{\infty-1}\MT\theta^{n},$$
and the right-vertical map is the \textit{scanning map}, whose effect in rational cohomology is well-known to give the \emph{tautological} or \emph{Morita-Miller-Mumford  classes} of manifold bundles. Since homology preserves colimits, the bottom map is certainly surjective in homology, and thus $B\Cob^{n}_{\theta^{n}} \longrightarrow B\Cob_{\theta^{n}}$ being a weak equivalence onto a path component would imply surjectivity (in homology) of the right-vertical map. 
However, the main theorem of \cite{E 09} implies that this map has a non-trivial kernel in rational cohomology (consisting at least of the tautological classes of Hirzebruch's $\mathcal L$-polynomials, most of which do not vanish in the cohomology of $\Omega^{\infty-1}\MT\theta^{n}$), and thus cannot be surjective in homology.

We therefore see that the failure of Theorem \ref{theorem: connected object equivalence} in the case that $l = n$ is fundamental, and not merely a technical shortcoming of the methods of \cite{GRW 14}.


\section{Cobordism Categories of Manifolds Equipped with Lagrangians}\label{sec 4}
In this section we set out to describe our modification of the replacement for the cobordism category in detail (Definition \ref{defn: lagrangian cobordism category}) and state the main technical theorems we will prove in the paper (Theorems \ref{theorem: inclusion of disk} - \ref{theorem: infinite loopspace structure}). The results from the introduction are immediate consequences and we spell this out at the end of the section. First, however, we need to cover some preliminaries regarding bilinear and quadratic forms and their Lagrangian subspaces.

\subsection{Preliminaries on quadratic forms and Lagrangian subspaces} \label{subsection: quadratic forms}
Let $\varepsilon = \pm 1$. An $\varepsilon$-symmetric \textit{bilinear form} is a pair $(P, \lambda)$ where $P$ is a finitely generated $\Z$-module and $\lambda: P\otimes P \longrightarrow \Z$ is a bilinear map with the property that $\lambda(x, y) = \varepsilon\cdot\lambda(y, x)$ for all $x, y \in P$. An $\varepsilon$-symmetric bilinear form $(P, \lambda)$ is said to be \textit{non-singular} if the map 
$${P} \longrightarrow \textstyle{\Hom_{\Z}}({P}, \Z), \quad x  \longmapsto  \lambda(x, -)$$
becomes an isomorphism after modding out torsion. If ${V} \leq {P}$ is a submodule, we let ${V}^{\perp}$ denote the orthogonal complement of ${V}$ in ${P}$, i.e. 
$${V}^{\perp}  =  \{ x \in {P}  | \lambda(x, v) = 0 \quad \text{for all $v \in {V}$} \}.$$
 An $\varepsilon$ \textit{quadratic form} is a triple $({P}, \lambda, \mu)$, such that $(P, \lambda)$ is an $\varepsilon$-symmetric form and $\mu P \rightarrow \mathbb Z/(1-\varepsilon)$ is a quadratic refinement of $\lambda$ in the sense that
$$\mu(k x) = k^2 \mu(x) \quad \text{and} \quad \mu(p+q) = \mu(p)+\mu(q) + [\lambda(p,q)]$$
hold for all $k \in \mathbb Z$ and $x,y \in P$.

\begin{defn} \label{defn: lagrangian subspace}
Let $\varepsilon = \pm 1$, and let $(P, \lambda, \mu)$ be an $\varepsilon$-quadratic form.
A submodule $L \leq P$ will be called a \textit{Lagrangian} if $L = {L}^{\perp}$ and $\mu_{| L} = 0$.
\end{defn}

\begin{remark} \label{remark: refinement on even an symmetric form}
Let us note immediately, that a symmetric form (i.e. $\varepsilon = 1$) admits a quadratic refinement, only if it is even in the sense that $\lambda(p,p)$ is even for every $p \in P$. If that is the case, then there exists a unique refinement, namely $\mu(p) = \lambda(p,p)/2$. 
In particular the second condition in the definition of Lagrangian is implied by the first in this case and one can wholly disregard the quadratic refinement. 

In the case of an anti-symmetric form the situation is quite the opposite: Any torsionfree such form admits a subspace $L$ with $L^\perp = L$, and thus admits a quadratic refinement (since it can then be split apart into standard hyperbolics all of which do admit a quadratic refinement), but such a quadratic refinement is neither unique nor forced to vanish on $L$. 
\end{remark}

Let $M$ be a $2n$-dimensional compact oriented manifold. The main example of a bilinear form that we will consider is the intersection pairing 
$$\lambda: H_{n}(M)\otimes H_{n}(M) \longrightarrow \Z, \quad (x, y) \mapsto \langle D(x), j(y) \rangle,$$
where $D: H_{n}(M) \stackrel{\cong} \longrightarrow H^{n}(M, \partial M)$ is the Leftschetz duality isomorphism, $j: H_{n}(M) \longrightarrow H_{n}(M, \partial M)$ is the map induced by inclusion, and $\langle \cf\cdot, \cf \cdot \rangle: H^{n}(M, \partial M)\otimes H_{n}(M, \partial M) \longrightarrow \Z$ is the pairing between homology and cohomology. It follows from basic properties of the Lefschetz-duality isomorphism that $\lambda$ is $(-1)^{n}$-symmetric.
In the case that $M$ is a closed manifold it follows that the form $(H_{n}(M), \lambda)$ is non-singular. Throughout the paper we will refer to this bilinear form $(H_{n}(M), \lambda)$ as the \textit{intersection form} associated to $M$ and we will now endow it with a quadratic refinement when $M$ comes equipped with a highly connected tangential structure. With $\dim(M) = 2n$ the construction of the quadratic refinement breaks down in two separate cases depending on the parity of the integer $n$. Let again $\theta \colon B \rightarrow BO(2n+1)\langle n \rangle$ be a tangential structure.

\begin{Construction}
Suppose $n$ even and consider a $2n$-dimensional $\theta$-manifold $M$. It follows immediately that the $n$th Wu class $v_n(TM)$ vanishes. For even $n$, the element $v_n(TM)$ is characteristic for the modulo-$2$ intersection pairing (i.e.\ $\lambda(x,x) \equiv \lambda(\rho_2x,v_n)$ modulo $2$ for all $x \in H_n(M)$) by the Wu formula, and thus $v_n(TM) = 0$ forces this pairing to be even. From Remark \ref{remark: refinement on even an symmetric form}, we automatically obtain a quadratic refinement of the intersection form for such manifolds. 
\end{Construction}

In the case of odd $n$ one has to work harder to obtain a quadratic refinement for the intersection form of a $2n$-dimensional manifold. In general it was shown by Browder in his work on the Arf-Kervaire invariant, that a Wu-orientation, i.e. a lift of the classifying map $M \rightarrow BO$ of the normal bundle to the fibre of the map $v_{n+1} \colon BO \rightarrow K(\mathbb Z/2,n+1)$, determines such a refinement. Moreover, it follows from a calculation of Stong \cite{S 63}, that all $n$-parallelized $2n$-manifolds are \emph{canonically} Wu-oriented, unless $n$ is a Hopf dimension. We start out with the latter claim:

\begin{lemma}
We have $0 = v_{n+1} \in H^{n+1}(BO\langle n \rangle, \mathbb Z/2)$, whenever $n \neq 0,1,3,7$, and consequently $-\gamma$ admits a canonical Wu-orientation on $BO(2n+1)\langle n \rangle$.
\end{lemma}

\begin{proof}
For $n=5$ (or more generally when $n+1$ is not a power of $2$) this is immediate, since then $v_6 \in H^6(BO, \mathbb Z/2)$ (just like all non-two-power degree elements) is decomposable over the Steenrod-algebra. For general $n \geq 9$ Stong's calculations imply that $v_{n+1} \in H^{n+1}(BO\langle n-1 \rangle,\mathbb Z/2)$ lies in the image of the Postnikov section $BO\langle n-1 \rangle \rightarrow K(\pi_{n}BO,n)$, which in turn implies that $0 = v_{n+1} \in H^{n+1}(BO \langle n \rangle,\mathbb Z/2)$. 

For the second claim, note that in $BO\langle n \rangle$ we necessarily have $v_{n+1} = v_{n+1}(-\gamma)$, since the inversion on $BO\langle n\rangle$ is an automorphism of $H^{n+1}(BO\langle n \rangle)$ and this group is either 0 or $\mathbb Z/2$. This shows the existence of a Wu-orientation. The uniquness follows from obstruction theory: Lifts are parametrised by $H^n(BO(2n+1)\langle n \rangle,\mathbb Z/2) = 0$ once the obstruction $v_{n+1}$ vanishes.
\end{proof}

We shall now follow Brown \cite{B 72}, who gave a simple construction of Browder's quadratic refinement, see \cite{JR 78} for another published account. In fact, for $n$-parallelized manifolds Brown's construction may be simplified substantially:

\begin{Construction}
Let $n$ be odd with either $n = 5$ or $n \geq 9$ and $\theta\colon B \rightarrow BO(2n+1)\langle n \rangle$. Let $M$ be a $2n$-dimensional $\theta$-manifold. By the discussion above, the stable normal bundle of $M$ has a canonical Wu-orientation. Let us denote by $K(A,n)$ and $H(A,n)$ the Eilenberg-Mac Lane space and spectrum, respectively, whose sole non-trivial homotopy group is in degree $n$ with value $A$. Given a class $x \in H_n(M,\mathbb Z/2)$ we can then represent its Poincar\'e dual by a map $(M,\partial M) \rightarrow (K(\mathbb Z/2,n),pt)$. This map determines an element $\mu(x)$ in the relative bordism group $\Omega^{\langle n \rangle}_{2n}(K(\mathbb Z/2,n),pt)$ of $n$-parallelized manifolds (which is represented by the spectrum usually (mis)named $MO\langle n \rangle$). With our assumptions on the integer $n$, it turns out that this bordism group is isomorphic to $\Z/2$, see below.

By identifying $\Omega^{\langle n \rangle}_{2n}(K(\mathbb Z/2,n),pt) \cong \Z/2$, the assignment $x \mapsto \mu(x)$ yields the desired quadratic refinement of the intersection form (see \cite[Corollary 1.11]{B 72}).

The argument that $\Omega^{\langle n \rangle}_{2n}(K(\mathbb Z/2,n),pt) \cong \Z/2$ 
proceeds as follows: First observe that the map $\Sigma^\infty K(\mathbb Z/2,n) \rightarrow H(\mathbb Z/2,n)$ admits a lift into the homotopy fibre $F$ of 
$$Sq^{n+1}: H(\mathbb Z/2,n) \rightarrow H(\mathbb Z/2,2n+1).$$

Any such lift turns out to be a $(2n+1)$-equivalence by a direct calculation of the cohomology groups involved. In particular, $\Omega^{\langle n \rangle}_{2n}(K(\mathbb Z/2,n),pt) \cong \Omega^{\langle n \rangle}_{2n}(F)$. Smashing the fibre sequence defining $F$ with $MO\langle n \rangle$ gives an exact sequence
$$H_{n+1}(MO\langle n \rangle, \mathbb Z/2) \stackrel{Sq^{n+1}_*}{\longrightarrow} H_{0}(MO\langle n\rangle, \mathbb Z/2) \longrightarrow \Omega^{\langle n \rangle}_{2n}(F) \longrightarrow H_n(MO\langle n \rangle, \mathbb Z/2).$$ 
Clearly the fourth term is $0$, the second one $\mathbb Z/2$ and the first map may be identified with $$\chi(Sq^{n+1})^*: H^{n+1}(MO\langle n \rangle, \mathbb Z/2)^* \longrightarrow H^{0}(MO\langle n\rangle, \mathbb Z/2)^*$$
Since $\chi(Sq^{n+1})(u) = v_{n+1}u$ by definition of the Wu-class ($u \in H^0(MO\langle n \rangle, \mathbb Z/2)$ the Thom class), the map vanishes by the previous lemma and we obtain the desired isomorphism. 
\end{Construction}

The two most important facts about this refinement for us are that
\begin{enumerate}
\item [i)] it is preserved under codimension zero embeddings and 
\item[ii)] on a closed $(n-1)$-connected, $2n$-dimensional manifold $M$ it precisely obstructs representability of degree $n$ homology classes by embedded spheres with trivial normal bundle.
\end{enumerate}

Upon investing that every class in $H_n(M)$ can be represented by an embedded sphere unique up to regular homotopy by Haefliger's embedding theorem \cite{Ha 61} and the Smale-Hirsch theorem, the second statement is proven in \cite[Corollary 1.13]{B 72}. In fact, the standard machinery of surgery in the middle dimension then gives:

\begin{theorem} \label{corollary: lagrangian is a surgery solution}
Let $n \geq 4, n \neq 7$ and $\theta: B \rightarrow BO(2n+1)$ be weakly once stable with $B$ $n$-connected. Let $(M, \ell)$ be an $(n-1)$-connected, $2n$-dimensional, closed, $\theta$-manifold. Let $L \leq H_{n}(M)$ be a Lagrangian subspace for the self-intersection pairing. Then there exists a finite set $\Sigma$ and an embedding $f: \Sigma\times S^{n}\times D^{n} \rightarrow M$ that satisfies the following conditions:
\begin{enumerate} 
	\item[(a)] The $\theta$-structure on $\Sigma\times S^{n}\times D^{n}$ given by the composition,
	$$\xymatrix{T(\Sigma\times S^{n}\times D^{n})\oplus\epsilon^{1} \ar[rr]^-{Df\oplus\Id} && TM\oplus\epsilon^{1} \ar[rr]^{\ell} && \theta^{*}\gamma^{2n+1},}$$
	extends to a $\theta$-structure on $\Sigma\times D^{n+1}\times D^{n}$. 
	\item[(b)] The homology classes, $[f|_{\{\sigma\}\times S^{n}\times\{0\}}] \in H_{n}(M),$ $\sigma \in \Sigma,$ yield a basis for the Lagrangian subspace $L \leq H_{n}(M)$. 
\end{enumerate}
\end{theorem}

Note that the assumptions on $M$ automatically force $H_n(M)$ to be torsion-free and thus make the Lagrangian a free module as well, so that (b) indeed makes sense. Condition (b) in particular implies that the manifold $\widetilde{M}$ obtained by performing surgery on the embedding $f$ is $n$-connected, i.e. a homotopy sphere. 

For a proof see e.g.\ \cite[Proposition 5.2]{R 80} or \cite[Proposition 4.13]{CML} with two comments: First, the cited references use Wall's intersection pairing defined on the group of regular homotopy classes of immersions $S_n(M)$ (see \cite[Theorem 5.2]{W 70} or \cite{W 62b} for a definition). To compare this to Browder's, recall that a tangential structure map $\ell:M \rightarrow B$ gives a map $\pi_{n+1}(\ell) \rightarrow S_n(M)$ (see \cite[Lemma 4.60]{CML} for a pleasant exposition). For $n$-connected $B$ we then obtain a surjection $\pi_{n+1}(\ell) \rightarrow H_n(M)$ from the exact sequence of $\ell$ and it follows that the two compositions,
$$\xymatrix{\pi_{n+1}(\ell) \ar[r] &  H_n(M) \ar[r] & \mathbb Z/2 \quad\text{and}\quad\pi_{n+1}(\ell) \ar[r] & S_n(M) \ar[r] & \mathbb Z/2,}$$
agree, since both precisely obstruct the desired representability. Secondly, the references work with stable bundles over Poincar\'e complexes, but the Poincar\'e condition does not enter into the special case above, and the construction of the required bundle data really only uses weak once-stability (see the discussion in \cite[Sections 4.1 \& 5.1]{GRW 14}).

For the reader's convenience we supply a proof of the first fact:

\begin{proposition}
For an embedding $i: W \rightarrow M$ for a closed $2n$-manifold $M$ and a compact $2n$-manifold $W$, the map $i_*:H_n(W) \rightarrow H_n(M)$ preserves both the intersection and the self-intersection pairing.
\end{proposition}

\begin{proof}
That the intersection pairing is preserved is a simple calculation:
\begin{align*}
\lambda_M(i_*(x),i_*(y)) &= \langle i_*(x), D^Mi_*(y) \rangle_M \\
               &= \langle i_*(x), i^!D^{(W,\partial W)}y \rangle_M \\
               &= \langle i_!i_*(x), D^{(W,\partial W)}y \rangle_{(W,\partial W)} \\
						   &= \langle incl_*(x), D^{(W,\partial W)}y \rangle_{(W,\partial W)} \\
							 &= \lambda_W (x,y)
\end{align*}
Where the shriek maps are induced by $(M,\emptyset) \rightarrow (M,M-\Int W) \leftarrow (W,\partial W)$ (the right arrow inducing an isomorphism by excision), $incl$ denotes the inclusion $(W,\emptyset) \rightarrow W,\partial W)$ and $D^Mi_*(y) = i^! D^{(W,\partial W)} y$ follows from the naturality of cap products $H_*(X,A \cup B) \times H^*(X,A) \rightarrow H_*(X,B)$ applied to $(M,\emptyset, \emptyset) \rightarrow (M,M-\Int W, \emptyset)$ using the fact that $i_!([M]) = [W,\partial W]$. \\
The preservation of the self-intersection pairing now follows since also in $\Omega^{\langle n \rangle}(W,\partial W)$ we have $i_!([M]) = [W, \partial W]$ for the fundamental classes represented by the identity maps (since this can be checked locally around some point, by the definition of fundamental classes), so 
\begin{align*}
\mu_M(i_*(y)) &= \big[D^Mi_*(y) : (M,\emptyset) \rightarrow (K(\mathbb Z,n), pt)\big] \\
              &= \big[i^!D^{(W,\partial W)}y : (M,\emptyset) \rightarrow (K(\mathbb Z,n), pt)\big] \\
							&= (D^{(W,\partial W)}y)_* i_! ([M]) \\
							&= (D^{(W,\partial W)}y)_* ([W,\partial W]) \\
							&= \big[D^{(W,\partial W)}y: (W, \partial W) \rightarrow (K(\mathbb Z,n),pt)\big] \\
							&= \mu_{(W,\partial W)}(y)
\end{align*}
\end{proof}

\subsection{Cobordism categories of manifolds equipped with Lagrangian subspaces} \label{section: lagrangian cobordism category}
Fix a tangential structure $\theta: B \longrightarrow BO(2n+1)\langle n \rangle$ and suppose $n \geq 4, n \neq 7$.

\begin{defn} \label{defn: lagrangian cobordism category}
The non-unital topological category $\Cob^{\mathcal{L}}_{\theta}$ has as its object space the subspace of $\bpsi^{\Delta,n}_{\theta,2n}(\infty, 0)$ given by those tuples $(M, \ell, L)$ for which $L \leq H_{n}(M)$ is a Lagrangian subspace with respect to the intersection form $(H_{n}(M), \lambda, \mu)$. 
The space of morphisms is given by the following subspace of the product $\mathbb R \times \bpsi_{\theta,2n+1}(1 + \infty, 1) \times \bpsi^{\Delta,n}_{\theta,2n}(\infty, 0) \times \bpsi^{\Delta,n}_{\theta,2n}(\infty, 0)$: A tuple 
$$(t, (W,\ell),(M,\ell_M,L_M),(N,\ell_N,L_N))$$ 
is a morphism (from $(M,\ell_M,L_M)$ to $(N,\ell_N,L_N)$) if 
\begin{enumerate} 
\item[(i)] $(t,W,\ell) \in \Cob^{\mb{m}}_{\theta}((M,\ell_M),(N,\ell_N))$
\item[(ii)] $\iota_{in}(L_M) = \iota_{out}(L_N)$ as subspaces of $H_n(W|_{[0, t]})$.
\end{enumerate}
Here, 
$\iota_{\text{in}}: H_{n}(M) \longrightarrow H_{n}(W|_{[0, t]})$ and $\iota_{\text{out}}: H_{n}(N) \longrightarrow H_{n}(W|_{[0, t]})$
are the homomorphisms induced by the boundary inclusions
$M = W|_{0} \hookrightarrow W|_{[0, t]} \hookleftarrow W|_{t} = N$.
\end{defn}

\begin{remark}
We note that the forgetful functor $\Cob^{\mathcal L}_\theta \rightarrow \Cob_\theta$ is faithful. It is, however, the increase in extra structure on objects that makes the definition of the entire morphism space of $\Cob^{\mathcal L}_\theta$ more complicated than that of $\Cob_\theta$ as a cobordism no longer determines its source or target.
\end{remark}

When denoting a morphism in $\Cob^{\mathcal L}_\theta$ we will (nevertheless) usually drop the source and target from the notation and just write $(t, W, \ell)$ for $\left(t, (W,\ell),(M,\ell_M,L_M),(N,\ell_N,L_N)\right)$. We proceed to filter the cobordism category $\Cob^{\mathcal{L}}_{\theta}$ by subcategories analogous to those from Definition \ref{defn: disk category no L}. Let $D \subset (-\tfrac{1}{2}, 0]\times(-1, 1)^{\infty-1}$ be the $2n$-dimensional disk from (\ref{equation: fixed disk}). Let $\ell_{D}$ be the chosen $\theta$-structure on $D$ and let $\ell_{\R\times D}$ be the $\theta$-structure on $\R\times D$ induced by $\ell_{D}$. 

\begin{defn} \label{defn: disk category}
We define a sequence of subcategories of $\Cob^{\mathcal{L}}_{\theta}$ as follows:
\begin{enumerate} 
\item[(i)] The topological subcategory $\Cob^{\mathcal{L}, D}_{\theta} \subseteq \Cob^{\mathcal{L}}_{\theta}$ has as its objects those $(M, \ell, L)$ with
$(M, \ell) \in \Ob\Cob^{D}_{\theta}$. Similarly, it has as its morphisms those $(t, W, \ell)$ that give a morphism in $\Cob^{D}_{\theta}$.
\item[(ii)] Let $l \in \Z_{\geq -1}$.
The topological subcategory $\Cob^{\mathcal{L}, l}_{\theta} \subseteq \Cob^{\mathcal{L}, D}_{\theta}$ is the full-subcategory on those objects $(M, \ell, L)$ such that $M$ is $l$-connected, or in other words $(M, \ell) \in \Ob\Cob^{l}_{\theta}$.
\end{enumerate}
\end{defn}

In other words, the categories $\Cob^{\mathcal{L}, D}_{\theta}$ and $\Cob^{\mathcal{L}, l}_{\theta}$ are the evident pull-backs.

\begin{Observation} \label{remark: equality of n - L}
The forgetful functor $\Cob^{\mathcal{L}, n}_{\theta} \rightarrow \Cob^{n}_{\theta}$ is clearly an isomorphism, and thus $\Cob^{\mathcal{L}, n}_{\theta}$ can be considered a subcategory of $\Cob_{\theta}$.
\end{Observation}

\subsection{The technical theorems} \label{subsection: the main theorems}
We now state our results about the category $B\Cob^{\mathcal{L}}_{\theta}$ and granting them for the moment deduce the results in the introduction from them. Their proofs occupy the remaining sections, roughly in order. Let again $\theta\colon B \rightarrow BO(2n+1)\langle n\rangle$ be a map with $n \geq 4, n \neq 7$.

The first result occupies Section \ref{subsection: infinite loopspaces}: To state it recall that Nguyen \cite{N 15} constructed a $\Gamma$-space structure on $B\Cob_{\theta}$ underlain by disjoint union for which the equivalence to $\Omega^{\infty-1}MT\theta$ becomes one of infinite loopspaces.

\begin{theorem} \label{theorem: infinite loopspace structure}
The operation of disjoint union gives $B\Cob^{\mathcal{L}}_{\theta}$ the structure of a special $\Gamma$-space, such that the forgetful functor 
\[\Cob^{\mathcal{L}}_{\theta} \longrightarrow \Cob_{\theta}\]
induces a map of $\Gamma$-spaces.
In particular, $B\Cob^{\mathcal{L}, \emptyset}_{\theta}$ carries the structure of an infinite loopspace. 
\end{theorem}

\begin{remark}
One may wonder whether $B\Cob^{\mathcal{L}}_{\theta}$ itself is an infinite loopspace, the only question being whether the $\Gamma$-space structure from the above theorem makes it grouplike. This is indeed the case, but a proof is most readily given by showing $\Cob^{\mathcal L}_\theta$ to be equivalent to a cobordism category with no connectivity assumption on the morphisms (following the procedure in \cite{GRW 14}), where it is then immediate that the components form a group. Since we will not make use of this more general assertion we have omitted it.
\end{remark}

The next result is proven using the same ideas as \cite[Corollary 2.17]{GRW 14}, but in a different set-up. We indicate the necessary changes in Section \ref{section: Proof of disk inclusion}.

\begin{theorem} \label{theorem: inclusion of disk}
The inclusion $B\Cob^{\mathcal{L}, D}_{\theta} \hookrightarrow B\Cob^{\mathcal{L}}_{\theta}$ is a weak homotopy equivalence.
\end{theorem}

The next theorem is proven in Section \ref{section: surgery on objects below the middle dimension}. It is the first result of the paper whose proof requires a substantial amount of technical work. The constructions that go into it, however, closely resemble those from \cite{GRW 14}.

\begin{theorem} \label{theorem: surgery on objects below middle dim}
Let $l \leq n-1$ and assume that $B$ is $l$-connected and $\pi_{l+1}(B)$ is finitely generated. Then the inclusion $B\Cob^{\mathcal{L}, l}_{\theta} \hookrightarrow B\Cob^{\mathcal{L}, l-1}_{\theta}$ is a weak homotopy equivalence.
\end{theorem}

By combining the theorems stated above, we obtain the weak homotopy equivalence 
$$B\Cob^{\mathcal{L}, n-1}_{\theta} \simeq B\Cob^{\mathcal{L}}_{\theta}$$
in analogy with Theorem \ref{theorem: connected object equivalence}, whenever $B$ is $(n-1)$-connected and $\pi_n(B)$ finitely generated. Finally we have:

\begin{theorem} \label{theorem: surgery on objects in degree n} 
Suppose that $\theta$ is weakly once-stable and that $B$ is $n$-connected. Then the inclusion $B\Cob^{\mathcal{L}, n}_{\theta} \hookrightarrow B\Cob^{\mathcal{L}, n-1}_{\theta}$ 
is a weak homotopy equivalence.
\end{theorem}

We again emphasize that this theorem is in stark contrast to the situation for cobordism categories without Lagrangians. The proof will occupy Sections \ref{section: surgery on objects in the middle dimension} and \ref{section: contractibility n}. Via the isomorphism $\Cob^{\mathcal{L}, n}_{\theta} \cong \Cob^{n}_{\theta}$ the above theorems imply that there is a weak homotopy equivalence $B\Cob^{n}_{\theta} \simeq B\Cob^{\mathcal{L}}_{\theta}$ whenever $\theta$ satisfies the conditions of Theorem \ref{theorem: surgery on objects in degree n}.

\subsection{Deduction of the main results}
Supposing these four theorems, we now proceed to prove all results stated in the introduction. Let $n \geq 4$ be an integer except $7$ and specialise to $\theta^n$, the $n$-connected cover $BO(2n+1)\langle n \rangle \longrightarrow BO(2n+1)$. Theorem A, follows immediately by combining the weak homotopy equivalence $B\Cob^{n}_{\theta^{n}} \simeq B\Cob^{\mathcal{L}}_{\theta^{n}}$ with the weak homotopy equivalences 
$$B\mathcal{M}_{2n+1} \simeq B\mathcal{M}_{\theta^{n}} \simeq B\Cob^{n, \emptyset}_{\theta^{n}},$$ 
proven in Section \ref{section: reduction to a monoid}. Corollary B then follows by combining this with Theorem \ref{theorem: infinite loopspace structure}. 
Corollary C follows by combining the weak homotopy equivalence $B\mathcal{M}_{2n+1} \simeq B\Cob^{\mathcal{L}, \emptyset}_{\theta^{n}}$, with the homology equivalence $\mb{B}_{\infty} \rightarrow \Omega_{0}B\mathcal{M}_{2n+1}$ established in Section \ref{subsection: group completion}. 
To obtain Corollary D we do the following.  Consider the commutative diagram
$$\begin{xy}\xymatrix{\coprod_W \BDiff(W,D^{2n+1}) \ar[r]\ar[d]& \Omega_0 B \Cob_{\theta^{n}} \\
                          \mb{B}_{\infty} \ar[r]^-{\simeq_{H_*}}& \Omega_0 B \Cob_{\theta^{n}}^{\mathcal L} \ar[u]}
\end{xy}$$
and notice that the downwards left arrow (which is given by the inclusions of the terms of the colimit sequence into the colimit) is surjective on homology and thus injective in cohomology (with rational coefficients). The kernel of the left hand vertical map in cohomology is therefore the same as that of the top horizontal arrow. Composing with the weak homotopy equivalence $\Omega B\Cob_{\theta^{n}} \stackrel{\simeq} \longrightarrow \Omega^\infty\MT\theta^{n}$ yields the claim.


\section{The infinite loopspace structure} \label{subsection: infinite loopspaces}
In this section we equip $B\Cob^{\mathcal{L}, \emptyset}_{\theta}$ with the structure of a $\Gamma$-space, proving Theorem \ref{theorem: infinite loopspace structure}. 
In order to do this we first need to introduce a more convenient model for the nerve of $\Cob^{\mathcal{L}, \emptyset}_{\theta}$.

\subsection{A model for the nerve} 
Let 
$\theta\colon B \longrightarrow BO(2n+1)\langle n \rangle$ be  a tangential structure, with $n \geq 4, n \neq 7$.

\begin{defn} \label{defn: convenient model for nerve}
For each $p \in \Z_{\geq 0}$, $\mb{C}_{p}$ and $\widehat{\mb C}_p$, respectively, are defined to be the spaces of tuples $(a, \varepsilon, W, \ell)$ where:
\begin{itemize} 
\item[(i)] $a \in \R^{p+1}$, $\varepsilon \in (0, \infty)^{p+1}$ are $(p+1)$-tuples 
with the property that $a_{i} + \varepsilon_{i} < a_{i+1} - \varepsilon_{i+1}$, and $a_i \leq a_{i+1}$ for all $i = 0, \dots, p-1$, respectively;
\item[(ii)] $(W, \ell)$ is an element of $\bpsi_{\theta}(\infty, 1)$ with the property that $(W, \ell)|_{(a_{i}-\varepsilon_{i}, a_{i}+\varepsilon_{i})}$ is cylindrical for all $i = 0, \dots, p$, i.e. equal to a cylinder on $W_{|a_i}$.
\end{itemize}
Both $\mb{C}_{p}$ and $\widehat{\mb C}_p$ are topologized as subspaces of $\R^{p+1}\times\R^{p+1}\times\bpsi_{\theta}(\infty, 1)$ and the assignments 
\[[p] \longmapsto \mb{C}_{p} \quad \text{and} \quad [p] \longmapsto \widehat{\mb{C}}_{p}\] 
yield semi-simplicial spaces with $i$-th face map given by deleting $a_i$ and $\varepsilon_i$. $\widehat{\mb{C}}_{\bullet}$ is in fact a simplicial space with $i$-th degeneracy map doubling $a_i$ and $\varepsilon_i$.
\end{defn} 

$\mb C_\bullet$ appears in the proof of \cite[Proposition 2.14]{GRW 14} as the nerve of a certain poset $\mb D_\theta^\perp$ (a slightly different poset appears in the proof of \cite[Theorem 3.9]{GRW 09} under the same name; we, however, want to reserve the letter $\mb D$ for yet another object that incorporates Lagrangians and will be made use of in the same way $\mb D$ is used in the later parts of \cite{GRW 14}, see for example Section \ref{section: surgery on objects below the middle dimension}).

For each $p$-simplex $(a, \varepsilon, W, \ell) \in \mb{C}_{p}$, the tuple $(a, W, \ell)$ determines a unique element in $N_{p}\Cob_{\theta}$ by remembering only the various $W|_{[a_i,a_{i+1}]}$ (appropriately translated and with infinite collars attached) with their induced $\theta$-structures. This correspondence defines a semi-simplicial map $\mb{C}_{\bullet} \longrightarrow N_{\bullet}\Cob_{\theta}$ which is easily shown to be a level-wise weak homotopy equivalence (see \cite[Theorem 3.9]{GRW 09} for details). 

The reason for including the version with degeneracies is as follows: We shall momentarily construct a simplicial $\Gamma$-space using (Lagrangian enhancements of) the spaces $\widehat{\mb C}_p$ that is levelwise special. The speciality is then automatically inherited by its realization and will give the main result of the section; the same is not true when working with $\mb C_p$ and its associated \emph{semi}-simplicial $\Gamma$-space. We thank Johannes Ebert for pointing out this oversight in a previous version. We shall furthermore use the spaces $\widehat{\mb C}_\bullet$ to correct a small mistake in the proofs of \cite[Theorems 4.5 \& 5.14]{GRW 14} in turn, see Lemma \ref{lemma: inclusion of zero simplices}. To compare the two versions we have:

\begin{lemma}\label{lemma: equivalence of geometric realizations}
The simplicial space $\hat{\mb C}_\bullet$ is \emph{good}, i.e. its degeneracies are closed cofibrations, and the inclusion $\mb C_\bullet \longrightarrow \widehat{\mb C}_\bullet$ is a levelwise equivalence and thus in particular realizes to a weak equivalence.
\end{lemma}

By the first part it does not matter whether we use the thin or the thick realisation for the target, as they are equivalent by \cite[Proposition A.1, (iv)]{S 74}.

\begin{proof}
To see that $s_{i}$ is a closed cofibration it is enough (e.g. by combining \cite[Satz 3.13 and Satz 3.26]{DKP 70}) to show that (i) $s_{i}(\widehat{\mb{C}}_{p})  \subset \widehat{\mb{C}}_{p+1}$ is given as the vanishing set of some real-valued function and (ii) admits a halo $U$ that contracts onto it (recall that any neighbourhood given by a strict inequality for some positive function to the real numbers is a halo \cite[Definition 3.1]{DKP 70}).

For property (i) consider
\[
(a, \varepsilon, W, \ell) \longmapsto  |\varepsilon_i - \varepsilon_{i+1}| + a_{i+1} - a_{i}. 
\]
Property (ii) is witnessed by the neighborhood 
$$
U = \left\{(a, \varepsilon, W, \ell) \in \widehat{\mb{C}}_{p+1} \; | \; \frac{a_{i+1} - a_{i}}{\min\{\varepsilon_{i}, \varepsilon_{i+1}\}} <1\; \right\},
$$
with the deformation retraction moving $a_{i+1}$ towards $a_i$ (using the fact that $W$ is cylindrical between $a_i$ and $a_{i+1}$ for elements in $U$) and shrinking $\varepsilon_i$ and $\varepsilon_{i+1}$ as required.

The second part is trivial.
\end{proof}

\begin{defn} \label{defn: convenient model for nerve L}
For each $p \in \Z_{\geq 0}$, $\mb{C}^{\mathcal{L}}_{p}$ and $\widehat{\mb{C}}^{\mathcal{L}}_{p}$, respectively, are the spaces of tuples $(a, \varepsilon, W, \ell, L)$ subject to the following conditions:
\begin{enumerate}
\item[(a)] The tuple $(a, \varepsilon, W, \ell)$ is an element of  $\mb{C}_{p}$ or $\widehat{\mb C}_p$, respectively.
Furthermore, for each $i = 0, \dots, p$, the pair $(W|_{[a_{i}, a_{i+1}]}, W|_{a_{i+1}})$ is $(n-1)$-connected, 
\item[(b)] $L = (L_{0}, \dots, L_{p})$ is a $(p+1)$-tuple with $L_{i} \leq H_{n}(W|_{a_{i}})$ a Lagrangian subspace for each $i = 0, \dots, p$.
\item[(c)] For each $i = 0, \dots, p,$ we require $\iota^{\text{in}}_{i}(L_{i}) = \iota^{\text{out}}_{i+1}(L_{i+1})$ where 
$$\iota^{\text{in}}_{i}: H_{n}(W|_{a_{i}}) \longrightarrow H_{n}(W|_{[a_{i}, a_{i+1}]}) \quad \text{and} \quad \iota^{\text{out}}_{i+1}: H_{n}(W|_{a_{i+1}}) \longrightarrow H_{n}(W|_{[a_{i}, a_{i+1}]})$$ 
are the maps induced by inclusion.
\end{enumerate}
\end{defn}

Each space $\mb{C}^{\mathcal{L}}_{p}$ is topologized as a subspace of the product $\mb{C}_{p}\times(\Ob\Cob^{\mathcal{L}}_{\theta})^{p+1}$. 
The assignment $[p] \mapsto \mb{C}^{\mathcal{L}}_{p}$ defines a semi-simplicial space with face maps defined similarly to the previous definition and similarly for the simplicial version. As in the case with $\mb{C}_{\bullet}$, there is a semi-simplicial map $\mb{C}^{\mathcal{L}}_{\bullet} \longrightarrow N_{\bullet}\Cob_{\theta}^{\mathcal{L}}$ and applying the same arguments as before we obtain weak equivalences
\begin{equation} \label{models for cobl}
|\widehat{\mb{C}}^{\mathcal{L}}_{\bullet}| \longleftarrow |\mb{C}^{\mathcal{L}}_{\bullet}| \longrightarrow B\Cob_{\theta}^{\mathcal{L}}.
\end{equation}
Corresponding to the sequence of subcategories from Definition \ref{defn: disk category}, we have a sequence of sub-semi-simplicial spaces 
\begin{equation} \label{equation: filtration of c-spaces}
\mb{C}^{\mathcal{L}, l}_{\bullet} \subset \cdots \subset \mb{C}^{\mathcal{L}, -1}_{\bullet} \subset \mb{C}^{\mathcal{L}, D}_{\bullet} \subset \mb{C}^{\mathcal{L}}_{\bullet},
\end{equation}
defined analogously to the subcategories of Definition \ref{defn: disk category}. Applying the construction from \cite[Theorem 3.9]{GRW 09} again, we also obtain weak homotopy equivalences,
$$|\mb{C}^{\mathcal{L}}_{\bullet}| \simeq B\Cob^{\mathcal{L}}_{\theta}, \quad  |\mb{C}^{\mathcal{L}, D}_{\bullet}| \simeq B\Cob^{\mathcal{L}, D}_{\theta},  \quad |\mb{C}^{\mathcal{L}, l}_{\bullet}| \simeq B\Cob^{\mathcal{L}, l}_{\theta}.$$
The analoguous result using the full simplicial spaces also holds, but we shall have no need for it.

\subsection{The $\Gamma$-space}
We now proceed to give $|\widehat{\mb{C}}_{\bullet}^{\mathcal L}|$ the structure of a $\Gamma$-space. We follow \cite{N 15} and make every $\widehat{\mb{C}}^{\mathcal L}_p$ into the underlying space of a $\Gamma$-space $A_p$, which assemble into a simplicial $\Gamma$-space $A_\bullet$. By (\ref{models for cobl}) the realization (in the simplicial direction) is then a $\Gamma$-space underlain by $B\Cob_{\theta}^{\mathcal{L}}$. 

\begin{Construction} \label{theorem: gamma space structure on D-L}
Let $S$ be a finite set. We define $A_p(S)$ to be the subset of $(\widehat{\mb{C}}^{\mathcal{L}}_{p})^S$ 
consisting of those tuples 
$$(a, \varepsilon, W, \ell, L)$$
that satisfy for all $s,s' \in S$:
\begin{enumerate}
\item[(a)] $a_{s} = a_{s'}$ and $\varepsilon_s = \varepsilon_{s'}$
\item[(b)] the submanifolds, 
$W_{s}$ and $W_{s'} \; \subset \; \R\times(-1, 1)^{\infty}$,
are disjoint. 
\end{enumerate}
For the rest of this section, we will suppress $\varepsilon$ and $\ell$ from notation as they will play no role. Given a morphism $\phi: S \longrightarrow T$ (in the category $\Gamma$), the map 
$$A_{p}(\phi): A_{p}(T) \longrightarrow A_{p}(S)$$ 
is defined by sending a tuple $(a, W, L) \in A_{p}(T)$ to the element in $A_p(S)$ whose entry in the $s$-th spot is:
\[\left(a, \; \bigsqcup_{t \in \phi(s)}W_{t}, \sum_{t \in \phi(s)} L_{t}\right).\]
It is readily checked that this really makes $A_p$ into a $\Gamma$-space and indeed $A_\bullet$ into a simplicial $\Gamma$-space. 
\end{Construction}

Note that $A_\bullet(1) = \widehat{\mb{C}}^{\mathcal{L}}_{\bullet}$ and therefore
 \[|A_\bullet|(1) =  |\widehat{\mb{C}}^{\mathcal{L}}_\bullet| \simeq B\Cob^{\mathcal L}_\theta.\]

\begin{lemma}\label{a is special}
The $\Gamma$-space $A_p$ is special for every $p$.
\end{lemma}

\begin{proof}
We need to show that the Segal maps
\[p_{k}: A_{p}(k) \longrightarrow A_{p}(1)^k = (\widehat{\mb{C}}^{\mathcal{L}}_{p})^k\]
are weak homotopy equivalences. Note that $p_k$ is just the inclusion of disjoint $k$-tuples of manifolds (with equal $a$'s and $\varepsilon$'s) into the space of all $k$-tuples. But clearly an arbitrary $k$-tuple can be made disjoint by shrinking the $i$-th manifold for example into $\mathbb R \times (\frac i k,\frac {i+1} k) \times (-1,1)^{\infty-1}$ and so the fact that the space of embeddings of a manifold into $\mathbb R^\infty$ is contractible gives the claim, compare e.g. \cite[Proposition 5.10]{N 15}.
\end{proof}

\begin{proposition} \label{corollary: geometric realization of gamma space}
The geometric realization $B\Cob^{\mathcal L}_\theta \simeq |\widehat{\mb{C}}^{\mathcal{L}}_{\bullet}|$ has the structure of a special $\Gamma$-space.
\end{proposition}

\begin{proof}
All that remains to check is that 
\[p_{k}: |A_{\bullet}|(k) \longrightarrow |A_{\bullet}|(1)^k\]
is a weak equivalence. But the same argument as in Lemma \ref{lemma: equivalence of geometric realizations} shows that the simplicial space $A_\bullet(k)$ is good for every $k$, whence the claim follows from Lemma \ref{a is special}, since levelwise weak equivalences between good simplicial spaces are preserved by realization.
\end{proof}

Alternatively, we could have used the thick realisation of the simplicial space $A_\bullet(S)$ and applied \cite[Theorem 7.2]{ERW 17} bypassing the goodness of $\widehat{\mb{C}}^{\mathcal{L}}_{\bullet}$.

\begin{corollary}
The space $\Omega B\Cob^{\mathcal L}_\theta$ admits an infinite loopspace structure.
\end{corollary}

\begin{proof}
Note only that the identity component of a special $\Gamma$-space is automatically very special and thus an infinite loopspace by \cite[Proposition 1.4]{S 74}.
\end{proof}

It is evident straight from the definitions that the $\Gamma$-space structure just constructed is compatible with that of \cite{N 15}.


\section{Alternate Models for the Nerve} \label{section: A substitute for the cobordism category}
The category defined in Section \ref{section: lagrangian cobordism category} is difficult to analyze directly, as is $\mb C^\mathcal L_\bullet$. 
In order to prove Theorems \ref{theorem: inclusion of disk} - \ref{theorem: surgery on objects in degree n} we will need to work with a more flexible substitute $\mb X_\bullet^\mathcal L$ for the semi-simplicial nerve. In fact we will have three models $\mb C^\mathcal L_\bullet$, $\mb D^\mathcal L_\bullet$ and $\mb X^\mathcal L_\bullet$ by the end of this section; we named these by analogy with the objects from \cite{GRW 14}.  The middle model is forced on us by the definition of $X_\bullet$ from \cite[Definition 2.8]{GRW 14}: Its constituent manifolds no longer have preferred slices in which to place Lagrangians. The semi-simplicial space $\mb D_\bullet^\mathcal L$ therefore reformulates the Lagrangians as objects on the entirety of a long manifold rather than its slices. The relation between $\mb D_\bullet^\mathcal L$ and $\mb X_\bullet^\mathcal L$ will then be the same as that between $\mb D_\bullet$ and $\mb X_\bullet$ in \cite{GRW 14}.

\subsection{Models with spread-out Lagrangians}
Recall the semi-simplicial space $\mb{C}_{\bullet}$ from Definition \ref{defn: convenient model for nerve} and $\mb{C}^{\mathcal{L}}_{\bullet}$ from Definition \ref{defn: convenient model for nerve L}. 

\begin{defn} \label{defn: basic lagrangian cob cat}
For $p \in \Z_{\geq 0}$, $\mb{D}^{\mathcal{L}}_{p}$ is defined to be the space of tuples $(a, \varepsilon, (W, \ell), V)$ subject to the following conditions:
\begin{enumerate} 
	\item[(i)] The tuple $(a, \varepsilon, (W, \ell))$ is an element of $\mb{C}_{p}$ with the property that the pair $(W|_{[a_{i}, a_{i+1}]}, W|_{a_{i+1}})$ is $(n-1)$-connected for all $i = 0, \dots, p-1$.
	\item[(ii)] $V = (V_{0}, \dots, V_{p})$ is a $p$-tuple of subspaces 
	$$V_{i} \leq H^{\cpt}_{n+1}(W|_{(a_{0}-\varepsilon_{0}, a_{p}+\varepsilon_{p})}) \quad  \text{for $i = 0, \dots, p$,}$$ 
	subject to the following conditions:
	\begin{enumerate} 
		\item[(a)] For each $i$, the restriction $V_{i}|_{a_{i}} \; \leq \; H_{n}(W|_{a_{i}})$ is a Lagrangian subspace (recall from Section \ref{section: homological preliminaries} that $V_{i}|_{a_{i}}$ is the image of $V_{i}$ under the map $j_{!}: H^{\cpt}_{n+1}(W) \longrightarrow H_{n}(W|_{a_{i}})$).
		\item[(b)] Let $i \neq j$. Then the subspace $V_{i}|_{a_{j}} \leq H_{n}(W|_{a_{j}})$ is contained in the subspace $V_{j}|_{a_{j}}$.
	\end{enumerate}
\end{enumerate}
To topologize $\mb{D}^{\mathcal{L}}_{p}$ we use the following construction.  Choose once and for all a family of increasing diffeomorphisms 
\[\psi = \psi(a_{0}, \varepsilon_{0}, a_{p}, \varepsilon_{p}):  (0, 1) \stackrel{\cong} \longrightarrow (a_{0} - \varepsilon_{0}, a_{p} +\varepsilon_{p}),\]
varying smoothly in the data $(a_{0}, \varepsilon_{0}, a_{p}, \varepsilon_{p})$. We then let 
$$\bar{\psi} = \bar{\psi}(a_{0}, \varepsilon_{0}, a_{p}, \varepsilon_{p}): (0, 1)\times\R^{\infty} \longrightarrow (a_{0} - \varepsilon_{0}, a_{p} +\varepsilon_{p})\times\R^{\infty}$$
be the smooth family of diffeomorphisms given by the product, $\psi\times\Id_{\R^{\infty}}$. For each $i$, we define a map $\pi_{i}: \mb{D}^{\mathcal{L}}_{p} \longrightarrow \bPsi_{\theta}^{\Delta}((0, 1)\times\R^{\infty})$ by 
$$(a, \varepsilon, W, \ell, V) \; \mapsto \; \left(\bar{\psi}^{-1}(W|_{(a_{0}, \varepsilon_{0}, a_{p}, \varepsilon_{p})}), \; \; \ell|_{(a_{0}, \varepsilon_{0}, a_{p}, \varepsilon_{p})}\circ D\bar{\psi}, \; \; \bar{\psi}^{-1}(V_{i})\right).$$
Using these maps we obtain an embedding $\mb{D}^{\mathcal{L}}_{p} \hookrightarrow \mb{C}_{p}\times\bPsi_{\theta}^{\Delta}((0, 1)\times\R^{\infty})^{p+1}	$ defined by the formula
$$(a, \varepsilon, W, \ell, V) \; \; \mapsto \; \; \left((a, \varepsilon, W, \ell),  \; \; \pi_{0}(a, \varepsilon, W, \ell, V),\; \dots, \; \pi_{p}(a, \varepsilon, W, \ell, V)\right).$$
By this embedding we topologize $\mb{D}^{\mathcal{L}}_{p}$ as a subspace of $\mb{C}_{p}\times\bPsi_{\theta}^{\Delta}((0, 1)\times\R^{\infty})^{p+1}$. For $0 < i < p$, the face maps $d_{i}: \mb{D}^{\mathcal{L}}_{p} \longrightarrow \mb{D}^{\mathcal{L}}_{p-1}$ are defined by 
$$d_{i}(a, \varepsilon, W, \ell, V) \; = \; (a(i), \varepsilon(i), W, \ell, V(i))$$
where $a(i)$, $\varepsilon(i)$, and $V(i)$ are the $(p-1)$-tuples obtained by removing the $i$-th entry from the $p$-tuples $a$, $\varepsilon$, and $V$ respectively. For the face maps $d_{0}, d_{p}: \mb{D}^{\mathcal{L}}_{p} \longrightarrow \mb{D}^{\mathcal{L}}_{p-1}$, a small change is needed in the definition. The map $d_{0}$ is defined by 
$$d_{0}(a, \varepsilon, W, \ell, V) = \left(a(0), \; \varepsilon(0), \; W, \; \ell, \; V(0)|_{(a_{1}-\varepsilon_{1}, a_{p}+\varepsilon_{p})}\right).$$
and $d_{p}$ is defined by 
$$d_{p}(a, \varepsilon, W, \ell, V) = \left(a(p), \; \varepsilon(p), \; W, \; \ell, \; V(p)|_{(a_{0}-\varepsilon_{0}, a_{p-1}+\varepsilon_{p-1})}\right).$$
These face maps are continuous as a consequence of Proposition \ref{conti} and the assignment $[p] \mapsto \mb{D}^{\mathcal{L}}_{p}$ makes $\mb{D}^{\mathcal{L}}_{\bullet}$ into a semi-simplicial space.
\end{defn}

Since any two families of diffeomorphisms as in (\ref{equation: parametrization}) are isotopic the topology on $\mb{D}^{\mathcal{L}}_{p}$ is independent of the choice of $\psi$. The reason for having distinct subspaces $V_i$ (instead of a single one that restricts to a Lagrangian for every slice) is technical in nature: We do not know how to show that $|\mb C^\mathcal L_\bullet| \simeq |\mb D^\mathcal L_\bullet|$ for that variant of the definition.

We filter $\mb{D}^{\mathcal{L}}_{\bullet}$ by a sequence of sub-semi-simplicial spaces 
\begin{equation} \label{equation: sequence of D subspaces}
\mb{D}^{\mathcal{L}, n}_{\bullet} \subset \cdots \subset \mb{D}^{\mathcal{L}, -1}_{\bullet}  \subset \mb{D}^{\mathcal{L}, D}_{\bullet} \subset \mb{D}^{\mathcal{L}}_{\bullet},
\end{equation}
defined analogously to (\ref{equation: filtration of c-spaces}).

In order to compare $\mb{D}^{\mathcal{L}}_{\bullet}$ to $\mb{C}^{\mathcal{L}}_{\bullet}$ we need the following proposition. 

\begin{proposition} \label{proposition: d to c map}
Let $p \in \Z_{\geq 0}$. For any $(a, \varepsilon, (W, \ell), V) \in \mb{D}^{\mathcal{L}}_{p}$ the associated tuple
$$(a, \varepsilon, (W, \ell), V_{0}|_{a_{0}}, \dots, V_{p}|_{a_{p}})$$
is an element of $\mb{C}^{\mathcal{L}}_{p}$. Thus the correspondence
$$(a, \varepsilon, (W, \ell), V) \; \mapsto \; (a, \varepsilon, (W, \ell), V_{0}|_{a_{0}}, \dots, V_{p}|_{a_{p}})$$
yields a well defined semi-simplicial map $\mathcal{F} \colon \mb{D}^{\mathcal{L}}_{\bullet} \longrightarrow \mb{C}^{\mathcal{L}}_{\bullet}$.
\end{proposition}

\begin{proof}
Let $(a, \varepsilon, (W, \ell), V) \in \mb{D}^{\mathcal{L}}_{p}$ with $V = (V_{0}, \dots, V_{p})$. We need to show that for all $0\leq i < p$
$$\iota^{\text{in}}(V_{i}|_{a_{i}}) \; = \; \iota^{\text{out}}(V_{i+1}|_{a_{i+1}}),	$$
where 
$$\iota^{\text{in}}: H_{n}(W|_{a_{i}}) \longrightarrow H_{n}(W|_{[a_{i}, a_{i+1}]}) \quad \text{and} \quad \iota^{\text{out}}: H_{n}(W|_{a_{i+1}}) \longrightarrow H_{n}(W|_{[a_{i}, a_{i+1}]})$$
are the maps induced by inclusion. Let $x \in V_{i}|_{a_{i}}$ and choose $v \in V_{i}$ such that $v|_{a_{i}} = x$. By Definition \ref{defn: basic lagrangian cob cat} (condition (ii), part (b)), we have $v|_{a_{i+1}} \in V_{i+1}|_{a_{i+1}}$. 	Let 
$$\bar{v} \in H_{n+1}(W|_{[a_{i}, a_{i+1}]}, W|_{a_{i}}\sqcup W|_{a_{i+1}})$$ 
denote the image of $v$ under
$$\xymatrix{H^{\cpt}_{n+1}(W) \ar[rr]^-{-|_{[a_i,a_{i+1}]}} &&  H_{n+1}(W|_{[a_{i}, a_{i+1}]}, W|_{a_{i}}\sqcup W|_{a_{i+1}})}$$
One readily checks from the definition that the diagram
\[\xymatrix{& H^{\cpt}_{n+1}(W) \ar[ld]_{-|_{[a_i,a_{i+1}]}} \ar[rd]^{-|_{a_{i+\nu}}} & \\
H_{n+1}(W|_{[a_{i}, a_{i+1}]}, \; W|_{a_{i}}\sqcup W|_{a_{i+1}}) \ar[r]^-{\partial} & H_{n}(W|_{a_{i}}\sqcup W|_{a_{i+1}}) \ar[r]^-{\text{pr}_\nu} & H_{n}(W|_{a_{i+\nu}})}\]
commutes up to the sign $(-1)^{\nu+1}$ for $\nu = 0, 1$. It follows that 
\[\partial(\bar{v}) =  v|_{a_{i+1}} - v|_{a_{i}}.\]
By exactness of 
$$\xymatrix{H_{n+1}(W|_{[a_{i}, a_{i+1}]}, W|_{a_{i}}\sqcup W|_{a_{i+1}}) \ar[r]^-{\partial} & H_{n}(W|_{a_{i}}\sqcup W|_{a_{i+1}}) \ar[rrr]^-{\iota^{\text{in}} + \iota^{\text{out}}} &&& H_{n}(W|_{[a_{i}, a_{i+1}]}),}$$
we then find
$$\iota^{\text{in}}(x) = \iota^{\text{in}}(v|_{a_{i}}) = \iota^{\text{out}}(v|_{a_{i+1}}),$$
so $\iota^{\text{in}}(x) \in \iota^{\text{out}}(V_{i+1}|_{a_{i+1}})$. Thus
$$\iota^{\text{in}}(V_{i}|_{a_{i}}) \leq \iota^{\text{out}}(V_{i+1}|_{a_{i+1}}).$$
and exchanging indices shows that $\iota^{\text{in}}(V_{i}|_{a_{i}}) \geq \iota^{\text{out}}(V_{i+1}|_{a_{i+1}}).$ 
\end{proof}

The following theorem is the main result of this section, its proof occupies the entire next section.

\begin{theorem} \label{theorem: replacement of nerve}
The semi-simplicial map $\mathcal F \colon \mb{D}^{\mathcal{L}}_{\bullet} \longrightarrow \mb{C}^{\mathcal{L}}_{\bullet}$ induces the weak homotopy equivalences 
$$|\mb{D}^{\mathcal{L}}_{\bullet}| \simeq |\mb{C}^{\mathcal{L}}_{\bullet}|, \quad |\mb{D}^{\mathcal{L}, D}_{\bullet}| \simeq |\mb{C}^{\mathcal{L}, D}_{\bullet}|,  \quad |\mb{D}^{\mathcal{L}, l}_{\bullet}| \simeq |\mb{C}^{\mathcal{L}, l}_{\bullet}|,$$
for all $l \in \Z_{\geq 0}$.
\end{theorem}

\subsection{Proof of Theorem \ref{theorem: replacement of nerve}} \label{subsection: proof of d to c equivalence}
We will only explicitly prove the weak homotopy equivalence $|\mb{D}^{\mathcal{L}}_{\bullet}| \simeq |\mb{C}^{\mathcal{L}}_{\bullet}|$. The other weak homotopy equivalences asserted in the theorem are established by repeating the exact same argument, which is largely formal. The proof makes use of the simplicial technique introduced in \cite[Section 6.2]{GRW 14} and largely follows \cite[Section 6.3 \& 6.4]{GRW 14}. The first step is to define an augmented bi-semi-simplicial space $\mb{C}^{\mathcal{L}}_{\bullet, \bullet} \longrightarrow \mb{C}^{\mathcal{L}}_{\bullet, -1}$, with $\mb{C}^{\mathcal{L}}_{\bullet, -1} = \mb{C}^{\mathcal{L}}_{\bullet}$.

\begin{defn} \label{defn: bi-simplicial resolution 1}
Let $p \in \Z_{\geq 0}$ and let $x = (a, \varepsilon, (W, \ell), L) \in \mb{C}^{\mathcal{L}}_{p}$. For each $q \in \Z_{\geq-1}$, we define ${\mb{Z}}_{q}(x)$ to be the set of tuples $(V^{0}, \dots, V^{q})$ subject to the following conditions:
\begin{enumerate}
	\item[(i)] Each $(a, \varepsilon, (W, \ell), V^{j})$ is an element of  $\mb{D}^{\mathcal{L}}_{p}$. In other words for each $j$, $V^{j} = (V^{j}_{0}, \dots, V^{j}_{p})$ is a $(p+1)$-tuple of subspaces of $H^{\cpt}_{n+1}(W|_{(a_{0}-\varepsilon_{0}, a_{p}+\varepsilon_{p})})$, subject to the conditions from Definition \ref{defn: basic lagrangian cob cat}.
	\item[(ii)] The equality 
	$$V^{j}_{i}|_{a_{i}} \; = \; L_{i}$$ 
	holds for all $j = 0, \dots, q$ and $i = 0, \dots, p$. In other words, $\mathcal{F}(a, \varepsilon, W, \ell, V^{j}) \; = \; (a, \varepsilon, W, \ell, L)$ for all $j = 0, \dots, q$, where recall that $\mathcal{F}$ is the map from (\ref{proposition: d to c map}).
\end{enumerate}
For $p, q \in \Z_{\geq-1}$, the space ${\mb{C}}^{\mathcal{L}}_{p, q}$ is defined by 
\[{\mb{C}}^{\mathcal{L}}_{p, q} = \{(x, y) \; | \; x \in \mb{C}^{\mathcal{L}}_{p} \quad \text{and} \quad y \in {\mb{Z}}_{q}(x)\}\]
and topologized as a subspace of $(\mb D_p^{\mathcal L})^{q+1}$. The assignment $[p, q] \mapsto {\mb{C}}^{\mathcal{L}}_{p, q}$ defines a bi-semi-simplicial space ${\mb{C}}^{\mathcal{L}}_{\bullet, \bullet}$. The forgetful maps 
$${\mb{C}}^{\mathcal{L}}_{p, q} \longrightarrow \mb{C}^{\mathcal{L}}_{p}, \quad (x, y) \mapsto x,$$ 
define an augmented bi-semi-simplicial space ${\mb{C}}^{\mathcal{L}}_{\bullet, \bullet} \longrightarrow \mb{C}^{\mathcal{L}}_{\bullet, -1} = \mb{C}^{\mathcal{L}}_{\bullet}$. 
\end{defn}

\begin{Observation} \label{remark: contractibility fo Z}
By condition (i) in the above definition, it follows that ${\mb{Z}}_{q}(x) \cong [\mb{Z}_{0}(x)]^{q+1}$ for all $x$. It follows that the semi-simplicial set given by the correspondence $[q] \mapsto {\mb{Z}}_{q}(x)$ is contractible whenever it is non-empty.
\end{Observation}

Notice that the semi-simplicial space ${\mb{C}}^{\mathcal{L}}_{\bullet, 0}$ is nothing but $\mb{D}^{\mathcal{L}}_{\bullet}$. Under this identification the forgetful map ${\mb{C}}^{\mathcal{L}}_{\bullet, 0} \longrightarrow {\mb{C}}^{\mathcal{L}}_{\bullet}$ coincides with the map $\mathcal{F}: \mb{D}^{\mathcal{L}}_{\bullet} \longrightarrow \mb{C}^{\mathcal{L}}_{\bullet}.$ Inclusion of zero-simplices yields an embedding $|\mb{D}^{\mathcal{L}}_{\bullet}| = |{\mb{C}}^{\mathcal{L}}_{\bullet, 0}|  \hookrightarrow |{\mb{C}}^{\mathcal{L}}_{\bullet, \bullet}|$.
To prove Theorem \ref{theorem: replacement of nerve} it will suffice to prove that the maps 
$$|{\mb{C}}^{\mathcal{L}}_{\bullet, 0}| \hookrightarrow |{\mb{C}}^{\mathcal{L}}_{\bullet, \bullet}| \longrightarrow |{\mb{C}}^{\mathcal{L}}_{\bullet}|
$$
are both weak homotopy equivalences, where the first map is given by inclusion of zero-simplices and the second is induced by the augmentation.
We break this up into two steps, Lemma \ref{lemma: augmentation equivalence} and Lemma \ref{lemma: inclusion of zero simplices}. 

\begin{lemma} \label{lemma: augmentation equivalence}
The map $|{\mb{C}}^{\mathcal{L}}_{\bullet, \bullet}| \longrightarrow |{\mb{C}}^{\mathcal{L}}_{\bullet}|$ induced by the augmentation is a weak homotopy equivalence. 
\end{lemma}

\begin{proof}
We will apply \cite[Theorem 6.2]{GRW 14} for each $p \in \Z_{\geq 0}$ to show that the induced maps $|{\mb{C}}^{\mathcal{L}}_{p, \bullet}| \longrightarrow {\mb{C}}^{\mathcal{L}}_{p}$ are weak homotopy equivalences for each $p \in \Z_{\geq 0}$. Geometrically realizing the first coordinate will then imply the lemma. We thus need to verify conditions (i), (ii), and (iii) from \cite[Theorem 6.2]{GRW 14} (it is clear that ${\mb{C}}^{\mathcal{L}}_{p, \bullet} \rightarrow {\mb{C}}^{\mathcal{L}}_{p}$ is an augmented topological flag complex as in \cite[Definition 6.1]{GRW 14}). 

Condition (i) is proven similarly to \cite[Proposition 6.10]{GRW 14} and so we omit the proof. Condition (iii) is trivial (see Observation \ref{remark: contractibility fo Z}). We proceed to verify condition (ii).

Let $x = (a, \varepsilon, (W, \ell), L) \in  \mb{C}^{\mathcal{L}}_{p}$ be with $L = (L_{0}, \dots, L_{p})$. We will need to show that $\mb{Z}_{0}(x)$ is non-empty. By the definition of $\mb{C}^{\mathcal{L}}_{p}$, we have 
\begin{equation} \label{equation: equal image condition}
\iota^{\text{in}}(L_{0}) \; = \; \iota^{\text{out}}(L_{1})
\end{equation}
as subspaces of $H_{n}(W|_{[a_{0}, a_{1}]})$, where $\iota^{\text{in}}$ and $\iota^{\text{out}}$ are the maps induced by the inclusions 
$$W|_{a_{0}} \hookrightarrow W|_{[a_{0}, a_{1}]} \hookleftarrow W|_{a_{1}}.$$
Let $x_{1}, \dots, x_{k} \in L_{0}$ be a set of generators. By (\ref{equation: equal image condition}), for each $i = 1, \dots, k$, we may choose $y_{i} \in L_{1}$ such that 
$$\iota^{\text{in}}(x_{i}) = \iota^{\text{out}}(y_{i}).$$
By exactness of  
$$\xymatrix{H_{n+1}(W|_{[a_{0}, a_{1}]}, \; W|_{a_{0}}\sqcup W|_{a_{1}}) \ar[r]^-{\partial} & H_{n}(W|_{a_{0}}\sqcup W|_{a_{1}}) \ar[r]^-{\iota^{\text{in}} + \iota^{\text{out}}} & H_{n}(W|_{[a_{0}, a_{1}]}),}$$
it follows that for each $i = 1, \dots, k$, there exists a class $w_{i} \in H_{n+1}(W|_{[a_{0}, a_{1}]}, \; W|_{a_{0}}\sqcup W|_{a_{1}})$ such that 
$$\partial w_{i} \; = \; x_{i} - y_{i}.$$ 
Since $y_i \in L_1$ for all $i$, we can similarly find classes $v_i \in H_{n+1}(W|_{[a_{1}, a_{2}]}, \; W|_{a_{1}}\sqcup W|_{a_{2}})$, with $\partial_{\text{in}}(v_i) = y_i$ and $\partial_{\text{out}}(v_i) = -z_i$ for some classes $z_i \in H_n(W|_{a_2})$ with $\iota_{in}(y_i) = \iota_{out}(z_i)$. Now consider the element $j_0(w_i) + j_1(v_i) \in H_{n+1}(W|_{[a_{0}, a_{2}]}, \; W|_{a_0} \sqcup W|_{a_{1}}\sqcup W|_{a_{2}})$ where 
$$j_\nu: H_{n+1}((W|_{[a_{\nu}, a_{\nu+1}]}, \; W|_{a_{\nu}}\sqcup W|_{a_{\nu+1}}) \longrightarrow H_{n+1}(W|_{[a_{0}, a_{2}]}, \; W|_{a_0} \sqcup W|_{a_{1}}\sqcup W|_{a_{2}})$$
is the obvious inclusion for $\nu = 0,1$. In the long exact sequence 
$$\xymatrix{ \cdots \ar[r] & H_{n+1}(W|_{[a_{0}, a_{2}]}, W|_{a_{0}}\sqcup W|_{a_{2}}) \ar[d] & \\
& H_{n+1}(W|_{[a_{0}, a_{2}]}, W|_{a_0} \sqcup W|_{a_{1}}\sqcup W|_{a_{2}}) \ar[d]_\partial & \\
& H_n(W|_{a_0} \sqcup W|_{a_{1}}\sqcup W|_{a_{2}}, W|_{a_{0}}\sqcup W|_{a_{2}}) \ar[r] & \cdots }$$
of the triple $(W|_{a_{0}}\sqcup W|_{a_{2}}, W|_{a_0} \sqcup W|_{a_{1}}\sqcup W|_{a_{2}}, W|_{[a_{0}, a_{2}]})$ it is clearly mapped to zero: 

\begin{align*}
\partial(j_0(w_i) + j_1(v_i)) &= \partial_{\text{in}}(w_i) + \partial_{\text{out}}(w_i) + \partial_{\text{in}}(v_i) + \partial_{\text{out}}(v_i) \\
&= c(x_i-y_i + y_i - z_i) \\
&= 0,
\end{align*}

\noindent where $c \colon H_n(W|_{a_0} \sqcup W|_{a_{1}}\sqcup W|_{a_{2}}) \rightarrow H_n(W|_{a_0} \sqcup W|_{a_{1}}\sqcup W|_{a_{2}}, W|_{a_0} \sqcup W|_{a_{2}})$ is induced by the inclusion, as $c(x_i-z_i)$ comes from $H_n(W|_{a_0} \sqcup W|_{a_2})$.	We can therefore pick preimages $u_i \in H_{n+1}(W|_{[a_{0}, a_{2}]}, W|_{a_{0}}\sqcup W|_{a_{2}})$ of $j_0(w_i) + j_1(v_i)$. These satisfy $\partial_{\text{out}} u_i = z_i$ and so we can repeat the process until we have constructed a subspace
$$V^0 \subseteq H_{n+1}(W|_{[a_0,a_p]}, W|_{a_0} \sqcup W|_{a_p}) \cong H_{n+1}^{\cpt}(W|_{(a_0-\varepsilon_0, a_p + \varepsilon_p)})$$
By construction and signed commutativity of the diagram
$$\xymatrix{& H^{\cpt}_{n+1}(W|_{(a_{0}-\varepsilon_{0}, a_{i}+\varepsilon_{i})}) \ar[dl]_{-|_{[a_0,a_{i}]}}^\cong \ar[dr]^{\cdot|_{a_{\nu}}} & \\
H_{n+1}(W|_{[a_{0}, a_{i}]}, \; W|_{a_{0}}\sqcup W|_{a_{i}}) \ar[r]^-{\partial} & H_{n}(W|_{a_{0}}\sqcup W|_{a_{i}}) \ar[r]^-{\text{pr}_\nu} & H_{n}(W|_{a_{\nu}})}$$
for $\nu = 0,i$ we find $V^0|_{a_0} = L_0$ and $V^0|_{a_i} \leq L_i$ for all $i \neq 0$. 

The other required subspaces $V^i$ are constructed in an entirely analogous fashion.
\end{proof}

\begin{lemma} \label{lemma: inclusion of zero simplices}
The map $|\mb{D}^{\mathcal{L}}_{\bullet}| = |{\mb{C}}^{\mathcal{L}}_{\bullet, 0}| \hookrightarrow |{\mb{C}}^{\mathcal{L}}_{\bullet, \bullet}|$ induced by inclusion of zero-simplices is a weak homotopy equivalence. 
\end{lemma}

\begin{proof}
The proof of this lemma follows the argument from \cite[Page 327]{GRW 14}. We spell it out as there is a small oversight in the final argument in \cite{GRW 14}, which we correct.
	
To begin, we resurrect the simplicial space $\widehat{\mb C}_\bullet^{\mathcal L}$ from Definition \ref{defn: convenient model for nerve}; recall that this meant that the inequalities $a_{i} + \varepsilon_{i} < a_{i+1} - \varepsilon_{i+1}$ are replaced by $a_{i} \leq a_{i+1}$. There is an evident augmented bi-semi-simplicial space $\widehat{\mb{C}}_{\bullet, \bullet} \longrightarrow \widehat{\mb{C}}_{\bullet, -1}$ defined similarly with $\widehat{\mb{C}}_{\bullet, -1} = \widehat{\mb{C}}_{\bullet}$. By Lemma \ref{lemma: equivalence of geometric realizations} it suffices to show that $|\widehat{\mb{C}}_{\bullet, 0}| \hookrightarrow  |\widehat{\mb{C}}_{\bullet, \bullet}|$ is a weak homotopy equivalence. 
	
We will define a retraction $r: |\widehat{\mb{C}}_{\bullet, \bullet}| \longrightarrow |\widehat{\mb{C}}_{\bullet, 0}|$ which is a weak homotopy equivalence. For each $p, q \in \Z_{\geq 0}$ there is a map
\[h_{p,q}: \widehat{\mb{C}}_{p, q} \; \longrightarrow \; \widehat{\mb{C}}_{(p+1)(q+1)-1, 0}\]
given by considering $p+1$ regular values, each equipped with $(q+1)$ collections of subspaces of $H^{\cpt}_{n+1}(W)$, as $(p+1)(q+1)$ not-necessarily distinct regular values each equipped with a single collection of subspaces of $H^{\cpt}_{n+1}(W)$. For example, in the case that $p = 1$ and $q = 2$, the map $\widehat{\mb{C}}_{1, 2} \; \longrightarrow \; \widehat{\mb{C}}_{5, 0}$ is given by sending 
$$\left((a_{0}, a_{1}), \; \; (W, \ell), \; \; (L_{0}, L_{1}), \; \; (V^{0}_{0}, V^{0}_{1}), (V^{1}_{0}, V^{1}_{1}), (V^{2}_{0}, V^{2}_{1})\right)$$
to the element
$$\left((a_{0}, a_{0}, a_{0}, a_{1}, a_{1}, a_{1}), \; \; (W, \ell), \; \; (L_{0}, L_{0}, L_{0}, L_{1}, L_{1}, L_{1}), \; \;  (V^{0}_{0}, V^{1}_{0}, V^{2}_{0},  V^{0}_{1}, V^{1}_{1}, V^{2}_{1})\right),$$
where we have dropped the data $\varepsilon = (\varepsilon_{0}, \varepsilon_{1})$ from the notation to save space. Being able to do this is the very reason for having distinct subspaces $V_i$ for every slice $W|_{a_i}$. There is also a map 
\[\rho_{p,q}: \Delta^{p}\times\Delta^{q} \longrightarrow \Delta^{(p+1)(q+1)-1} \; \subset \; \R^{(p+1)(q+1)}\]
with $(i + (q+1)j)$th coordinate given by $(t, s) \mapsto t_{j}s_{i}$. Taking the product of these maps yields 
\[r_{p,q}: \widehat{\mb{C}}_{p,q}\times\Delta^{p}\times\Delta^{q} \; \longrightarrow \; \widehat{\mb{C}}_{(p+1)(q+1)-1, 0}\times\Delta^{(p+1)(q+1)-1}\]
which glue together to give a map $r: |\widehat{\mb{C}}_{\bullet, \bullet}| \longrightarrow |\widehat{\mb{C}}_{\bullet, 0}|$. It is clear that this map is a retraction and thus the induced map on homotopy groups is surjective. Consider the augmentation map $|\varepsilon|: |\widehat{\mb{C}}_{\bullet, \bullet}| \longrightarrow |\widehat{\mb{C}}_{\bullet}|$ in the second bi-semi-simplicial coordinate. By Lemma \ref{lemma: augmentation equivalence}, this map is a weak homotopy equivalence. The fact that $r$ induces an injection on homotopy groups will follow once we prove that $|\varepsilon|\colon |\widehat{\mb{C}}_{\bullet, \bullet}| \longrightarrow |\widehat{\mb{C}}_{\bullet}|$ induces the same map on homotopy groups as
\[\xymatrix{|\widehat{\mb{C}}_{\bullet, \bullet}| \ar[r]^{r} & |\widehat{\mb{C}}_{\bullet, 0}|  \ar[r]^{|\varepsilon_{0}|} & |\widehat{\mb{C}}_{\bullet}|.}\]
In \cite{GRW 14} it is claimed that the (analoguous) two maps are in fact equal, but this is not true: Chasing an element through the composition shows only that it is a degeneracy of its image under $|\varepsilon|$, as its entries have been repeated as indicated in the example above. 	It can, however, be checked by hand that the diagram 
\begin{equation} \label{equation: degeneracy commutative diagram}
\xymatrix{|\widehat{\mb{C}}_{\bullet, \bullet}| \ar[dr]^{|\varepsilon|} \ar[r]^{r} & |\widehat{\mb{C}}_{\bullet, 0}| \ar[r]^{|\varepsilon_{0}|} & |\widehat{\mb{C}}_{\bullet}| \ar[d]^{p}   \\
& |\widehat{\mb{C}}_{\bullet}| \ar[r]^{p}  &  |\widehat{\mb{C}}_{\bullet}|^{\mathrm{th}}}
\end{equation}
is commutative, where $|-|^{\mathrm{th}}$ denotes the thin realization collapsing degenerate simplices. From Lemma \ref{lemma: equivalence of geometric realizations} we find that $p$ induces an isomorphism in homotopy groups which completes the argument.
\end{proof}

Again the goodness of $\widehat{\mb{C}}_{\bullet}$ can be avoided by instead noting that the two maps $|\varepsilon|$ and $|\varepsilon| \circ r$ are homotopic through a straight line homotopy (through the degenerate simplices by whose appearance they differ). With the two lemmas above it follows that the maps $|{\mb{C}}^{\mathcal{L}}_{\bullet, 0}| \hookrightarrow |{\mb{C}}^{\mathcal{L}}_{\bullet, \bullet}| \longrightarrow |{\mb{C}}^{\mathcal{L}}_{\bullet}|$ are weak homotopy equivalences. Identifying $|\mb{D}^{\mathcal{L}}_{\bullet}| = |{\mb{C}}^{\mathcal{L}}_{\bullet, 0}|$, establishes the weak homotopy equivalence $|\mb{D}^{\mathcal{L}}_{\bullet}| \simeq |{\mb{C}}^{\mathcal{L}}_{\bullet}|$
and completes the proof of Theorem \ref{theorem: replacement of nerve}.

\subsection{Flexible models}\label{flexmod} 
The following definition builds on \cite[Definition 2.8]{GRW 14}. We do, however, need a slightly more general version than provided by loc.cit. for the proof of Theorem \ref{theorem: inclusion of disk} at the end of this section.

\begin{defn} \label{defn: more flexible model}
Define $\mb{X}^{\mathcal{L}}_{\bullet}$ to be the semi-simplicial space with $p$-simplices consisting of certain tuples $(a, \varepsilon, (W, \ell), (V_{0}, \dots, V_{p}))$ with $a \in \R^{p+1}$, $\varepsilon \in \R^{p+1}_{>0}$, and
$$(W, \ell, V_{i}) \in \bPsi^{\Delta}_{\theta}((a_{0}- \varepsilon_{0}, a_{p}+\varepsilon_{p})\times\R^{\infty}) \quad \text{for all $i = 0, \dots, p$,}$$
subject to the following conditions:
\begin{enumerate} 
	\item[(i)] $W$ is contained in $(a_{0}-\varepsilon_{0}, a_{p}+\varepsilon_{p})\times(-1, 1)^{\infty}$;
	\item[(ii)] $a_{i-1} + \varepsilon_{i-1} < a_{i} - \varepsilon_{i}$ for $i = 1, \dots, p$;
	\item[(iii)] For any two regular values $b < c \in \cup_{i=0}^{p}(a_{0}-\varepsilon_{0}, a_{p}+\varepsilon_{p})$ of the height function $W \longrightarrow \R$, the pair $(W|_{[b, c]}, W|_{c})$ is $(n-1)$-connected. 
	\item[(iv)] 	Let $i = 0, \dots, p$. If $c \in (a_{i}-\varepsilon_{i}, a_{i}+\varepsilon_{i})$ is a regular value for the height function $W \longrightarrow \R$, then $V_{i}|_{c} \leq H_{n}(W|_{c})$ is a Lagrangian subspace and $V_{j}|_{c} \leq V_{i}|_{c}$ for all other $j = 0, \dots, p$.
\end{enumerate}
\end{defn}

The space $\mb{X}^{\mathcal{L}}_{p}$ is topologized in the same way as the space $\mb{D}^{\mathcal{L}}_{p}$ and possesses evident face maps $d_{i}: \mb{X}^{\mathcal{L}}_{p} \longrightarrow \mb{X}^{\mathcal{L}}_{p-1}$ for $0 \leq i \leq p$, by forgetting the $i$-th piece of data (and appropriately restricing $V$ and $W$ for $i=0,p$). Proposition \ref{conti} implies that these maps are continuous and we obtain a semi-simplicial space $\mb{X}^{\mathcal{L}}_{\bullet}$. Notice that the principal difference between $\mb{X}^{\mathcal{L}}_{\bullet}$ and $\mb{D}^{\mathcal{L}}_{\bullet}$ is that for any $(a, \varepsilon, (W, \ell), V) \in \mb{X}^{\mathcal{L}}_{p}$ the manifold $W$ is not required to be cylindrical over the intervals $(a_{i}-\varepsilon_{i}, a_{i}+\varepsilon_{i})$. Furthermore, these intervals need not even be comprised entirely of regular values for the height function. The formula
\begin{equation} \label{equation: D to X comparison} 
(a, \varepsilon, W, \ell, V) \; \mapsto \; \left(a,\; \varepsilon, \; W|_{(a_{0}-\varepsilon_{0}, a_{p}+\varepsilon_{p})}, \ell|_{(a_{0}-\varepsilon_{0}, a_{p}+\varepsilon_{p})}, \; V\right)
\end{equation}
induces a semi-simplicial map $\mb{D}^{\mathcal{L}}_{\bullet} \longrightarrow \mb{X}^{\mathcal{L}}_{\bullet}$ which is continuous by Proposition \ref{conti}.

\begin{defn} \label{defn: sub flexible models}
We define a sequence of sub-semi-simplicial spaces of $\mb{X}^{\mathcal{L}}_{\bullet}$  (compare (\ref{equation: sequence of D subspaces})):
\begin{enumerate} 
	\item[(a)] $\mb{X}^{\mathcal{L}, D}_{\bullet} \subset \mb{X}^{\mathcal{L}}_{\bullet}$ has as its $p$-simplices those $(a, \varepsilon, (W, \ell), V)$ such that $W$ contains 
	$$(a_{0}-\varepsilon_{0}, a_{p}+\varepsilon_{p})\times D$$ 
	and such that the restriction of $\ell$ to $(a_{0}-\varepsilon_{0}, a_{p}+\varepsilon_{p})\times D$ agrees with the structure $\ell_{\R\times D}$ used in the definition of $\Cob^{D}_{\theta}$.
	\item[(b)] Let $l \in \Z_{\geq -1}$. $\mb{X}^{\mathcal{L}, l}_{\bullet} \subset \mb{X}^{\mathcal{L}, D}_{\bullet}$ has as its $p$-simplices those $(a, \varepsilon, (W, \ell), V)$ with the property that for any regular value $c \in \cup_{i=0}^{p}(a_{i} - \varepsilon_{i}, a_{i} + \varepsilon_{i})$ of the height function, the manifold $W|_{c}$ is $l$-connected. 
\end{enumerate}
\end{defn}

\begin{proposition} \label{proposition: equivalence to flexible model}
	The map from  (\ref{equation: D to X comparison}) induces a weak homotopy equivalence,
	$|\mb{D}^{\mathcal{L}}_{\bullet}| \simeq |\mb{X}^{\mathcal{L}}_{\bullet}|.$
	Similarly, it induces weak equivalences 
	$|\mb{D}^{\mathcal{L}, D}_{\bullet}| \simeq |\mb{X}^{\mathcal{L}, D}_{\bullet}|$ and $|\mb{D}^{\mathcal{L}, l}_{\bullet}| \simeq |\mb{X}^{\mathcal{L}, l}_{\bullet}|.$
\end{proposition}

\begin{proof}[Proof sketch]
Let $\mb{D}^{\mathcal{L}, \pitchfork}_{\bullet}$ be the semi-simplicial space with $p$-simplices given by tuples $(a, \varepsilon, (W, \ell), V)$ as in the definition of $\mb{D}^{\mathcal{L}}_{\bullet}$, but instead of requiring $W$ to be cylindrical over the intervals $(a_{i}-\varepsilon_{i}, a_{i}+\varepsilon_{i})$, we only require $W$ to be \textit{regular} over the intervals $(a_{i}-\varepsilon_{i}, a_{i}+\varepsilon_{i})$. By this we mean that we require each interval $(a_{i}-\varepsilon_{i}, a_{i}+\varepsilon_{i})$ to consist entirely of regular values for the height function on $W$. There is an embedding $\mb{D}^{\mathcal{L}}_{\bullet} \hookrightarrow \mb{D}^{\mathcal{L}, \pitchfork}_{\bullet}$ and the map from the statement of the proposition factors as the composite,
$$\xymatrix{|\mb{D}^{\mathcal{L}}_{\bullet}| \ar[r]^{(1)} & |\mb{D}^{\mathcal{L}, \pitchfork}_{\bullet}| \ar[r]^{(2)} & |\mb{X}^{\mathcal{L}}_{\bullet}|.}$$
The proposition follows from the fact that both maps in this composition are weak homotopy equivalences. The proof that map $(1)$ is a weak homotopy equivalence proceeds exactly as \cite[Theorem 3.9]{GRW 09}, while the proof that map $(2)$ is a weak homotopy equivalence goes through in the same way as \cite[Proposition 2.20]{GRW 14}. 
\end{proof}

\subsection{Proof of Theorem \ref{theorem: inclusion of disk}}\label{section: Proof of disk inclusion}.
Finally, we give the proof of Theorem \ref{theorem: inclusion of disk}, which asserts the weak homotopy equivalence $B\Cob^{\mathcal{L}, D}_{\theta} \simeq B\Cob^{\mathcal{L}}_{\theta}$. 

\begin{proof}[Proof of Theorem \ref{theorem: inclusion of disk}]
By the results above it will suffice to prove that the inclusion of semi-simplicial spaces $\mb{X}^{\mathcal{L}, D}_{\bullet} \hookrightarrow \mb{X}^{\mathcal{L}}_{\bullet}$ induces a weak homotopy equivalence on geometric realization. This is proven in essentially the same way as \cite[Proposition 2.16]{GRW 14}. In loc.cit.\ the proof is carried out at the level of spaces of manifolds by construction of a homotopy inverse $r \colon \psi_\theta(\infty,1) \rightarrow \psi_{\theta,D}(\infty,1)$ to the inclusion. The map $r$ squeezes a given manifold $W \subseteq \mathbb R \times (-1,1)^\infty$ into $\mathbb R \times (1/2,1) \times (-1,1)^\infty$ and then adds the subset $\mathbb R \times S \subset \mathbb R \times (-1/2,1/2) \times (-1,1)^\infty$, where $S$ is a sphere with $S \cap (-1/2,0] \times (-1,1)^\infty = D$. Clearly this procedure provides maps
\[r_p \colon \mb{X}^{\mathcal{L}}_{\bullet} \rightarrow \mb{X}^{\mathcal{L}, D}_{\bullet}\]
as well (the Lagrangian being transported in the evident manner) and the homotopies decribed by Galatius and Randal-Williams show $r_p$ is a levelwise homotopy inverse to the inclusion. 
\end{proof}

The same argument would not work at the level of $\mb C$- or $\mb D$-spaces as the homotopies move critical points through the sets $(a_i-\epsilon_i,a_i+\epsilon_i)$ and we do not have analogues of the $\psi$-spaces.


\section{Surgery on Manifolds Equipped with a Lagrangian Subspace} \label{section: preliminaries on surgery}
In this section we develop some technical results about surgery on manifolds that will be needed in Sections \ref{section: surgery on objects below the middle dimension} and \ref{section: surgery on objects in the middle dimension}. Without loss of continuity, the reader can skip this section and come back to it when its results are used.

\subsection{Transport of Lagrangians}
For what follows, let $M$ be a closed, $2n$-dimensional, oriented manifold and let $L \leq H_{n}(M)$ be a Lagrangian subspace with respect to the intersection form $(H_{n}(M), \lambda, \mu)$. Let 
\begin{equation} \label{equation: phi surgery embedding}
\phi: S^{k}\times D^{2n-k} \longrightarrow M
\end{equation}
be an embedding. Let $\widetilde{M}$ denote the manifold obtained by performing surgery on $M$ along the embedding $\phi$ and $M'$ the complement $M\setminus\phi(S^{k}\times\Int(D^{2n-k}))$. Let finally 
$$\xymatrix{H_{n}(M) & H_{n}(M') \ar[l]_{\alpha}  \ar[r]^{\beta} & H_{n}(\widetilde{M})}$$
denote the maps induced by inclusion. Consider the subspaces,
\[L' := \alpha^{-1}(L) \leq H_{n}(M') \quad \text{and} \quad \widetilde{L} := \beta(\alpha^{-1}(L)) \leq H_{n}(\widetilde{M}).\]
We will use this same notation throughout the rest of the section, which is devoted to proving the following:

\begin{theorem} \label{theorem: transport of lagrangian through surgery}
	Let $\phi: S^{k}\times D^{2n-k} \longrightarrow M$ be an (orientation preserving) embedding as in (\ref{equation: phi surgery embedding}) and suppose that $k < n$. 
	Then the subspace $\widetilde{L} \leq H_{n}(\widetilde{M})$ is a Lagrangian subspace with respect to $(H_{n}(\widetilde{M}), \lambda, \mu)$.
\end{theorem}

The reader willing to believe this result may well skip the rest of this section. The verification is a lengthy but routine homology calculation. We begin by proving the special case of this theorem when $k < n-1$. This proof of this special case is significantly easier than the $k = n-1$ case. 

\begin{proof}[Proof of Theorem \ref{theorem: transport of lagrangian through surgery} for $k < n-1$]
By excision we have isomorphisms
\begin{align*}
H_{*}(M, M') &\cong H_{*}(S^{k}\times D^{2n-k}, \; S^{k}\times S^{2n-k-1}), \\
H_{*}(\widetilde{M}, M') &\cong H_{*}(D^{k+1}\times S^{2n-k-1}, \; S^{k}\times S^{2n-k-1}),
\end{align*}
and since $k \leq n-2$ it follows that 
$$0 \; = \; H_{n}(M, M') \; = \; H_{n+1}(M, M') \; = \; H_{n}(\widetilde{M}, M') \; = \; H_{n+1}(\widetilde{M}, M').$$
From the long exact sequence associated to the pairs $(M, M')$ and $(\widetilde{M}, M')$ it follows that the maps
$$\xymatrix{H_{n}(M) & H_{n}(M') \ar[l]_{\alpha}  \ar[r]^{\beta} & H_{n}(\widetilde{M})}$$
are both isomorphisms. The maps $\alpha$ and $\beta$ are induced by codimension-$0$ embeddings and thus they preserve both the intersection pairing $\lambda$ and its refinement $\mu$. It follows that $\widetilde{L} = \beta(\alpha^{-1}(L)) \leq H_{n}(\widetilde{M})$ is a Lagrangian subspace.
\end{proof}

\subsection{Proof of Theorem \ref{theorem: transport of lagrangian through surgery} for $k = n-1$}
We now focus on proving Theorem \ref{theorem: transport of lagrangian through surgery} in the harder case that $k = n-1$. Let $\phi: S^{k}\times D^{2n-k} \longrightarrow M$ be as in (\ref{equation: phi surgery embedding}) and let $x \in H_{k}(M)$ be the class determined by $\phi|_{S^{k}\times\{0\}}: S^{k} \longrightarrow M$. The following is \cite[Lemma 5.6]{KM 63}. 

\begin{proposition} \label{lemma: intersection inclusion}
	Let
	$j: H_{2n-k}(M) \longrightarrow H_{2n-k}(M, M')$ be the map induced by inclusion and let
	$$\alpha_{2n-k} \in H_{2n-k}(M, M') \cong \Z$$ 
	be the generator induced by the orientation on $(D^{2n-k},S^{2n-k-1})$.
	The map $j$ is given by the formula 
	$$
	j(y) =  \lambda(x, y)\cdot\alpha_{2n-k} 
	$$
	for all $y \in H_{2n-k}(M)$,
	where $\lambda: H_{k}(M)\otimes H_{2n-k}(M) \longrightarrow \Z$ is the intersection pairing.
\end{proposition}

\begin{corollary} \label{corollary: disjunction corollary}
Let $y \in H_{2n-k}(M)$ be a class such that $\lambda(x, y) = 0$. Then the class $y$ is in the image of the map $H_{2n-k}(M') \longrightarrow H_{2n-k}(M)$ induced by inclusion.
\end{corollary}

\begin{proof}
Immediate from the long exact sequence associated to the pair $(M, M')$.
\end{proof}

Now consider the situation of Theorem \ref{theorem: transport of lagrangian through surgery} again.

\begin{lemma} \label{claim: L' lagrangian}
The subspace $L' = \alpha^{-1}(L) \leq H_{n}(M')$ is a Lagrangian.
\end{lemma}

\begin{proof}
Let us start with the inclusion $(L')^\perp \leq L'$.	Let $v \in (L')^{\perp}$ and let $w \in L$. By surjectivity of $\alpha: H_{n}(M') \longrightarrow H_{n}(M)$, we choose $w' \in L' = \alpha^{-1}(L)$ such that $\alpha(w') = w$. Since $v \in (L')^{\perp}$ we have 
$$0 = \lambda(v, w') = \lambda(\alpha(v), w).$$
Since $w$ was arbitrary $\alpha(v) \in L^{\perp}$ and since $L$ is a Lagrangian it follows that $\alpha(v) \in L$ and so $v \in L' = \alpha^{-1}(L)$. This proves $(L')^{\perp} \leq L'$.
	
For the other inclusion suppose that $v, w \in L'$. Since $\alpha$ preserves the intersection pairing we have $\lambda(v, w) = \lambda(\alpha(v), \alpha(w))$.  Since $\alpha(v), \alpha(w) \in L$ and $L$ is Lagrangian it follows that $\lambda(v, w) = \lambda(\alpha(v), \alpha(w)) = 0$. This proves that $L'$ is isotropic. The same argument shows that $\mu$ vanishes on $L'$.
\end{proof}

\begin{proof}[Proof of Theorem \ref{theorem: transport of lagrangian through surgery} for $k = n-1$]
Let $x \in H_{n-1}(M)$ denote the class represented by the embedding $\phi|_{S^{n-1}\times\{0\}}: S^{n-1} \longrightarrow M$. Let $x' \in H_{n-1}(M')$ be the unique class that maps to $x$ under the map $H_{n-1}(M') \longrightarrow H_{n-1}(M)$ induced by inclusion, which is an isomorphism by the long exact sequence associated to the pair $(M, M')$. The proof breaks down into two cases: the case where $x$ is of infinite order and the case where $x$ is of finite order. 

\textit{Case 1:} Suppose that the class $x \in H_{n-1}(M)$ has infinite order. By Lemma \ref{claim: L' lagrangian} it will suffice to prove that the map $\beta: H_{n}(M') \longrightarrow H_{n}(\widetilde{M})$ is an isomorphism. Since $x$ has infinite order it follows that $x' \in H_{n-1}(M')$ has infinite order as well. Since the boundary map $H_n(\widetilde{M},M') \rightarrow H_{n-1}(M')$ of the long exact sequence for the pair $(\widetilde M, M')$ sends a generator to $x'$ it is injective. It follows that $\beta: H_{n}(M') \longrightarrow H_{n}(\widetilde{M})$ is surjective. Since $H_{n+1}(\widetilde{M}, M') = 0$, it follows $\beta$ is injective as well and thus an isomorphism. 

\textit{Case 2:} Suppose that $x$ is of order $m < \infty$. It follows that the class $x' \in H_{n-1}(M')$ (that maps to $x$) has order $m < \infty$ as well. As before $x'$ generates the image of $H_n(\widetilde{M},M') \rightarrow H_{n-1}(M')$ and using the same exact sequence as before we obtain
\begin{equation} \label{equation: beta exact sequence}
\xymatrix{0 \ar[r] & H_{n}(M') \ar[r]^-{\beta} & H_{n}(\widetilde{M}) \ar[r] & \kernel(\partial) \cong m\cdot \Z \ar[r] & 0.}
\end{equation}
Let now $\alpha_{n+1} \in H_{n+1}(M, M') \cong \Z$ denote the generator from Lemma \ref{lemma: intersection inclusion} and let $y' \in H_{n}(M')$ denote the class $\partial(\alpha_{n+1})$, where $\partial: H_{n+1}(M, M') \longrightarrow H_{n}(M')$ is the boundary map (it is represented by $\phi\colon\{0\}\times S^{n+1} \rightarrow M'$). We will need to use the following basic property about $y'$, whose proof we postpone until after the proof of the current proposition. 

\begin{uclaim} \label{claim: basic properties of y}
The class $y'$ has infinite order. Furthermore, $y' \in L'$ and $\lambda(y', v) = 0$ for all $v \in H_{n}(M')$.
\end{uclaim}

Let $\tilde{y} = \beta(y')$ for $\beta: H_{n}(M') \longrightarrow H_{n}(\widetilde{M})$. Since $\beta$ is injective it follows that $\tilde{y}$ has infinite order. Moreover, it follows that $\tilde{y} \in \widetilde{L} = \beta(L')$ by the claim. We make one more observation about the class $\tilde{y}$: $\langle \tilde{y} \rangle^\perp = \text{im}(\beta)$.

Indeed, the map $\lambda(\tilde y, \cdot): H_n(\widetilde M) \rightarrow \mathbb Z$ annihilates the image of $\beta$ by the claim above (giving one inclusion) and therefore factors over $\ker(\partial) \cong m\cdot\mathbb Z$ by exactness of (\ref{equation: beta exact sequence}). Also, it cannot be the null map, as the intersection pairing on $H_n(\widetilde M)$ is non-degenerate. As a non-zero homomorphism from one infinite cyclic group to another it is injective. These two facts imply that $\lambda(\tilde y, v) = 0$ if and only if the image of $v$ under $H_{n}(\widetilde{M}) \longrightarrow \ker(\partial) \subset H_{n}(\widetilde{M}, M')$ is equal to zero, which gives the other inclusion.
	
We are finally in a position to show that $\widetilde{L}$ is a Lagrangian subspace. Let $w \in \widetilde{L}^{\perp}$. Since $\tilde{y} \in \widetilde{L}$, we have $\lambda(\tilde y, w) = 0$ and thus $w = \beta(w')$ for some $w' \in H_n(M')$. Since $\beta$ preserves the intersection pairing it follows that $w' \in (L')^{\perp}$. By Lemma \ref{claim: L' lagrangian}, $L'$ is a Lagrangian subspace, so $w' \in L'$ which yields $w \in \widetilde L$.
	
This proves that $\widetilde{L}^{\perp} \leq \widetilde{L}$. Since $\widetilde{L}$ is by definition equal to $\beta(L')$, $\beta$ preserves the intersection pairing, and $L'$ is an isotropic subspace (i.e. $L' \leq (L')^{\perp}$), it follows that $\widetilde{L}$ is an isotropic subspace as well, so indeed $\widetilde{L}^{\perp} = \widetilde{L}$. The fact that $\mu$ vanishes on $\widetilde L$ also follows by the self-intersection form being preserved by $\beta$.
\end{proof}

It remains to verify the claim.

\begin{proof}[Proof of the Claim]
We begin by showing that the class $y' = \partial(\alpha_{n+1}) \in H_{n}(M')$ has infinite order. By assumption, the class $x \in H_{n-1}(M)$ has finite order. It follows that $\lambda(x, v) = 0$ for all $v \in H_{n+1}(M)$. It then follows from Lemma \ref{lemma: intersection inclusion} that the map $H_{n+1}(M) \longrightarrow H_{n+1}(M, M')$ is the zero map. By exactness the boundary map 
$$\partial: H_{n+1}(M, M') \longrightarrow H_{n}(M')$$
is then injective. Since $y' = \partial(\alpha_{n+1})$ (where $\alpha_{n+1} \in H_{n+1}(M, M') \cong \Z$ is the generator) it follows that $y'$ has infinite order. 
	
Since $y'$ is in the image of the boundary map $\partial$, it follows by exactness that $y'$ is in the kernel of $\alpha: H_{n}(M') \longrightarrow H_{n}(M)$. It follows from this that $y' \in \alpha^{-1}(L) = L'$, since $\alpha^{-1}(L)$ contains the kernel of $\alpha$. This establishes the third assertion of Claim \ref{claim: basic properties of y}. Let $v \in H_{n}(M')$. We have 
$$\lambda(v, y') = \lambda(\alpha(v), \alpha(y')) = \lambda(\alpha(v'), 0) = 0.$$
This proves that $\lambda(v, y') = 0$ for all $v \in H_{n}(M')$.
\end{proof}


\section{Surgery on Objects Below the Middle Dimension} \label{section: surgery on objects below the middle dimension}
Let $l \in \Z_{\geq -1}$. We proceed to prove Theorem \ref{theorem: surgery on objects below middle dim} which asserts that there is a weak homotopy equivalence $B\Cob^{\mathcal{L}, l-1}_{\theta} \simeq B\Cob^{\mathcal{L}, l}_{\theta}$ whenever $l \leq n-1$ and the tangential structure $\theta: B \longrightarrow BO(2n+1)\langle n \rangle$ is such that $B$ is $l$-connected and of type $F_{l+1}$. By Theorem \ref{theorem: replacement of nerve}, it will suffice to prove the weak homotopy equivalence $|\mb{D}^{\mathcal{L}, l-1}_{\bullet}| \simeq |\mb{D}^{\mathcal{L}, l}_{\bullet}|$. The proof will closely follow \cite[Section 4]{GRW 14}, so closely in fact, that we shall forego spelling out the construction of the surgery moves and instead ask the reader to have his copy of \cite{GRW 14} at the ready. In particular, we will reuse the constructions and notation of \cite{GRW 14} and only indicate the differences and extra steps that have to be taken. 

To give an outline, one considers a bi-semi-simplicial resolution $|\mb{D}^{\mathcal{L}, l}_{\bullet, \bullet}| \rightarrow |\mb{D}^{\mathcal{L}, l-1}_{\bullet}|$, in which a $(p,q)$-simplex consits of an element of $\mb{D}^{\mathcal{L}, l-1}_{p}$ together with $q+1$ disjoint pieces of surgery data, together with a `perform surgery map' $|\mb{D}^{\mathcal{L}, l}_{\bullet, \bullet}| \rightarrow |\mb{D}^{\mathcal{L}, l}_{\bullet}|$ and shows that both of these are weak equivalences. In truth, just as in \cite{GRW 14}, the second map is, however, only defined with source and target replaced by weakly equivalent spaces and the proofs for the maps being weak equivalences are intertwined with these auxillary spaces as well.

\subsection{A semi-simplicial resolution} \label{subsection: resolution n-1}
To conform with the notation of \cite[Section 4]{GRW 14}, set $d = 2n+1, \kappa = n-1, N = \infty$, $L = D$ and fix a positive integer $l < n$. Then the semi-simplicial space $\mb D^{\mathcal L,l}_\bullet$ of Definition \ref{defn: basic lagrangian cob cat} agrees with the nerve of the topological poset $D_{\theta,L}^{\kappa,l}$ of \cite[Section 2.6]{GRW 14}, except for the appearence of Lagrangians (of course) and our requirement that cobordisms be cylindrical in the $\epsilon$-neighbourhood of the boundary whereas Galatius and Randal-Williams only require the projection onto the first coordinate to not have critical points in the same $\epsilon$-neighbourhood. The second difference will play no role throughout the rest of this paragraph as it is preserved by all constructions to come.

For $x = (a, \varepsilon, (W, \ell_{W}), V) \in \mb{D}^{\mathcal{L}, l-1}_{p}$, put $x_u = (a, \varepsilon, (W, \ell_{W}))$ and consider the semi-simplicial space $Y^{l}_{q}(x_u)$ of surgery data from \cite[Definition 4.3]{GRW 14}. In correspondence with \cite[Definition 4.4]{GRW 14}, set
\[\mb{D}^{\mathcal{L}, l}_{p, q} = \{(x,y) \mid x \in \mb{D}^{\mathcal{L}, l-1}_{p}, y \in Y^{l}_{q}(x_u)\}.\]
Forgetting the surgery data produces an augmentation $\mb{D}^{\mathcal{L}, l}_{p, q} \rightarrow \mb{D}^{\mathcal{L}, l-1}_{p}$. Let us emphasize that the subspaces $(V_0, \dots, V_p)$ associated to an element $x \in \mb{D}^{\mathcal{L},l-1}_p$ play no role in the resolution. Therefore, the verification of the following result corresponds to that of its counterpart \cite[Theorem 4.5]{GRW 14} (given in \cite[Section 6]{GRW 14}) verbatim.

\begin{theorem} \label{theorem: contractibility of space of surgery data (n-1)-conn}
Let $l \leq n-1$ and 
suppose that $\theta: B \longrightarrow BO(2n+1)$ is such that $B$ is $l$-connected and of type $F_{l+1}$. 
Then there are weak homotopy equivalences
$$|\mb{D}^{\mathcal{L}, l}_{\bullet, 0}| \stackrel{\simeq} \longrightarrow |\mb{D}^{\mathcal{L}, l}_{\bullet, \bullet}| \stackrel{\simeq} \longrightarrow  |\mb{D}^{\mathcal{L}, l-1}_{\bullet}|,$$
where the first map is induced by inclusion of zero-simplices and the second is induced by the augmentation.
\end{theorem}

\subsection{A surgery move} \label{subsection: surgery move below middle dimension}

To implement the surgery we resurrect the homotopy 
\[\mathcal S \colon [0,1] \times |D^{\kappa, l}_{\theta,L}(\mathbb R^N)_{\bullet, 0}| \longrightarrow |X_\bullet^{\kappa,l-1}|\]
from \cite[Lemma 4.7]{GRW 14}, starting at the forgetful map (followed by the inclusion $D^{\kappa,l-1}_{\theta,L} \subseteq X^{\kappa,l-1})$ and ending at a map that factors through the inclusion $X_\bullet^{\kappa,l} \subseteq X_\bullet^{\kappa,l-1}$. We would like to extend it to a commutative diagram
\[\xymatrix{[0,1] \times |\mb{D}^{\mathcal{L}, l}_{\bullet, 0}| \ar[r]^-{\mathcal F} \ar[d]& |\mb{X}^{\mathcal{L}, l-1}_{\bullet}| \ar[d] \\
[0,1] \times |D^{\kappa, l}_{\theta,L}(\mathbb R^N)_{\bullet, 0}| \ar[r]^-{\mathcal S} & |X_\bullet^{\kappa,l-1}|}\]
with the corresponding properties (and forgetful vertical maps). 

Comparing the definitions of $\mb{X}^{\mathcal{L}, l-1}_{\bullet}$ and $X_\bullet^{\kappa,l-1}$ all that remains is to produce the homological data on the underlying manifolds given by $\mathcal S$. To do so recall that $\mathcal S$ is glued from maps
\[\mathcal S_p \colon [0,1]^{p+1} \times D^{\kappa,l}_{\theta,L}(\mathbb R^N)_{p,0} \longrightarrow X_p^{\kappa,l-1}\]
given by 
\[(t,(a,\epsilon,(W,\ell_W)),(e,\ell)) \longmapsto (a,\epsilon/2,\mathcal K^t_{e,\ell}(W,\ell_W))\]
(after suppressing tangential structures) for $(e,\ell) \in Y^{l}_{0}(a,\epsilon,W)$ and the family 
\[\mathcal K^t_{e,\ell}(W,\ell_W) \in \Psi_\theta\big((a_0-\epsilon_0,a_p,\epsilon_p) \times \mathbb R^N\big)\]
from \cite[Lemma 4.6]{GRW 14}.
It therefore suffices to lift these maps $\mathcal S_p$ to maps 
\[\mathcal F_p \colon [0,1]^{p+1} \times \mb{D}^{\mathcal L,l}_{p,0} \longrightarrow \mb{X}_p^{\mathcal L,l-1}.\]
To do so we need to describe how to transport the subspaces $V_{0}, \dots, V_{p} \leq H_{n+1}^{\cpt}(W|_{(a_{0}- \varepsilon_{0}, a_{p}+\varepsilon_{p})})$ over to the homology group $H_{n+1}^{\cpt}(\mathcal{K}^{t}_{e_{i}, \ell_{i}}(W, \ell_{W}))$ for every constituent $(e_i,\ell_i)$ of $(e,\ell)$. Let $W'$ denote the complement $W|_{(a_{0}- \varepsilon_{0}, a_{p}+\varepsilon_{p})}\setminus\Int(\Image(e_{i}))$. 
For each $t \in [0, 1]$, let
$$\xymatrix{H^{\cpt}_{n+1}(W|_{(a_{0}- \varepsilon_{0}, a_{p}+\varepsilon_{p})})  && H_{n+1}^{\cpt}(W') \ar[ll]_-{\alpha} \ar[rr]^-{\beta^{t}} && H^{\cpt}_{n+1}(\mathcal{K}^{t}_{e_{i}, \ell_{i}}(W, \ell_{W}))}$$
denote the maps induced by inclusion. The inclusions of $W'$ are both proper maps and so the homomorphisms $\alpha$ and $\beta_{t}$ are indeed well-defined.

For $t \in [0, 1]$ and $j = \{0, \dots, p\}$ define $V^{t}_{j}$ by
\begin{equation} \label{equation: induced subspaces (n-1)-surgery}
\beta^{t}(\alpha^{-1}(V_{j}|_{(a_{0}-\varepsilon_{0}, a_{p}+\varepsilon_{p})})) = V^{t}_{j} \; \leq \; H^{\cpt}_{n+1}(\mathcal{K}^{t}_{e_{i}, \ell_{i}}(W, \ell_{W})).
\end{equation}

\begin{proposition} \label{proposition: continuity of Kappa n-1}
The above construction defines a (continuous) map
\[[0,1] \times \mb D_{p,0}^{\mathcal L,l} \longrightarrow \bPsi_{\theta}^{\Delta}((a_{0}- \varepsilon_{0}, a_{p}+\varepsilon_{p})\times\R^{\infty})\]
\[(t,(W,\ell_W),V,(e,\ell)) \longmapsto (\mathcal{K}^{t}_{e_{i}, \ell_{i}}(W, \ell_W), V^{t}_{j}),\]
with initial value given by
$$(\mathcal{K}^{0}_{e_{i}, \ell_{i}}(W, \ell_{W}), V^{0}_{j}) =  \left(W|_{(a_{0}-\varepsilon_{0}, a_{p}+\varepsilon_{p})}, \ell_W|_{(a_{0}-\varepsilon_{0}, a_{p}+\varepsilon_{p})}, V_{j}|_{(a_{0}-\varepsilon_{0}, \; a_{p}+\varepsilon_{p})}\right).$$
\end{proposition}

For the verification we need:

\begin{lemma}\label{alpha surjective}
The map $\alpha: H^{\cpt}_{n+1}(W') \longrightarrow H^{\cpt}_{n+1}(W|_{[a_0-\varepsilon_0, a_p+\varepsilon_p]})$ is an isomorphism in the case that $l < n-1$, and is surjective in the case that $l = n-1$.
\end{lemma}

\begin{proof}
Consider the exact sequence on $H^{\cpt}_{*}$ associated to the pair $(W, W')$: 
$$\xymatrix{\cdots \ar[r] & H^{\cpt}_{n+2}(W, W') \ar[r]^{\partial} & H^{\cpt}_{n+1}(W') \ar[r]^{\alpha} &  H^{\cpt}_{n+1}(W) \ar[r] & H^{\cpt}_{n+1}(W, W') \ar[r] &  \cdots}$$
Let $P_{i}$ and $P^{\partial}_{i}$ denote the manifolds
\begin{align*}
P_{i} & =  \Lambda_{i}\times(a_{i}-\varepsilon_{i}, a_{p}+\varepsilon_{p})\times S^{l}\times\R^{2n-l}, \\
 P^{\partial}_{i} & = \Lambda_{i}\times(a_{i}-\varepsilon_{i}, a_{p}+\varepsilon_{p})\times S^{l}\times(\R^{2n-l}\setminus \Int(D^{2n-l})).
 \end{align*} 
A simple calculation gives
\[H^{\cpt}_{k}(P_{i}, \; P^{\partial}_{i}) = 0  \quad \text{for all $k < 2n-l+1$.}\]
and via the embedding $e_i$ we have
\[H^{\cpt}_{k}(W, W') \; \cong \; H^{\cpt}_{k}(P_{i}, \; P^{\partial}_{i}) \quad \text{for all $k$.}\]
by excision. Using this equation in the exact sequence gives the claim.
\end{proof}

\begin{proof}[Proof of Proposition \ref{proposition: continuity of Kappa n-1}]
The claim of continuity follows immediately from the family $\mathcal K^t$ being locally generated by vector fields (see Definition \ref{defn: locally generated by vector fields}), which is readily checked from the construction (and is in fact needed to check that the underlying family of manifolds is continuous in \cite{GRW 14}) and Proposition \ref{proposition: continuity of the family}. The manifold part of the initial condition is immediate from the construction of $\mathcal K^t$ (it follows from \cite[Proposition 4.2 (i)]{GRW 14}) and by definition of $V^0_j$ we have 
$$V^{0}_{j} = \beta^{0}(\alpha^{-1}(V_{j}|_{(a_{0}-\varepsilon_{0}, a_{p}+\varepsilon_{p})})).$$ 
The maps $\beta^{0}$ and $\alpha$ agree and so $V^{0}_{j} = \alpha(\alpha^{-1}(V_{j}|_{(a_{0}-\varepsilon_{0}, a_{p}+\varepsilon_{p})}))$.
To prove that 
$$V^{0}_{j} =  V_{j}|_{(a_{0}-\varepsilon_{0}, \; a_{p}+\varepsilon_{p})},$$ it will therefore suffice to show that $\alpha: H_{n+1}^{\cpt}(W') \longrightarrow H^{\cpt}_{n+1}(W|_{(a_{0}-\varepsilon_{0}, a_{p}+\varepsilon_{p})})$ maps surjectively onto the subspace $V_{j}|_{(a_{0}-\varepsilon_{0}, a_{p}+\varepsilon_{p})}$. This follows from the lemma above. 
\end{proof}

We shall now verify that these subspaces $V_j^t$ indeed let us define the map $\mathcal F_p$ 
$$((a, \varepsilon, W, \ell_{W}, V), e, \ell) \in \mb{D}^{\mathcal{L}, l-1}_{p, 0}.$$ 
Fix $i \in \{0, \dots, p\}$ and let $(W^{i}_{t}, \ell^{i}_{t})$ denote the family of $\theta$-manifolds 
$\mathcal{K}^{t}_{e_{i}, \ell_{i}}(W, \ell_{W})$. For each $j = 0, \dots, p$  thus
$$V^{t}_{j} \leq H^{\cpt}_{n+1}(W^{i}_{t}).$$ 
Finally, let us denote by $h: W^{i}_{t} \longrightarrow \R$ the \textit{height function} on $W^{i}_{t}$ given by projecting $W^{i}_{t} \subset (a_{0}-\varepsilon_{0}, a_{p}+\varepsilon_{p})\times\R^{\infty-1}$ onto the first coordinate of the ambient space and abuse notation by setting
\[W := W|_{(a_{0}- \varepsilon_{0}, a_{p}+\varepsilon_{p})}.\] 

\begin{lemma} \label{lemma: commutativity of restrictions}
Let $c \in \cup_{k=0}^{p}(a_{k}-\varepsilon_{k}, a_{k}+\varepsilon_{k})$ be a regular value for the height function $h: W^{i}_{t} \longrightarrow \R$. Let 
$$\xymatrix{H_{n}(W|_{c})  && H_{n}(W'|_{c}) \ar[ll]_{\alpha_{c}}  \ar[rr]^{\beta^{t}_{c}} && H_{n}(W^{i}_{t}|_{c})}$$
denote the maps induced by inclusion. Then for any $j = 0, \dots, p$, the two subspaces 
$$V^{t}_{j}|_{c} \; \leq \; H_{n}(W^{i}_{t}|_{c}) \quad \text{and} \quad  \beta^{t}_{c}(\alpha^{-1}_{c}(V_{j}|_{c})) \; \leq \; H_{n}(W^{i}_{t}|_{c})$$
are equal.
\end{lemma}

\begin{proof}
Let 
$$\pi_{c}: H^{\cpt}_{n+1}(W) \longrightarrow H_{n}(W|_{c}) \quad \text{and} \quad \pi'_{c}: H^{\cpt}_{n+1}(W') \longrightarrow H_{n}(W'|_{c})$$
denote the restriction maps. To prove the lemma it will suffice to show that 
$$\alpha^{-1}_{c}(\pi_{c}(V_{j})) = \pi'_{c}(\alpha^{-1}(V_{j})):$$ 
The result is then immediate from the commutativity of the diagram
$$\xymatrix{H^{\cpt}_{n+1}(W') \ar[d]^{\pi'_{c}} \ar[rr]^{\beta^{t}} && H^{\cpt}_{n+1}(W^{i}_{t})  \ar[d] \\
H_{n}(W'|_{c}) \ar[rr]^{\beta^{t}_{c}}&& H_{n}(W^{i}_{t}|_{c}). }$$
To show the equality $\alpha^{-1}_{c}(\pi_{c}(V_{j})) = \pi'_{c}(\alpha^{-1}(V_{j}))$ we need to make some calculations. Recall from Lemma \ref{alpha surjective}, that $\alpha\colon H^{\cpt}_{n+1}(W') \longrightarrow H^{\cpt}_{n+1}(W)$ is an isomorphism $l<n-1$ and surjective when $l = n-1$. The verification of the equality $\alpha^{-1}_{c}(\pi_{c}(V_{j})) = \pi'_{c}(\alpha^{-1}(V_{j}))$ breaks down into two cases: 

\textit{Case 1:} Suppose $l < n-1$. We desire to show that $\alpha_{c}: H_{n}(W'|_{c}) \longrightarrow H_{n}(W|_{c})$
is an isomorphism. With this, the equality $\alpha^{-1}_{c}(\pi_{c}(V_{j})) = \pi'_{c}(\alpha^{-1}(V_{j}))$ will follow from the commutative diagram
$$\xymatrix{H^{\cpt}_{n+1}(W) \ar[d]^{\pi_{c}} \ar[rr]_{\cong}^{\alpha^{-1}} && H^{\cpt}_{n+1}(W')  \ar[d]^{\pi'_{c}} \\
H_{n}(W|_{c}) \ar[rr]^{\alpha^{-1}_{c}}_{\cong} && H_{n}(W'|_{c}).}$$
To prove that $\alpha_{c}$ is an isomorphism, we need to analyze the pair $(W|_{c}, W'|_{c})$. This pair takes on two forms depending on whether or not $c$ is contained in the interval $(a_{i}-\varepsilon_{i}, a_{p}+\varepsilon_{p})$. 

Let us first suppose that $c \in (a_{i}-\varepsilon_{i}, a_{p}+\varepsilon_{p})$. In this case we have
\begin{equation} \label{equation: level set of P}
\begin{aligned}
P_{i}|_{c} &= \Lambda_{i}\times\{c\}\times S^{l}\times\R^{2n-l}, \\
P^{\partial}_{i}|_{c} & = \Lambda_{i}\times\{c\}\times S^{l}\times(\R^{2n-l}\setminus \Int(D^{2n-l})),
\end{aligned}
\end{equation}
where $P$ and $P^{\partial}$ are from (\ref{equation: defn of P}). Since $l < n-1$, it follows that $H_{k}(P_{i}|_{c}, P^{\partial}_{i}|_{c}) = 0$ for all $k \leq n+1$. Excision for the pair $(W|_{c}, W'|_{c})$ yields
$$H_{k}(W|_{c}, W'|_{c}) \cong H_{k}(P|_{c}, P^{\partial}|_{c}) \quad \text{for all $k$,}$$
and thus we obtain 
$$H_{n+1}(W|_{c}, W'|_{c}) \quad \text{for all $k \leq n+1$.}$$
From the exact sequence associated to $(W|_{c}, W'|_{c})$ it follows that
$$\alpha_{c}: H_{n}(W'|_{c}) \stackrel{\cong} \longrightarrow H_{n}(W|_{c})$$
is an isomorphism whenever $c \in  (a_{i}-\varepsilon_{i}, a_{p}+\varepsilon_{p})$ (assuming $l < n-1$). 

For $c \notin  (a_{i}-\varepsilon_{i}, a_{p}+\varepsilon_{p})$, we have $W|_c = W'|_c$ and $\alpha_c$ is the identity so there is nothing to show.

\textit{Case 2:} Suppose that $l = n-1$. In this case the maps $\alpha$ and $\alpha_{c}$ are not necessarily isomorphisms and so we cannot employ the same argument used above. Consider the commutative diagram 
\[\xymatrix{ 0 & H^{\cpt}_{n+1}(W) \ar[d]_-{\pi_{c}}  \ar[l] && H^{\cpt}_{n+1}(W') \ar[ll]_-{\alpha} \ar[d]^-{\pi'_{c}} && H^{\cpt}_{n+2}(W, W') \ar[ll]_-{\partial} \ar[d]^{\bar{\pi}_{c}} \\
0 & H_{n}(W|_{c}) \ar[l]  &&  H_{n}(W'|_{c}) \ar[ll]_-{\alpha_{c}} && H_{n+1}(W|_{c}, W'|_{c}) \ar[ll]_-{\partial_{c}}}\]
which has exact rows. To establish $\alpha^{-1}_{c}(\pi_{c}(V_{j})) = \pi'_{c}(\alpha^{-1}(V_{j}))$, it will suffice to prove that the right-vertical map $\bar{\pi}_{c}$ is surjective: Indeed, the equality $\alpha^{-1}_{c}(\pi_{c}(V_{j})) = \pi'_{c}(\alpha^{-1}(V_{j}))$ can then be verified through a simple diagram chase. The map $\bar{\pi}_{c}$ takes on two forms depending on whether or not $c$ is contained in the interval $(a_{i}-\varepsilon_{i}, a_{p}+\varepsilon_{p})$. 

So assume that $c \in (a_{i}-\varepsilon_{i}, a_{p}+\varepsilon_{p})$. By (\ref{equation: defn of P}) and (\ref{equation: level set of P}) it follows that the restriction map 
$$H^{\cpt}_{n+2}(P_{i}, P^{\partial}_{i}) \longrightarrow H^{\cpt}_{n+1}(P_{i}|_{c}, P^{\partial}_{i}|_{c})$$
is an isomorphism. By the commutativity of the diagram 
$$\xymatrix{H^{\cpt}_{n+2}(W, W') \ar[r]^{\cong} \ar[d]^{\bar{\pi}_{c}} & H_{n+2}(P_{i}, P^{\partial}_{i}) \ar[d]^{\cong} \\
H_{n+1}(W|_{c}, W'|_{c}) \ar[r]^{\cong} & H_{n+1}(P_{i}|_{c}, P_{i}^{\partial}|_{c})}$$
it follows that $\bar{\pi}_{c}$ is an isomorphism, and hence surjective. This establishes the first case. 

For $c \notin (a_{i}-\varepsilon_{i}, a_{p}+\varepsilon_{p})$ we again have $W'|_{c} = W|_{c}$ making $\bar{\pi}_{c}$ is surjective since its target vanishes.
\end{proof}

The next proposition requires the use of Lemma \ref{lemma: commutativity of restrictions} and the results of Section \ref{section: preliminaries on surgery}. 

\begin{proposition} \label{proposition: induced lagrangians}
Fix $j \in \{0, \dots, p\}$. If $c \in (a_{j}-\tfrac{1}{2}\varepsilon_{j}, a_{j}+\tfrac{1}{2}\varepsilon_{j})$ is a regular value for the height function $h$, 
then the submodule $V_{j}^{t}|_{c} \leq H_{n}(W^{i}_{t}|_{c})$ is a Lagrangian subspace. Furthermore, for all $k = 0, \dots, p$, we have $V_{k}^{t}|_{c} \leq V_{j}^{t}|_{c}$.
\end{proposition}

\begin{proof}
Let $c \in (a_{j}-\tfrac{1}{2}\varepsilon_{j}, a_{j}+\tfrac{1}{2}\varepsilon_{j})$ is a regular value for the height function $h$.
Proving that $V_{j}^{t}|_{c}$ is Lagrangian breaks down into two cases depending on the form the level set $W^{i}_{t}|_{c}$ takes: 
 $c$ is automatically a regular value for $h: W \longrightarrow \R$ and by design (compare \cite[Proposition 4.2 (iv)]{GRW 14}) either
\begin{enumerate} 
\item[(a)] there is a diffeomorphism $W^{i}_{t}|_{c} \cong W|_{c}, \; \rel \; W'|_{c}$, or
\item[(b)] $W^{i}_{t}|_{c}$ is obtained from $W|_{c}$ by a collection of $\theta$-surgeries of degree $l$.
\end{enumerate}

\textit{Case (a):} Since the diffeomorphism is relative to $W'$ we obtain a commutative diagram 
$$\xymatrix{H_{n}(W|_{c}) \ar[drr]^{\cong}_{\varphi} && H_{n}(W'|_{c}) \ar[ll]_{\alpha_{c}} \ar[d]^{\beta^{t}_{c}} \\
&& H_{n}(W^{i}_{t}|_{c}),}$$
where the diagonal map $\varphi$ is the isomorphism induced by the diffeomorphism 
$$W^{i}_{t}|_{c} \cong W|_{c}, \; \rel \; W'|_{c}.$$
Since $\varphi$ is an isomorphism that preserves the intersection form on $H_{n}(W^{i}_{t}|_{c})$, it follows that 
$$\varphi(V_{j}|_{c}) \leq H_{n}(W^{i}_{t}|_{c})$$ 
is a Lagrangian subspace. Now, $\alpha_{c}$ maps surjectively onto the subspace $V_{j}|_{c}$. This fact together with commutativity of the above diagram implies that$\beta^{t}_{c}(\alpha_{c}^{-1}(V_{j}|_{c})) = \varphi(V_{j}|_{c}).$ By Lemma \ref{lemma: commutativity of restrictions} we have $\beta^{t}_{c}(\alpha_{c}^{-1}(V_{j}|_{c})) = V^{t}_{j}|_{c}$, and thus $V^{t}_{j}|_{c}$ is a Lagrangian as well. 

\textit{Case (b):} In this case, Theorem \ref{theorem: transport of lagrangian through surgery} implies that the subspace $\beta^{t}_{c}(\alpha_{c}^{-1}(V_{j}|_{c})) \leq H_{n}(W^{i}_{t}|_{c})$ is a Lagrangian. Again, by Lemma \ref{lemma: commutativity of restrictions} we have 
$$\beta^{t}_{c}(\alpha_{c}^{-1}(V_{j}|_{c})) = V^{t}_{j}|_{c},$$ 
and thus $V^{t}_{j}|_{c}$ is Lagrangian. 

To obtain the addendum note that by definition of $\mb{X}^{\mathcal{L}, l}_{\bullet}$, we have $V_{k}|_{c} \leq V_{j}|_{c}$. Thus, for all $t$ we have 
$$\beta^{t}_{c}(\alpha_{c}^{-1}(V_{k}|_{c})) \leq \beta^{t}_{c}(\alpha_{c}^{-1}(V_{j}|_{c})).$$
Another application of Lemma \ref{lemma: commutativity of restrictions} finishes the proof.
\end{proof}

With these properties established we can define
\[\mathcal F_p \colon [0,1]^{p+1} \times \mb{D}^{\mathcal L,l}_{p,0} \longrightarrow \mb{X}_p^{\mathcal L,l-1}\]
by
\[(t,(a,\epsilon,(W,\ell_W),V),(e,\ell)) \longmapsto (a,\epsilon/2,\mathcal K^t_{e,\ell}(W,\ell_W),V^t),\]
where $V^t$ is obtained by iterating the above construction just as $\mathcal K^t_e$ is iteratively built from the various $\mathcal K^t_{e_i}$'s. 

The proof of Theorem \ref{theorem: surgery on objects below middle dim}, which is the goal of this entire section, may now be concluded just as that of \cite[Theorem 4.1]{GRW 14} is in \cite[Sections 3.3 \& 4.4]{GRW 14} upon replacing $\mathcal S_p$ by $\mathcal F_p$: First, the `perform surgery'-map $\mathcal F(1,-)\colon |\mb{D}^{\mathcal L,l}_{\bullet,0}| \rightarrow |\mb{X}^{\mathcal L,l-1}_\bullet|$ factors through the inclusion $i$ of $|\mb{X}^{\mathcal L,l}_\bullet|$ into the target. Secondly, the composition of $\mathcal F(1,-)\colon |\mb{D}^{\mathcal L,l}_{\bullet,0}| \rightarrow |\mb{X}^{\mathcal L,l}_\bullet|$ with the inclusion $|\mb{D}^{\mathcal L,l}_\bullet| \rightarrow |\mb{D}^{\mathcal L,l}_{\bullet,0}|$ given by the empty set of surgery data is a weak equivalence by Proposition \ref{proposition: equivalence to flexible model} and therefore $\mathcal F(1,-)$ is surjective on homotopy groups. Thirdly, regarded as a map $|\mb{D}^{\mathcal L,l}_{\bullet,0}| \rightarrow |\mb{X}^{\mathcal L,l-1}_\bullet|$, $\mathcal F(1,-)$ is homotopic to $\mathcal F(0,-)$, which is just forgetful map $|\mb{D}^{\mathcal L,l}_{\bullet,0}| \rightarrow |\mb{X}^{\mathcal L,l-1}_\bullet|$, and thus a weak equivalence by Theorem \ref{theorem: contractibility of space of surgery data (n-1)-conn}. This means that $\mathcal F(1,-)\colon |\mb{D}^{\mathcal L,l}_{\bullet,0}| \rightarrow |\mb{X}^{\mathcal L,l}_\bullet|$ is also injective on homotopy groups and so finally $i$ has to be weak equivalence as well.


\section{Surgery on Objects in the Middle Dimension} \label{section: surgery on objects in the middle dimension}

In this section and the next we prove Theorem \ref{theorem: surgery on objects in degree n} which asserts that there are weak homotopy equivalences
$B\Cob^{\mathcal{L}, n}_{\theta}\; \stackrel{\simeq} \longrightarrow \; B\Cob^{\mathcal{L}, n-1}_{\theta}$ whenever $n \geq 4$ and $n \neq 7$, the tangential structure $\theta: B \longrightarrow BO(2n+1)\langle n\rangle$ is weakly once-stable and $B$ $n$-connected. It is structured similarly to Section \ref{section: surgery on objects below the middle dimension}, but the Lagrangians really come into play now, so many of the geometric arguments are necessariy different from those of \cite{GRW 14}. The present section essentially contains the formal outline and those statements which do immediately follow from \cite[Section 5]{GRW 14}, which we again mimick closely, and some necessary homological arguments, while the new geometric arguments are relegated to the next section.

\subsection{A semi-simplicial resolution} \label{subsection: resolution middle dimension}
We want to consider a semi-simplicial space $Y^n_\bullet(x)$ of middle dimensional surgery data on some element $x \in \mb{D}^{\mathcal L,n-1}_\bullet$ \`a la \cite[Definition 5.13]{GRW 14}. There are two main differences to be taken into account: The minor one is that the entirety of \cite[Section 5]{GRW 14} is written for the case of a $2n$-cobordism and $n-1$-surgeries. Adopting the construction for $d=2n+1$ and $n$-surgeries, however, requires nothing more than careful remixing of the numbers $n-1,n$ and $n+1$ that appear as sub- and superscripts. After these adaptions the major point is that we have to add a condition to the definition of $Y^n_\bullet(x)$ taking the Lagrangian on $x$ into account. This condition is in fact the entire raison d'\^etre for carrying the Langrangians through the surgery process. 

To keep the section somewhat readable we decided against including an exhaustive list of the numerical changes and only indicate the most pertinent ones. In that spirit we alter the definition of $Y^n_q(x)$ from \cite[Definition 5.13]{GRW 14} for $x = (a,\varepsilon,(W,\ell_W),V) \in \mb{D}^{\mathcal L,n-1}_p$ as follows: As we want to perform $n$-surgeries, the embedding $e$ is to be of the form
\[\Lambda \times \mathbb R \times (-6,-2) \times \mathbb R^n \times D^{n+1} \longrightarrow \mathbb R \times (0,1) \times (-1,1)^{\infty-1}.\]
Conditions i) to iv) require no further change. Condition v) should be altered to condition v') below, but the main distinction with the even-dimensional case is the inclusion of condition vi'), the meaning of which is explained in the next lemma:
\begin{itemize}
\item[v')] The manifold arising from $M_i = W_{|a_i}$ by surgery along $\partial e_i$ is $n$-connected.
\item[vi')] The subspace $$\sum_{i=0}^{p}V_{i} \leq H^{\cpt}_{n+1}(W|_{(a_{0}-\varepsilon_{0}, a_{p}+\varepsilon_{p})})$$
is contained in the image of 
$$\alpha: H^{\cpt}_{n+1}(W') \longrightarrow H^{\cpt}_{n+1}(W|_{(a_{0}-\varepsilon_{0}, a_{p}+\varepsilon_{p})}),$$
where $\alpha$ is induced by the inclusion $W' \subset W|_{(a_{0}-\varepsilon_{0}, a_{p}+\varepsilon_{p})}$ for $W' = W|_{(a_{0}-\varepsilon_{0}, a_{p}+\varepsilon_{p})}\setminus\Int(\Image(e))$.
\end{itemize}
Condition vi') will be crucial in proving that these complexes of surgery data have contractible realizations; this is false for the complexes of surgery data without Lagrangians. vi') is most easily thought of via its following consequence:

\begin{lemma} \label{remark: embedding contained in lagrangian}
We have $V_i|_{M_i} \leq \mathrm{im}(H_n(M_i') \rightarrow H_n(M_i))$. In particular, for each $\lambda \in \Lambda_{i,j}$, the homology class represented by the embedding 
$$\partial e^{\lambda}_{i,j}: \{\lambda\}\times\{0\}\times\{a_{i}\}\times\{0\}\times\partial D^{n+1} \; \longrightarrow \; M_{i}$$
(obtained by restricting $\partial e_{i}$) is contained in $V_{i}|_{M_{i}}$.
\end{lemma}

\begin{proof}
The first claim is trivial. By Lemma \ref{lemma: intersection inclusion} we have an exact sequence
$$\xymatrix{H_n(M_i') \ar[r] & H_n(M_i) \ar[rr]^-{\lambda(\cdot, \; [\partial e^{\lambda}_{i}])} && \mathbb Z}.$$
Therefore $[\partial e^{\lambda}_{i}]$ pairs trivially with the entire image of $H_n(M_i')$, in particular with $V_{i}|_{M_{i}}$, by condition vi'). Since $V_{i}|_{M_{i}}$ is a Langrangian we find $[\partial e^{\lambda}_{i}] \in V_{i}|_{M_{i}}$.
\end{proof}

Now, just as in \cite[Definition 5.13]{GRW 14} put
\[\mb{D}^{\mathcal{L}, n}_{p, q} = \{(x,y) \mid x \in \mb{D}^{\mathcal{L}, n-1}_{p}, y \in \mb{Y}^{n}_{q}(x)\}\]
augmented over $\mb{D}^{\mathcal{L}, n-1}_{p}$ by forgetting $y$. 

\begin{theorem} \label{theorem: contractibility of space of surgery data}
Let $n \geq 4, n \neq 7$ Then the maps
$$|\mb{D}^{\mathcal{L}, n}_{\bullet, 0}| \longrightarrow |\mb{D}^{\mathcal{L}, n}_{\bullet, \bullet}| \longrightarrow |\mb{D}^{\mathcal{L}, n-1}_{\bullet}|,$$
induced by inclusion of zero-simplices and the augmentation, respectively, are weak equivalences.
\end{theorem}

The claim for the second map will be taken up in the next section. Granting this for now, it follows that the former map is an equivalence from the argument provided at the beginning of \cite[Section 6.1]{GRW 14}, once the correction we spelled out in Lemma \ref{lemma: inclusion of zero simplices} is taken into account.

\subsection{A surgery move} \label{subsection: the surgery move degree n}
To conclude the proof of Theorem \ref{theorem: surgery on objects in degree n} assuming Theorem \ref{theorem: contractibility of space of surgery data} we proceed just as in the previous section by resurrecting the surgery move of \cite[Section 5.2]{GRW 14}. This requires another bout of index adjustments to obtain the family that gives the surgery move and then some homological calculations to verify that it interacts correctly with the homological data.

To start the former we follow \cite[Section 5.2]{GRW 14} and consider the submanifold 
\[K = \{(x,y) \in \mathbb R^{n+1} \times \mathbb R^{n+1} \mid |y|^2 = \rho(|x|^2-1)\}.\]
Using this a starting point the construction of \cite[Sections 4.4 \& 5.2]{GRW 14} we obtain for each $(a,\varepsilon,(W,\ell_W),V,(e,\ell)) \in \mb{D}^{\mathcal L,n-1}_{p,0}$ a continuous family of manifolds 
\[\mathcal K^t_{e,\ell}(W,\ell_W) \in \Psi_\theta((a_0-\varepsilon_0,a_p+\varepsilon_p) \times \mathbb R^\infty)\]
and from it we want to define maps
\[\mathcal F_p \colon [0,1]^{p+1} \times \mb D^{\mathcal L,n}_{p,0} \longrightarrow \mb X_p^{\mathcal L,n-1}\]
just as in \cite[Sections 4.4 \& 5.4]{GRW 14}. 

To this end we need to check that $\mathcal K^t_{e,\ell}(W,\ell_W)$ preserves the connectivity assumption on $W$, i.e. condition iii) of Definition \ref{defn: more flexible model} on the one hand and produce new homological data on $\mathcal K^t_{e,\ell}(W,\ell_W)$ from $V$ on the other. The former is handled by \cite[Proposition 5.12 (iii)]{GRW 14} whose claim is not affected by the change in $K$: The critical points of the Morse function arising in the proof are now of index $n$ and $n+1$,  and this increase cancels against the dimension increase from $2n$ to $2n+1$. 

To construct the homological data for $\mathcal K^t_{e,\ell}(W,\ell_W)$ we use the same formula as in (\ref{equation: induced subspaces (n-1)-surgery}): Given subspaces $V_{0}, \dots, V_{p} \leq H_{n+1}^{\cpt}(W|_{(a_{0}- \varepsilon_{0}, a_{p}+\varepsilon_{p})})$ we again set
\[\beta^{t}(\alpha^{-1}(V_{j}|_{(a_{0}-\varepsilon_{0}, a_{p}+\varepsilon_{p})})) = V^{t}_{j} \; \leq \; H^{\cpt}_{n+1}(\mathcal{K}^{t}_{e_{i}, \ell_{i}}(W, \ell_{W})),\]
where $W'$ again denotes the complement $W|_{(a_{0}- \varepsilon_{0}, a_{p}+\varepsilon_{p})}\setminus\Int(\Image(e_{i}))$ and 
$$\xymatrix{H^{\cpt}_{n+1}(W|_{(a_{0}- \varepsilon_{0}, a_{p}+\varepsilon_{p})})  && H_{n+1}^{\cpt}(W') \ar[ll]_-{\alpha} \ar[rr]^-{\beta^{t}} && H^{\cpt}_{n+1}(\mathcal{K}^{t}_{e_{i}, \ell_{i}}(W, \ell_{W}))}$$
denote the maps induced by inclusion. The following proposition is the analogue of Proposition \ref{proposition: continuity of Kappa n-1}. 

\begin{proposition} \label{proposition: continuity of Kappa n}
The above construction defines a (continuous) map
\[[0,1] \times \mb D_{p,0}^{\mathcal L,n} \longrightarrow \bPsi_{\theta}^{\Delta}((a_{0}- \varepsilon_{0}, a_{p}+\varepsilon_{p})\times\R^{\infty})\]
\[(t,(W,\ell_W),V,(e,\ell)) \longmapsto (\mathcal{K}^{t}_{e_{i}, \ell_{i}}(W, \ell_W), V^{t}_{j}),\]
with initial value given by
$$(\mathcal{K}^{0}_{e_{i}, \ell_{i}}(W, \; \ell_{W}), \; V^{0}_{j}) \; = \; \left(W|_{(a_{0}-\varepsilon_{0}, a_{p}+\varepsilon_{p})}, \; \ell|_{(a_{0}-\varepsilon_{0}, a_{p}+\varepsilon_{p})}, \; V_{j}|_{(a_{0}-\varepsilon_{0}, \; a_{p}+\varepsilon_{p})}\right).$$
\end{proposition}

\begin{proof} This is entirely analogous to Proposition \ref{proposition: continuity of Kappa n-1}, except that the surjectivity of $\alpha$ is directly implied by condition vi') from the definition of $\mb Y^n_\bullet$.
\end{proof}

To verify that the $V^t_j$ are indeed eligible subspace we again need:

\begin{lemma} \label{lemma: commutativity of restrictions degree n}
Let $c \in \cup_{k=0}^{p}(a_{k}-\varepsilon_{k}, a_{k}+\varepsilon_{k})$ be a regular value for the height function $h: W^{i}_{t} \longrightarrow \R$.
Let 
$$\xymatrix{H_{n}(W|_{c})  && H_{n}(W'|_{c}) \ar[ll]_{\alpha_{c}}  \ar[rr]^{\beta^{t}_{c}} && H_{n}(W^{i}_{t}|_{c})}$$
denote the maps induced by inclusion. Then for any $j = 0, \dots, p$, the two subspaces 
$$V^{t}_{j}|_{c} \; \leq \; H_{n}(W^{i}_{t}|_{c}) \quad \text{and} \quad  \beta^{t}_{c}(\alpha^{-1}_{c}(V_{j}|_{c})) \; \leq \; H_{n}(W^{i}_{t}|_{c})$$
are equal.
\end{lemma}

\begin{proof} 
Consider the commutative diagram 
\begin{equation} \label{equation: restriction square n}
\xymatrix{H^{\cpt}_{n+1}(W) \ar[d]^{\pi_{c}} & H^{\cpt}_{n+1}(W') \ar[l]_-{\alpha}  \ar[d]^{\pi'_{c}} \\
H_{n}(W|_{c}) & H_{n}(W'|_{c}) \ar[l]_-{\alpha_{c}}.}
\end{equation}
As in the proof of Lemma \ref{lemma: commutativity of restrictions}, it will suffice to prove the equality $\alpha^{-1}_{c}(\pi_{c}(V_{j})) = \pi'_{c}(\alpha^{-1}(V_{j}))$. As in the proof of Lemma \ref{lemma: commutativity of restrictions}, let $P_{i}$ and $P^{\partial}_{i}$ be the manifolds
\begin{align*}
P_{i} & =  \Lambda_{i}\times(a_{i}-\varepsilon_{i}, a_{p}+\varepsilon_{p})\times S^{n}\times\R^{n}, \\
P^{\partial}_{i} & = \Lambda_{i}\times(a_{i}-\varepsilon_{i}, a_{p}+\varepsilon_{p})\times S^{n}\times(\R^{n}\setminus \Int(D^{n})).
\end{align*} 
As before, we have excision isomorphisms 
$$\begin{aligned}
H^{\cpt}_{k}(W, W') &\cong H^{\cpt}_{k}(P_{i}, P^{\partial}_{i}) \\
H_{k}(W|_{c}, W'|_{c}) &\cong 
\begin{cases}
H^{\cpt}_{k}(P_{i}|_{c}, P^{\partial}_{i}|_{c}) &\quad \text{if $c \in (a_{i}-\varepsilon_{i}, a_{p}+\varepsilon_{p})$,}\\
0 &\quad \text{if $c \notin (a_{i}-\varepsilon_{i}, a_{p}+\varepsilon_{p})$.}
\end{cases}
\end{aligned}$$
It follows that $H^{\cpt}_{*}(W, W')$ is trivial in all degrees other than $(n+1)$ and that $ H_{k}(W|_{c}, W'|_{c})$ is trivial in all degrees other than $n$ (if $c \notin (a_{i}-\varepsilon_{i}, a_{p}+\varepsilon_{p})$ then $H_{k}(W|_{c}, W'|_{c})$ is trivial in all degrees). Using this together with the long exact sequences on $H^{\cpt}_{*}$ and $H_{*}$ associated to the pairs $(W, W')$ and $(W|_{c}, W'|_{c})$, it follows that both maps $\alpha$ and $\alpha_{c}$ are injective. By condition vi'), every element of $V_{j}$ lies in the image of $\alpha$. Using these facts, a simple diagram chase in (\ref{equation: restriction square n}) proves that $\alpha^{-1}_{c}(\pi_{c}(V_{j})) = \pi'_{c}(\alpha^{-1}(V_{j}))$.
\end{proof}

The next proposition is the analogue of Proposition \ref{proposition: induced lagrangians}.

\begin{proposition}
Let $j = 0, \dots, p$. If $c \in (a_{j}-\tfrac{1}{2}\varepsilon_{j}, a_{j}+\tfrac{1}{2}\varepsilon_{j})$ is a regular value for the height function $h$, 
then the submodule $V^{t}_{j}|_{c} \leq H_{n}(W^{i}_{t}|_{c})$ is a Lagrangian subspace. Furthermore, $V^{t}_{k}|_{c} \leq V^{t}_{j}|_{c}$ for $k = 0, \dots, p$.
\end{proposition}

\begin{proof} By design of the surgery move there are two cases depending on the form that the level set $W^{i}_{t}|_{c}$ takes. To see that $V^{t}_{j}|_{c}$ is indeed a Lagrangian we seperate the cases again: 

\textit{Case 1:} There is a diffeomorphism $W^{i}_{t}|_{c} \cong W|_{c}, \; \rel \; W'|_{c}$. This is entirely the same as Case (a) in Proposition \ref{proposition: induced lagrangians}, except that ${V_j}|_c$ lies in the image of $\alpha_c$ directly by condition vi'). The property that $V^{t}_{k}|_{c} \leq V^{t}_{j}|_{c}$ for $k = 0, \dots, p$ also follows just as in Proposition \ref{proposition: induced lagrangians} using Lemma \ref{lemma: commutativity of restrictions degree n} instead of \ref{lemma: commutativity of restrictions}. 

\textit{Case 2:} The manifold $W^{i}_{t}|_{c}$ is obtained from $W|_{c}$ by a collection of surgeries in degree $n$ and $W^{i}_{t}|_{c}$ is $n$-connected. 
It follows that $H_{n}(W^{i}_{t}|_{c}) = 0$ and thus $V^{t}_{j}|_{c}$ is automatically Lagrangian and the claimed containment is a trivial equality.
\end{proof}

With these proposition established the proof is again concluded just as in \cite[Sections 4.4 \& 5.4]{GRW 14} (outlined at the end of Section \ref{subsection: surgery move below middle dimension}).


\section{Contractibility of the Space of Surgery Data}  \label{section: contractibility n}
We finally prove Theorem \ref{theorem: contractibility of space of surgery data} which asserts that there are weak homotopy equivalences, 
$$|\mb{D}^{\mathcal{L}, n}_{\bullet, 0}| \stackrel{\simeq} \longrightarrow |\mb{D}^{\mathcal{L}, n}_{\bullet, \bullet}| \stackrel{\simeq} \longrightarrow |\mb{D}^{\mathcal{L}, n-1}_{\bullet}|,$$
where the first map is induced by the inclusion of zero-simplices and the second is induced by the augmentation.
As mentioned before, the first homotopy equivalence $|\mb{D}^{\mathcal{L}, n}_{\bullet, 0}| \rightarrow|\mb{D}^{\mathcal{L}, n}_{\bullet, \bullet}|$ is deduced from the second just as in Lemma \ref{lemma: inclusion of zero simplices} (see also \cite[Page 327]{GRW 14}) and so we omit the proof of this and focus on establishing the second weak homotopy equivalence, $|\mb{D}^{\mathcal{L}, n}_{\bullet, \bullet}|  \rightarrow |\mb{D}^{\mathcal{L}, n-1}_{\bullet}|$. We would like to apply Theorem \cite[Theorem 6.2]{GRW 14} again but it turns out that a slightly stronger version is required, even in \cite{GRW 14} as pointed out in the erratum \cite{GRW 14b}.

\subsection{A stronger simplicial technique}
It was observed by the second author that property iii) of \cite[Theorem 6.2]{GRW 14} does not hold in the case of middle dimensional surgery, already in the even dimensional case. This oversight was corrected in the erratum \cite{GRW 14b} and we will need to use the following strengthening from \cite[Theorem 6.4]{BP 15}, which was abstracted from \cite{GRW 14b}.

Consider a symmetric relation $\mathcal{R}$ that is open and dense as a subset of the fibred product $X_{0}\times_{X_{-1}}X_{0}$.

\begin{theorem} \label{theorem: improved flag complex theorem}
Let $X_{\bullet} \longrightarrow X_{-1}$ be an augmented topological flag complex that satisfies conditions (i) and (ii) of \cite[Theorem 6.2]{GRW 14}.
Let $\mathcal{R} \subset X_{0}\times_{X_{-1}}X_{0}$ be an open and dense symmetric relation with the property that $X_{1} \subset \mathcal{R}$.
Suppose that $X_{\bullet} \longrightarrow X_{-1}$ satisfies the following further condition:
\begin{enumerate} 
\item[(iii)*] Let $x \in X_{-1}$. Given:
\begin{itemize} 
\item  a non-empty subset, $\{v_{1}, \dots, v_{m}\} \subset \varepsilon^{-1}(x)$, whose elements are pairwise related by $\mathcal{R}$, and 
\item  an arbitrary subset, $\{w_{1}, \dots, w_{k}\} \subset \varepsilon^{-1}(x)$, such that $(v_{i}, w_{j}) \in X_{1}$ for all $i, j$, 
\end{itemize}
there exists $v \in \varepsilon^{-1}(x)$ such that $(v, v_{i}) \in X_{1}$ and $(v, w_{j}) \in X_{1}$ for all $i, j$. 
\end{enumerate}
If condition (iii)* is satisfied for all $x \in X_{-1}$ then the induced map $|X_{\bullet}|  \longrightarrow X_{-1}$ is a weak homotopy equivalence. 
\end{theorem}

\subsection{Proof of Theorem \ref{theorem: contractibility of space of surgery data}}
To apply Theorem \ref{theorem: improved flag complex theorem} we will yet another modification of $\mb{D}^{\mathcal{L}, n}_{\bullet, \bullet}$. Define the \textit{core} of the surgery tube to be
\[C \; = \; \{0\}\times(-6, -2)\times\{0\}\times D^{n+1} \;  \; \subset \; \; \R\times(-6, -2)\times\R^{n}\times D^{n+1}.\]

\begin{defn} \label{defn: flexible version n}
For $x = (a, \varepsilon, (W, \ell_{W}), V) \in \mb{D}^{\mathcal{L}, n-1}_{\bullet}$, let $\widetilde{\mb{Y}}_{\bullet}(x)$ be the semi-simplicial space from Section \ref{subsection: resolution middle dimension} except now we only ask the map $e$ to be a smooth embedding on a neighborhood of the subset
\[\Lambda\times C \subset \Lambda\times\left(\R\times (-6, -2)\times\R^{n}\times D^{n+1}\right).\]
Furthermore, define a bi-semi-simplicial space by 
\[\widetilde{\mb{D}}^{\mathcal{L}, n}_{p, q} = \{(x, y) \; | \; x \in \mb{D}^{\mathcal{L}, n-1}_{p}, \; y \in \widetilde{\mb{Y}}_{q}(x)\}.\]
\end{defn}

Using the projection we obtain an augmented bi-semi-simplicial space $\widetilde{\mb{D}}^{\mathcal{L}, n}_{\bullet, \bullet} \; \longrightarrow \; \widetilde{\mb{D}}^{\mathcal{L}, n}_{\bullet}$ with $\widetilde{\mb{D}}^{\mathcal{L}, n}_{\bullet, -1} = \widetilde{\mb{D}}^{\mathcal{L}, n-1}_{\bullet}$. Let $\mathcal{T} \subset \widetilde{\mb{D}}^{\mathcal{L}, n}_{p, 0}\times_{\widetilde{\mb{D}}^{\mathcal{L}, n}_{p, -1}}\widetilde{\mb{D}}^{\mathcal{L}, n}_{p, 0}$ be the subset consisting of those 
$$\left((a, \varepsilon, (W, \ell_{W}), V), \; (\Lambda_{1}, \delta_{1}, e_{1}, \ell_{1}), (\Lambda_{2}, \delta_{2}, e_{2}, \ell_{2})\right)$$ 
such that the embeddings $e_{1}|_{\Lambda_{1}\times C}$ and  $e_{2}|_{\Lambda_{2}\times C}$ are transverse. This subset $\mathcal{T}$ is clearly a symmetric and open relation. By the Thom transversality theorem applied to each of the fibres over $\widetilde{\mb{D}}^{\mathcal{L}, n}_{p, -1}$ we see that it is a dense subset of the fibred product $\mathcal{T} \subset \widetilde{\mb{D}}^{\mathcal{L}, n}_{p,0}\times_{\widetilde{\mb{D}}^{\mathcal{L}, n}_{p, -1}}\widetilde{\mb{D}}^{\mathcal{L}, n}_{p, 0}$.

Just as in \cite[Proposition 6.15]{GRW 14} it follows that the inclusion $\widetilde{\mb{D}}^{\mathcal{L}, n}_{\bullet, \bullet} \hookrightarrow \mb{D}^{\mathcal{L}, n}_{\bullet, \bullet}$ is a level-wise weak homotopy equivalence and thus it induces a weak homotopy equivalence $|\widetilde{\mb{D}}^{\mathcal{L}, n}_{\bullet, \bullet}| \simeq |\mb{D}^{\mathcal{L}, n}_{\bullet, \bullet}|$. To prove Theorem \ref{theorem: contractibility of space of surgery data} we will as before show that for each $p \in \Z_{\geq 0}$, the augmented topological flag complex 
\[\widetilde{\mb{D}}^{\mathcal{L}, n}_{p, \bullet} \longrightarrow \widetilde{\mb{D}}^{\mathcal{L}, n-1}_{p}\]
induces a weak homotopy equivalence $|\widetilde{\mb{D}}^{\mathcal{L}, n}_{p, \bullet}| \simeq \widetilde{\mb{D}}^{\mathcal{L}, n-1}_{p}$. We shall do so by applying Theorem \ref{theorem: improved flag complex theorem} for the relation $\mathcal T$. 

\subsection{Verficiation of Condition (iii)*} \label{subsection: condition iii}
The main technical tool that we use to establish condition (iii) is the proposition stated below. For what follows let $n \geq 4$, $M_{0}$ and $M_{1}$ two $(n-1)$-connected, $2n$-dimensional, closed manifolds, and $W$ a cobordism between $M_{0}$ and $M_{1}$ that is $(n-1)$-connected as well.

\begin{proposition} \label{corollary: multiple inductive disjunction}
Let 
$$f, g_{1}, \dots, g_{k}: (S^{n}\times[0, 1], S^{n}\times\{0, 1\}) \longrightarrow (W, M_{0}\sqcup M_{1})$$
be a collection of embeddings and let $x, y_{1}, \dots, y_{k} \in H_{n}(M_{0})$ denote the classes represented by 
$$f|_{S^{n}\times\{0\}}, \; g_{1}|_{S^{n}\times\{0\}}, \; \dots, \; g_{k}|_{S^{n}\times\{0\}}$$ 
respectively. 
Let $K_{1}, \dots, K_{l} \subset W$ be a collection of pairwise transverse submanifolds of codimension $\geq 3$, and let $K$ denote the union $\cup_{i=1}^{l}K_{i}$.
Suppose that the following conditions are met:
\begin{enumerate} 
\item[(a)] $\lambda(x, y_{i}) = 0$ for all $i = 1, \dots, k$;
\item[(b)] the embeddings $g_{1}, \dots, g_{k}$ are pairwise transverse.
\item[(c)] the images of $f$ and $g_{1}, \dots, g_{k}$ are contained in the complement, $W\setminus K$.
\end{enumerate}
Then there exists an isotopy $f_{t}: (S^{n}\times[0, 1], S^{n}\times\{0, 1\}) \longrightarrow (W, M_{0}\sqcup M_{1})$ with  $t \in [0, 1],$ that satisfies:
\begin{itemize} 
\item $f_{0} = f$, 
\item $f_{t}(S^{n}\times[0, 1]) \subset W\setminus K$ for all $t \in [0, 1]$,
\item $f_{1}(S^{n}\times[0, 1])\cap g_{i}(S^{n}\times[0, 1]) = \emptyset$ for all $i = 1, \dots, k.$
\end{itemize}
Suppose further that $f$ is such that $f(S^{n}\times\{0\})\cap g_{i}(S^{n}\times\{0\}) = \emptyset$ for all $i = 1, \dots, k.$ Then the isotopy $f_{t}$ can be chosen so that $f_{t}|_{S^{n}\times\{0\}} = f|_{S^{n}\times\{0\}}$ for all $t \in [0, 1]$. 
\end{proposition}

\begin{proof}
By condition (a) we may apply the \textit{Whitney trick} \cite[Theorem 6.6]{Mi 65} inductively to obtain an isotopy of $f|_{S^{n}\times\{0\}}$, that pushes $f(S^{n}\times\{0\}) \subset M_{0}$ off of the submanifolds $g_{1}(S^{n}\times\{0\}), \dots,  g_{k}(S^{n}\times\{0\}) \subset M_{0}$, while staying in the complement, $M_{0}\setminus(M_{0}\cap K)$. Thus we reduce to the case where
$$f(S^{n}\times\{0\})\cap g_{i}(S^{n}\times\{0\}) = \emptyset \quad \text{for all $i = 1, \dots, k$.}$$
We remark that in order to inductively apply the Whitney trick as we did above, it is necessary that the submanifolds $g_{1}(S^{n}\times\{0\}), \dots, g_{k}(S^{n}\times\{0\}) \subset M_{0}$ be pairwise transverse, see \cite[Proposition 6.9]{BP 15}.
We also remark that in order to apply Whitney trick in this situation it is necessary that the manifolds $K_{1}, \dots, K_{l} \subset W$ be pairwise transverse and have codimension $\geq 3$.
Indeed, the lower bound on the codimension, together with the pairwise transversality, ensures that the complement $W\setminus K$ will be simply connected. 
Simple connectivity of ambient space is required to apply the Whitney trick. 

Let $W'$, $M'_{0}$, and $M'_{1}$ denote the complements $W\setminus K$, $M_{0}\setminus(M_{0}\cap K)$, and $M_{1}\setminus(M_{1}\cap K)$. Since the co-dimension of $K$ is greater than or equal to $3$, it follows that the pair $(W', M'_{0})$ is $2$-connected. To prove the corollary, it will suffice to construct an isotopy 
$$f_{t}: (S^{n}\times[0, 1], S^{n}\times\{0, 1\}) \longrightarrow (W', M'_{0}\sqcup M'_{1}), \quad t \in [0, 1],$$
with $f_{0} = f$ and $f_{t}|_{S^{n}\times\{0\}} = f|_{S^{n}\times\{0\}}$ for all $t \in [0, 1]$, such that
$$f(S^{n}\times[0,1])\cap g_{i}(S^{n}\times[0,1]) \; = \; \emptyset$$
for all $i = 1, \dots, k$. Such an isotopy exists by inductive application of \textit{higher dimensional half-Whitney trick} from \cite[Theorem C.3]{BP 15} using the same inductive argument employed in \cite[Corollary 5.10.1]{BP 15}. We note that here as in \cite[Corollary 5.10.1]{BP 15}, it is essential that the embeddings $g_{1}, \dots, g_{n}$ are pairwise transverse.
\end{proof}

The next proposition establishes condition (iii)*. The proof is very similar to \cite[Lemma 5.11]{BP 15} and \cite[Proposition 6.19]{GRW 14} except in our situation we must take the intersection form and the Lagrangian subspaces into account. 

\begin{proposition} \label{proposition: condition iii}
For $p \in \Z_{\geq 0}$, let $x = (a, \varepsilon, (W, \ell_{W}), V) \in \widetilde{\mb{D}}^{\mathcal{L}, n-1}_{p}$.
\begin{itemize} 
\item Let $\{v_{1}, \dots, v_{k}\} \in \widetilde{\mb{Y}}_{0}(x)$ be a non-empty collection of elements that are pairwise transverse. 
\item Let $\{w_{1}, \dots, w_{s}\} \in \widetilde{\mb{Y}}_{0}(x)$ be a collection of elements with $(v_{i}, w_{j}) \in \widetilde{\mb{Y}}_{1}(x)$ for all $i, j$.
\end{itemize}
Then there exists $u \in \widetilde{\mb{Y}}_{0}(x)$ such that $(u, v_{j}) \in \widetilde{\mb{Y}}_{1}(x)$ and $(u, w_{j}) \in \widetilde{\mb{Y}}_{1}(x)$ for all $i, j$.
\end{proposition}

\begin{proof}
For each $j = 1, \dots, k$, let $(\Lambda^{v}_{j}, \delta^{v}_{j}, e^{v}_{j}, \ell^{v}_{j})$ denote the element $v_{j}$, and for $r = 1, \dots, s$ let $(\Lambda^{w}_{r}, \delta^{w}_{r}, e^{w}_{r}, \ell^{w}_{r})$ denote the element $w_{r}$. We temporarily set $u = v_{1}$ and write $u = (\Lambda, \delta, e, \ell)$. Since $(w_{r}, v_{j}) \in \widetilde{\mb{Y}}_{1}(x)$ for all $r, j$, it follows that 
$$e^{v}_{j}(\Lambda^{v}_{j}\times C)\cap e^{w}_{r}(\Lambda^{w}_{r}\times C) \; = \; \emptyset \quad \text{for all $j, r$,}$$
where $C$ is the core from Definition \ref{defn: flexible version n}. Let $K \subset W$ denote the union $\cup_{r=1}^{s}e^{w}_{r}(C)  \subset W$ and let $W'$ denote $W\setminus K$. We have $e(\Lambda\times C) \subset W'$ and $e^{v}_{j}(\Lambda^{v}_{j}\times C) \subset W'$ for all $j = 1, \dots, k$.

By Remark \ref{remark: embedding contained in lagrangian} the homology classes in $H_{n}(W|_{a_{0}})$ determined by the submanifolds
$$e^{v}_{j}(\Lambda^{v}_{j}\times C)\cap W|_{a_{0}} \subset W|_{a_{0}} \quad j = 1, \dots, k, $$
are all contained in the subspace $V_{0}|_{a_{0}} \subset H_{n}(W|_{a_{0}}),$ which is Lagrangian by definition. Similarly, for each $\lambda \in \delta^{-1}(i)$, the homology class determined by the submanifold 
$$e(\lambda\times C)\cap W|_{a_{0}} \; \subset \; W|_{a_{0}},$$
is contained in $V_{i}|_{a_{0}} \subset H_{n}(W|_{a_{0}})$ as well. Since the intersection form vanishes on $V_{0}$, and the set $\{v_{1}, \dots, v_{k}\}$ is in general position, it follows from Proposition \ref{corollary: multiple inductive disjunction} that for each $\lambda \in \Lambda$, there exists an isotopy of $e(\lambda\times C)$ that pushes $e(\lambda\times C)\cap W|_{[a_{0}, a_{1}]}$ off of $e^{v}_{j}(\lambda^{v}_{j}\times C)\cap W|_{[a_{0}, a_{1}]}$ for all $j = 1, \dots, k$, and keeps $e(\lambda\times C)$ inside of $W' = W\setminus K \subset W$. The rest of the proof of \cite[Proposition 6.19]{GRW 14} (or \cite[Lemma 5.11]{BP 15}) now applies.
\end{proof}

\subsection{Verification of Condition (ii)} \label{subsection: condition ii}
To prove Condition (ii), it will suffice to show that $\widetilde{\mb{Y}}_{0}(x)$ is non-empty for any $p$-simplex $x \in \mb{D}^{\mathcal{L}, n-1}_{p}$.
This will require us to develop some preliminary results. The technical tools involved include the embedding theorems of Haefliger and Hudson from \cite{Ha 61} and \cite{H 69}. Both of these embedding theorems require the manifolds involved to be above a certain dimension, and this requirement is the source of our condition on the integer $n$. 

Let $n \geq 4$ and $M_{0}$, $M_{1}$, and $W$ be as in the last section. 

\begin{proposition} \label{proposition: representing homology by cylinders}
Every relative homology class 
$$x \in H_{n+1}(W, M_{0}\sqcup M_{1})$$ 
is represented by an embedding $(S^{n}\times[0, 1], S^{n}\times\{0, 1\}) \; \longrightarrow \; (W, M_{0}\sqcup M_{1}).$
\end{proposition}

\begin{proof}
Let $x \in H_{n+1}(W, M_{0}\sqcup M_{1})$ be as in the statement of the proposition. Consider the boundary map 
\[{\partial}: H_{n+1}(W, M_{0}\sqcup M_{1}) \longrightarrow H_{n}(M_{0}\sqcup M_{1}),\]
and let $y$ denote the class ${\partial}(x)$. By the Hurewicz theorem (applied to $\pi_{n}(M_{0})$ and $\pi_{n}(M_{1})$), the class $y$ is represented by a map, $\phi: S^{n}\times\{0, 1\} \longrightarrow M_{0}\sqcup M_{1}$ sending $S^{n}\times\{i\}$ into $M_{i}$ for $i = 0, 1$. Let $\iota_{i}: M_{i} \hookrightarrow W$ denote the inclusion. By exactness, the class $y$ maps to zero under $H_{n}(M_{0}\sqcup M_{1}) \longrightarrow H_{n}(W)$. It follows that the maps, 
$$\iota_{0}\circ\phi|_{S^{n}\times\{0\}}, \; \; -\iota_{1}\circ\phi|_{S^{n}\times\{1\}}: S^{n} \longrightarrow W,$$
are homotopic, where $-\iota_{1}\circ\phi|_{S^{n}\times\{1\}}$ denotes the pre-composition of $\iota_{1}\circ\phi|_{S^{n}\times\{1\}}$ with some reflection (which reverses orientation). It then follows that there exists a map 
$$\Phi: (S^{n}\times[0, 1], S^{n}\times\{0, 1\}) \longrightarrow (W, M_{0}\sqcup M_{1})$$
such that 
$$\Phi|_{S^{n}\times\{0\}} = \iota_{0}\circ\phi|_{S^{n}\times\{0\}} \quad \text{and} \quad \Phi|_{S^{n}\times\{1\}} = -\iota_{1}\circ\phi|_{S^{n}\times\{1\}}.$$
Using Haefliger's and Hudson's embedding results from \cite{Ha 61} and \cite{H 69}, we may deform $\Phi$, to a new map $\Phi'$ such that $\Phi'$ is an embedding.  We note that the use of these embedding theorems is precisely where the assumption $n \geq 4$ comes into play. 

Now, let $w \in H_{n+1}(W, M_{0}\sqcup M_{1})$ denote the class represented by this embedding $\Phi'$. It follows that, ${\partial}w = y = {\partial}x.$ Let $v$ denote the difference $w - x \in H_{n+1}(W, M_{0}\sqcup M_{1})$. The class $v$ is in the kernel of ${\partial}$ and thus is in the image of $H_{n+1}(W) \longrightarrow H_{n+1}(W, M_{0}\sqcup M_{1}),$ and so (by Hurewicz' theorem again) $v$ is represented by a map $h: S^{n+1} \longrightarrow W.$ Since $W$ is $(n-1)$-connected, by Haefliger's embedding theorem \cite{Ha 61} we may assume that the map $h$ is an embedding as well. 
Now, let $\Psi: (S^{n}\times[0, 1], S^{n}\times\{0, 1\}) \longrightarrow (W, M_{0}\sqcup M_{1})$ be the map constructed by forming the connected sum of the image of $\Phi'$ with the image of $-h,$ along an embedded arc. 
The map $\Psi$ represents the class $w - v = w - (w - x) = x,$ and thus equals $x$. 
The map may not be embedding because the image of $\Phi'$ may have non-empty intersection with the image of $h$. 
However, with the map $\Psi$ constructed, we may apply Hudson's theorem again \cite[Theorem 1]{H 69} to find a homotopy of $\Psi$ to a new map $\widehat{\Psi}$ which is an embedding. 
By Hudson's theorem, this homotopy may not be fixed on the boundary of $S^{n}\times[0, 1]$, none-the-less the resulting map $\widehat{\Psi}$ still represents the class $x \in H_{n+1}(W, M_{0}\sqcup M_{1})$.
This concludes the proof of the proposition.
\end{proof}

\begin{theorem} \label{proposition: conditions ii and iii}
For $p \in \Z_{\geq 0}$, let $x = (a, \varepsilon, (W, \ell_{W}), V) \in \mb{D}^{\mathcal{L}, n}_{p, -1}$. The set $\widetilde{\mb{Y}}_{0}(x)$ is non-empty.
\end{theorem} 

\begin{proof}
As usual we write $V = (V_{0}, \dots, V_{p})$. For each $j = 0, \dots, p$, we write $M_{j} = W|_{a_{j}}$. Let $i \in \{0, \dots, p\}$. Consider the subspace $V_{i}|_{[a_{i}, a_{p}]} \leq H_{n+1}(W|_{[a_{i}, a_{p}]}, M_{i}\sqcup M_{p})$. It follows from the definitions that $V_{i}|_{a_{i}}$ is equal to the image of $V_{i}|_{[a_{i}, a_{p}]}$ under the homomorphism 
$$\xymatrix{H_{n+1}(W|_{[a_{i}, a_{p}]}, M_{i}\sqcup M_{p}) \ar[r]^-{\partial} & H_{n}(M_{i}\sqcup M_{p}) \cong H_{n}(M_{i})\oplus H_{n}(M_{p}) \ar[r]^-{\text{pr}_*} & H_{n}(M_{i}),}$$
and thus every element of the Lagrangian subspace $V_{i}|_{a_{i}}$ is equal to the image of some element in $H_{n+1}(W|_{[a_{i}, a_{p}]}, M_{i}\sqcup M_{p})$, under the above map. By combining Proposition \ref{proposition: representing homology by cylinders} with Theorem \ref{corollary: lagrangian is a surgery solution}, there exists a finite set $\Lambda_{i}$ and an embedding
\[\varphi_{i}: \Lambda_{i}\times[a_{i}, a_{p}]\times S^{n} \; \longrightarrow \; W|_{[a_{i}, a_{p}]}\]
with the following properties:
\begin{enumerate} 
\item[(a)] The homology classes represented by the embeddings 
$$\varphi_{i}|_{\lambda\times[a_{i}, a_{p}]\times S^{n}}: \lambda\times[a_{i}, a_{p}]\times S^{n} \; \longrightarrow \; W|_{[a_{i}, a_{p}]}, \quad \lambda \in \Lambda_{i},$$
are all contained in the subspace $V_{i}|_{[a_{i}, a_{p}]} \leq H_{n+1}(W|_{[a_{i}, a_{p}]}, M_{i}\sqcup M_{p})$.
\item[(b)] The collection of embeddings 
$$\varphi_{i}|_{\lambda\times\{a_{i}\}\times S^{n}}: \Lambda_{i}\times\{a_{i}\}\times S^{n} \; \longrightarrow \; M_{i}, \quad \lambda \in \Lambda_{i}, $$
yields a basis for the subspace $V_{i} \leq H_{n}(M_{i})$.
\item[(c)]
The restriction $\varphi_{i}|_{\Lambda_{i}\times\{a_{i}\}\times S^{n}}: \Lambda_{i}\times\{a_{i}\}\times S^{n} \longrightarrow M_{i}$ extends to an embedding
$$\varphi': \Lambda_{i}\times\{a_{i}\}\times\R^{n}\times S^{n} \; \longrightarrow \; M_{i},$$
with the property that the induced bundle map 
$$\xymatrix{T(\Lambda_{i}\times\{a_{i}\}\times\R^{n}\times S^{n})\oplus \epsilon^{1} \ar[rr] && TM_{i}\times\epsilon^{1} \ar[rr]^{\ell_{W}|_{M_{i}}} && \theta^{*}\gamma^{2n+1},}$$
admits an extension to a $\theta$-structure on $\Lambda_{i}\times\{a_{i}\}\times\R^{n}\times D^{n+1}$.
\end{enumerate}

The map $(\Lambda_{i}\times\{a_{i}\}\times\R^{n}\times S^{n})\cup(\Lambda_{i}\times[a_{i}, a_{p}]\times S^{n}) \; \longrightarrow \; W|_{[a_{i}, a_{p}]},$
obtained by combining $\varphi_{i}$ and $\varphi'_{i}$, extends to an embedding $\bar{\varphi}_{i}: \Lambda_{i}\times[a_{i}, a_{p}]\times\R^{n}\times S^{n} \; \longrightarrow \; W|_{[a_{i}, a_{p}]},$ which in turn extends to an embedding 
\begin{equation} \label{equation: full extended embedding i}
\widehat{\varphi}_{i}: \Lambda_{i}\times\R\times(a_{i}-\varepsilon_{i}, a_{p}+\varepsilon_{p})\times\R^{n}\times D^{n+1} \; \longrightarrow \; \R\times(-1, 1)^{\infty-1}.
\end{equation}
By carrying out the exact same construction for all $i = 0, \dots, p,$ we obtain embeddings $\widehat{\varphi}_{0}, \dots, \widehat{\varphi}_{p}$ as in (\ref{equation: full extended embedding i}). Applying Proposition \ref{corollary: multiple inductive disjunction}, we may arrange for
$$\widehat{\varphi}_{i}\left(\Lambda_{i}\times\R\times(a_{i}-\varepsilon_{i}, a_{p}+\varepsilon_{p})\times\{0\}\times S^{n}\right)\bigcap\widehat{\varphi}_{j}\left( \Lambda_{j}\times\R\times(a_{j}-\varepsilon_{j}, a_{p}+\varepsilon_{p})\times\{0\}\times S^{n}\right) \; = \; \emptyset$$
for all $i, j = 0, \dots, p$. Let $\Lambda = \sqcup_{i=0}^{p}\Lambda_{i}$. By forming the disjoint union of the embeddings $\widehat{\varphi}_{0}, \dots, \widehat{\varphi}_{p}$ and applying a reparametrization $(a_{i}-\varepsilon_{i}, a_{p}+\varepsilon_{p}) \cong (-6, -2)$, we obtain an embedding 
$$e: \Lambda\times\R\times(-6, -2)\times\R^{n}\times D^{n+1} \; \longrightarrow \; \R\times(-1, 1)^{\infty-1}$$

This embedding $e$ determines part of the data of an element of $\widetilde{\mb{Y}}_{0}(x)$. The other part of the necessary data is a $\theta$-structure $\ell$ on $\Lambda\times K|_{(-6,  2)}$ that restricts to the $\theta$-structure on $\Lambda\times K|_{(-6, -2)}$ given by the composition
\[\xymatrix{T(\Lambda\times K|_{(-6, -2)}) \ar[r]^-{De} & W \ar[r]^-{\ell_{W}} & \theta^{*}\gamma^{2n+1}.}\]
This extension can be obtained by same argument employed in the proof of \cite[Proposition 6.22, page 350]{GRW 14} and then by construction $(e,\ell) \in \widetilde{\mb{Y}}_{0}(x)$.
\end{proof}

Its hypotheses established, we can now apply Theorem \ref{theorem: improved flag complex theorem} to obtain Theorem \ref{theorem: contractibility of space of surgery data} as described at the beginning of this section.


\end{document}